\documentclass[a4paper,12pt,reqno]{amsart}

\newtheorem{proposition}{Proposition}[section]
\newtheorem{lemma}[proposition]{Lemma}
\newtheorem{corollary}[proposition]{Corollary}
\newtheorem{theorem}[proposition]{Theorem}

\theoremstyle{definition}
\newtheorem{definition}[proposition]{Definition}

\newtheorem{examples}[proposition]{Examples}

\newcommand{\thlabel}[1]{\label{th:#1}}

\newcommand{\selabel}[1]{\label{se:#1}}
\newcommand{\seref}[1]{Section~\ref{se:#1}}

\newcommand{\colabel}[1]{\label{co:#1}}

\newcommand{\delabel}[1]{\label{de:#1}}

\newcommand{\id}{id}

\newcommand{\Cc}{\mathcal{C}}

\def\*C{{}^*\hspace*{-1pt}{\Cc}}
\def\text#1{{\rm {\rm #1}}}

\newcommand{\trl}{\triangleleft}
\newcommand{\trr}{\triangleright}
\newcommand{\ppl}{\leftharpoonup}
\newcommand{\ppr}{\rightharpoonup}

\input xy
\xyoption {all} \CompileMatrices

\usepackage{amssymb}
\usepackage{color,amssymb,graphicx,amscd,amsmath}
\usepackage[colorlinks,urlcolor=blue,linkcolor=blue,citecolor=blue]{hyperref}
\allowdisplaybreaks

\title{Extending structures for Zinbiel 2-algebras}

\author[L. Zhang]{Ling Zhang}

\author[T. Zhang]{Tao Zhang}

\begin{document}
 \maketitle

 \setcounter{section}{0}

\begin{abstract}
 The extending structures problem for Zinbiel 2-algebras is studied. We introduce the concept of unified products for Zinbiel 2-algebras.
 Some special cases of unified products such as crossed product and matched pair of Zinbiel 2-algebras are investigated.
 It is proved that the extending problem for Zinbiel 2-algebras can be classified by some non-abelian cohomology theory.

\end{abstract}

\maketitle
\section*{Introduction}

The notion of $L_{\infty}$-algebras which generalizes Lie algebras first
appeared in deformation theory and mathematics physics. The algebraic theory of Lie 2-algebras
was studied by Baez and Crans in \cite{BC}. It is showed that a
Lie 2-algebra can be seen as a categorification of a Lie algebra,
where the underling vector space is replaced by 2-vector space and
the Jacobi identity is replaced by a nature transformation which
satisfies some coherence law.

The study of Zinbiel algebras was initiated by Loday \cite{Lod}.
It proved that the cohomology of Leibniz algebras and Rack cohomology is a Zinbiel algebra, see  \cite{Lod, Covez}.
It is also known as Tortkara algebras, pre-commutative and chronological algebras \cite{Kaw,K16}.
The classification of low dimensional Zinbiel algebras was investigated in \cite{AOK,ALO,DT,KPPV,K21,Ni}.
According to operad theory, Zinbiel algebras are Koszul dual to Leibniz algebras.
Some other algebraic theory of Zinbiel algebras were studied in \cite{MS,N1,N2,S20,Yau}.
It has been appeared some interesting applications of Zinbiel algebras in multiple zeta values and construction of a Cartesian differential category
\cite{C1,IP}. For more recent studies of Zinbiel algebras,  see \cite{C2,Covez,GGZ}.

The extending structures  for some algebraic objects such as associative algebras, Lie algebras, Hopf algebras, Leibniz algebras and Poisson algebras
 have been studied  by A. L. Agore and G. Militaru in \cite{AM01,AM02,AM1,AM2,AM3,AM4,AM5, AM6}.
 The extending structures for  Zinbiel algebras  and Lie 2-algebras have been studied in  \cite{ZZ21, TW}.
 The extending structures for left-symmetric algebras, associative and Lie conformal algebras has also been studied by Y. Hong and Y. Su in  \cite{Hong1,Hong2,Hong3}.
 The extending structures for Lie conformal superalgebras, 3-Lie algebras, Lie bialgebras and infinitesimal bialgebras were studied  in  \cite{ZCY,Zhang2201,Zhang2202,Zhang2203}.

Recently,  Zinbiel 2-algebra was introduced by the second author in \cite{Zhang2101}.
It is proved that the category of (strict) Zinbiel 2-algebras and the category of crossed modules of Zinbiel algebras are equivalent.

This paper is going to study extending structures problem for Zinbiel 2-algebras.
We define  the extending datum and unified products for Zinbiel 2-algebras.
 It is proved that the extending extending  problem for Zinbiel 2-algebras can be classified by equivalent classes of extending datum which is called the non-abelian cohomology theory.
  Some special cases of unified products such as crossed product and bicrossed product of Zinbiel 2-algebras are also investigated.

The organization of this paper are as follows.
In \seref{prelim}, we recall some notations and facts about Zinbiel algebras and Zinbiel 2-algebras.
In \seref{unifiedprod}, we introduce concept of the extending datum and unified product of Zinbiel 2-algebras.
We prove that extending extending  problem for Zinbiel 2-algebras can be classified by equivalent classes of extending datum.
In \seref{cazurispeciale} we show some special cases of unified products: crossed product and bicrossed products.
Using them we solve the nonabelian extension problem and factorization problem for Zinbiel 2-algebras respectively.

Throughout this paper, all vector spaces are assumed to be over an algebraically closed field $k$ of characteristic different from 2 and 3.
Let $V=V_0\oplus V_1$ and $W=W_0\oplus w_1$ be two $\mathbb{Z}_2$-graded vector spaces. Then the direct sum $V\oplus W$ is graded
by $(V\oplus W)_0 = V_0\oplus W_0$ and $(V\oplus W)_1 = V_1\oplus W_1$.
The identity map of a vector space $V$ is denoted by $id_V: V\to V$ or simply $id: V\to V$.

\section{Preliminaries}\selabel{prelim}
In this section, a strict Zinbiel 2-algebras is introduced. This kind of Zinbiel 2-algebras can be described in terms of crossed modules of Zinbiel algebras.

\begin{definition}
A Zinbiel algebra is a vector space ${Z}$ together with a multiplication $\cdot  : {Z} \times {Z} \to
{Z}$ satisfying the following Zinbiel identity:
\begin{equation}
(x\cdot y)\cdot z=x\cdot(y\cdot z+z\cdot y)
\end{equation}
for all $x,y,z \in {Z}$. We denote a Zinbiel algebra by $(Z,\cdot)$ or simply by $Z$.
\end{definition}

\begin{definition}
Let ${Z}$ be a Zinbiel algebra, $V$ a vector space. A bimodule of ${Z}$ over a vector space $V$ is a pair of linear maps \, $ \trr :{Z}\times V \to V,(x,v) \to x \trr v$ and $\trl :V\times {Z} \to V,(v,x) \to v \trl x$  such that the following conditions hold:
\begin{eqnarray}
  &&(x \cdot y) \trr v = x \trr (y \trr v+v\trl y),\\
  && (v \trl x) \trl y =v\trl(x \cdot y+y \cdot x) ,\\
  &&(x\trr v)\trl y  = x\trr (v \trl y+y \trr v),
\end{eqnarray}
for all $x,y\in {Z}$ and $v\in V.$
\end{definition}

The proof of the following Proposition \ref{prop:01} is by direct computations, so we omit the details.

\begin{proposition}\label{prop:01}
Let $Z$ be a Zinbiel algebra and $V$ be $Z$-bimodule. Then the direct sum vector space $Z \oplus V$ is a Zinbiel algebra with multiplication defined by:
\begin{equation}
(x, u)\cdot (y, v)= (x\cdot y, x \triangleright v + u\triangleleft y)
\end{equation}
for all $ x, y \in Z, u, v \in V$. This is called the semi-direct product of $Z$ and $V$.
\end{proposition}

\begin{definition}
A Zinbiel algebra $Z_{0}$ acts on another Zinbiel algebra $Z_{1}$ if there are two bilinear maps, $Z_{0} \times Z_{1} \to Z_{1},(x, y) \mapsto x \trr y$ and $Z_{1} \times Z_{0} \rightarrow Z_{1},(y, x) \mapsto y \trl x$ such that $Z_1$ is a $Z_0$-bimodule and the following conditions hold:
\begin{eqnarray}
(x_{0} \trr x_{1}) \cdot y_{1}=x_{0} \trr (x_{1} \cdot y_{1}+y_{1} \cdot x_{1}),\\
(x_{1} \trl x_{0}) \cdot y_{1}=x_{1} \cdot(x_{0} \trr y_{1}+y_{1} \trl x_{0}),\\
(x_{1} \cdot y_{1}) \trl x_{0}=x_{1} \cdot (y_{1} \trl x_{0}+ x_{0}\trr  y_{1}),
 \end{eqnarray}
for all $x_{0}\in Z_{0}, x_{1}, y_{1} \in Z_{1}.$
\end{definition}

\begin{definition}
 Let $(Z_{0}, \cdot)$ and $(Z_{1}, \cdot)$ be two Zinbiel algebras. A (strict) Zinbiel 2-algebra or a crossed module of Zinbiel algebras is a homomorphism of Zinbiel algebras $\varphi: Z_{1} \rightarrow Z_{0}$ together with an action of $Z_{0}$ on $Z_{1}$, denoted by $\triangleright: Z_{0} \times Z_{1} \rightarrow Z_{1},(x_{0}, x_{1}) \mapsto x_{0} \triangleright x_{1}$ and $\triangleleft: Z_{1} \times Z_{0} \rightarrow Z_{1}, (x_{1}, x_{0}) \mapsto x_{1} \triangleleft x_{0}$,  the following conditions hold:
\begin{eqnarray}
\varphi(x_{0} \triangleright x_{1}) &=& x_{0} \cdot \varphi(x_{1}), \\
\varphi(x_{1} \triangleleft x_{0}) &=&  \varphi(x_{1}) \cdot x_{0},\\
\varphi(x_{1}) \triangleright y_{1} &=&  x_{1} \cdot y_{1} = x_{1} \triangleleft \varphi(y_{1}),
\end{eqnarray}
for all $x_{i}, y_{i} \in Z_{i}, i=0,1.$
\end{definition}
We remark that from the above crossed module conditions one get $\varphi$ is a Zinbiel algebra homomorphism:
$$\varphi( x_{1} \cdot y_{1}) = \varphi(\varphi(x_{1}) \trr y_{1} )=\varphi(x_{1}) \cdot \varphi( y_{1} ).$$
 In the following context, a  Zinbiel 2-algebra is always  means a crossed module of Zinbiel algebras which is denoted by $(Z_1, Z_0,\varphi)$.

\begin{examples}A 2-vector space is a pair of vector spaces $(V_1,V_0)$ with a  linear map between them ${d}: V_1\to V_0.$
We can see a 2-vector space as a  crossed module of trivial Zinbiel algebras  $Z_{0}=V_0$, $Z_{1}=V_0$ and all the structure maps except $\varphi={d}$ are the trivial ones.
\end{examples}

\begin{examples}\label{exam01}
Let $I$ be a two sided ideal of a Zinbiel algebra $Z$,  then
$(I, Z, i)$ is a crossed module, where $i$ is the canonical inclusion map. In particular, if
$I= 0$ or $I = Z$, then $(0, Z, 0)$ and $(Z, Z, id)$ are crossed modules.
For a given Zinbiel algebra $Z$, then $(0, Z, 0)$ will be called the trivial example of crossed module or  Zinbiel 2-algebra.
\end{examples}

\begin{definition}
A morphism between two Zinbiel 2-algebra $(Z_1, Z_0,\varphi)$ and $(Z'_1, Z'_0,\varphi')$  is a pair of maps $\phi=(\phi_1,\phi_0)$ such that $\phi_1:Z_1\to Z'_1$ and $\phi_0:Z_0\to Z'_0$ are Zinbiel algebra homomorphisms  and $\varphi'\circ \phi_1=\phi_0\circ \varphi$.
An isomorphism is a morphism such that $\phi_1$ and $\phi_0$ are all bijection maps.

A Zinbiel 2-algebra $(Z'_1, Z'_0,\varphi')$ is  a subalgebra  of $(Z_1, Z_0,\varphi)$  if $Z'_1$ is subalgebra of $Z_1$, $Z'_0$ is subalgebra of $Z_0$ and
$\varphi|_{Z'_1}=\varphi'$.
\end{definition}

\begin{definition}
Let $(Z_1, Z_0,\varphi)$ be a given Zinbiel 2-algebra, $(V_1,V_0,d)$ a 2-vector space.
An extending structure of $(Z_1, Z_0,\varphi)$ through $(V_1,V_0,d)$ is a Zinbiel 2-algebra on $(E_1,E_0,\varphi_E)$
such that $(Z_1, Z_0,\varphi)$ is a subalgebra of $(E_1,E_0,\varphi_E)$,
which fit into the following exact sequence as vector spaces
\begin{eqnarray} \label{diag01}
\xymatrix {0 \ar[r]^{}  & {Z_1} \ar[r]^{\iota_1} \ar[d]_{\varphi} & {E_1}  \ar[r]^{\pi_1}\ar[d]^{\varphi_E}
& {V_1}\ar[d]_{}\ar[r]^{} \ar[d]^{{d}} & 0 \\
0 \ar[r]^{}& {Z_0} \ar[r]^{\iota_0} & {E_0}\ar[r]^{\pi_0} & V_0  \ar[r]^{} & 0}
\end{eqnarray}
where $\pi_i: E_i \to V_i$ is the canonical projection of $E_i$ on $V_i$ and $\iota_i: {Z_i} \to E_i$ is the
inclusion map.
This means that $E_1=Z_1\oplus V_1$ and $E_0=Z_0\oplus V_0$ as vector spaces,  $Z_1$ is subalgebra of  $E_1$ and $Z_0$ is subalgebra of  $E_0$.
The extending problem is to describe and classify up to an isomorphism  the set of all Zinbiel 2-algebra structures that can be defined on $(E_1,E_0,\varphi_E)$ such that $(Z_1, Z_0,\varphi)$ is a subalgebra of $(E_1,E_0,\varphi_E)$.
\end{definition}

\begin{definition} \delabel{echivextedn}
Let $(Z_1, Z_0,\varphi)$  be a Zinbiel 2-algebra, $(E_1,E_0,\varphi_E)$ a 2-vector space such that
$(Z_1, Z_0,\varphi)$ is a subalgebra of $(E_1,E_0,\varphi_E)$. For linear maps $\phi=(\phi_1,\phi_0)$ where $\phi_1: E_1 \to E_1$ and $\phi_0: E_0 \to E_0$ we consider the following diagram:
\begin{eqnarray}\label{diag02}
\xymatrix{
Z_1\ar[rd]\ar[rr]^{\iota_1}\ar[dd]&&E_1\ar[rr]^{\pi_1}\ar[rd]^{\phi_1}\ar[dd]|\hole&&V_1\ar[rd]\ar[dd]|\hole\\
 &Z_1\ar[rr]^{\iota_1}\ar[dd]&&E_1\ar[rr]^{\pi_1}\ar[dd]&&V_1\ar[dd]\\
Z_0\ar[rd]\ar[rr]|\hole^{\iota_0}&&E_0\ar[rd]^{\phi_0}\ar[rr]|\hole^{\pi_0}&&V_0\ar[rd]\\
&Z_0\ar[rr]^{\iota_0}&&E_0\ar[rr]^{\pi_0}&&V_0\\ }
\end{eqnarray}
We say that $\phi=(\phi_1,\phi_0)$ \emph{stabilizes}$(Z_1, Z_0,\varphi)$  if the left cube of the diagram  is
commutative, and $\phi$ \emph{stabilizes} $(V_1, V_0, d)$ if the right cube of the diagram is
commutative.

Let $(E_1,E_0,\varphi_E)$ and $(E_1,E_0,\varphi'_E)$ be two Zinbiel 2-algebra structures
on $(E_1,E_0)$ both containing $(Z_1, Z_0,\varphi)$  as a subalgebra. They are called \emph{equivalent}, and we
denote this by $(E_1,E_0,\varphi_E)\equiv (E_1,E_0,\varphi'_E)$, if
there exists a Zinbiel 2-algebra isomorphism $\phi=(\phi_1,\phi_0)$ which stabilizes $(Z_1, Z_0,\varphi)$ . Denote by $Extd(E,{Z})$ the set of equivalent classes of $(Z_1, Z_0,\varphi)$  through $(V_1, V_0, d)$.
$(E_1,E_0,\varphi_E)$ and $(E_1,E_0,\varphi'_E)$  are called \emph{cohomologous},
and we denote it by $(E_1,E_0,\varphi_E)$\,$\approx$\,$ (E_1,E_0,\varphi'_E)$, if there exists a Zinbiel 2-algebra isomorphism
$\phi=(\phi_1,\phi_0)$ which stabilizes ${Z}$ and co-stabilizes $V$.
Denote by $Extd'(E,{Z})$  the set of  cohomologous classes of $(Z_1, Z_0,\varphi)$  through $(V_1, V_0, d)$.

\end{definition}

\section{Unified Products for Zinbiel 2-algebras}\selabel{unifiedprod}

In this section, we introduced a unified product for Zinbiel 2-algebras.

\begin{definition}\label{def:01}
 Suppose that $(Z_{1}, Z_{0}, \varphi)$ is a Zinbiel 2-algebra and $(V_{1}, V_{0}, d)$ is a 2-vector space. An extending datum of $(Z_1, Z_0,\varphi)$ by $(V_{1}, V_{0}, d)$ is a system $\Omega(Z, V)= (\leftharpoonup_{j}, \rightharpoonup_{j}, \triangleleft_{j}, \triangleright_{j}, \omega_{j}, *_{j}; j=0,1,2,3$) consisting of one linear map $\sigma: V_{1} \rightarrow Z_{0}$ and twenty-four bilinear maps
 \begin{eqnarray*}
\rightharpoonup_{0}: V_{0} \times Z_{0} \rightarrow Z_{0}, \quad\leftharpoonup_{0}: Z_{0} \times V_{0} \rightarrow Z_{0}, \quad \triangleright_{0}: Z_{0} \times V_{0} \rightarrow V_{0},
\quad \triangleleft_{0}: V_{0} \times Z_{0} \rightarrow V_{0},\\
\rightharpoonup_{1}: V_{1} \times Z_{1} \rightarrow Z_{1},\quad \leftharpoonup_{1}: Z_{1} \times V_{1} \rightarrow Z_{1}, \quad\triangleright_{1}: Z_{1} \times V_{1} \rightarrow V_{1},
\quad \triangleleft_{1}: V_{1} \times Z_{1} \rightarrow V_{1}, \\
\rightharpoonup_{2}: V_{0} \times Z_{1} \rightarrow Z_{1},\quad \leftharpoonup_{2}: Z_{0} \times V_{1} \rightarrow Z_{1},\quad \triangleright_{2}: Z_{0} \times V_{1} \rightarrow V_{1},
\quad \triangleleft_{2}: V_{0} \times Z_{1} \rightarrow V_{1},\\
\rightharpoonup_{3}: V_{1} \times Z_{0} \rightarrow Z_{1}, \quad\leftharpoonup_{3}: Z_{1} \times V_{0} \rightarrow Z_{1},\quad\triangleright_{3}: Z_{1} \times V_{0} \rightarrow V_{1},
\quad \triangleleft_{3}: V_{1} \times Z_{0} \rightarrow V_{1},\\
\omega_{0}: V_{0} \times V_{0} \rightarrow Z_{0},\quad *_{0}: V_{0} \times V_{0} \rightarrow V_{0},\\
\quad \omega_{1}: V_{1} \times V_{1} \rightarrow Z_{1},
\quad *_{1}: V_{1} \times V_{1} \rightarrow V_{1}\\
\quad \omega_{2}: V_{0} \times V_{1} \rightarrow  Z_{1}, \quad *_{2}: V_{0} \times V_{1} \rightarrow V_{1},\\
\quad \omega_{3}: V_{1} \times V_{0} \rightarrow Z_{1}, \quad *_{3}: V_{1} \times V_{0} \rightarrow V_{1}.
\end{eqnarray*}

Let $\Omega(Z, V)=\left(\leftharpoonup_{j}, \rightharpoonup_{j}, \triangleleft_{j}, \triangleright_{j}, \omega_{j}, \ast_{j}, \sigma; j=0,1,2,3\right)$ be an extending datum. We construct a new Zinbiel 2-algebra $(E_1,E_0,\varphi_E)$ as follows.

As a 2-vector space, $(E_1,E_0,\varphi_E)$ where $E_1=Z_{1}\oplus V_{1} $, $E_0=Z_{0}\oplus V_{0} $ and $\varphi_E:
Z_{1}\oplus V_{1} \rightarrow Z_{0} \oplus V_{0}$ is defined as follows:
 \begin{eqnarray}
 \varphi_E(x, u)=(\varphi(x)+\sigma(u), d(u)),
 \end{eqnarray}
for all $x \in Z_{1}$ and $u \in V_{1}$. The multiplication maps $\circ_{i}:\left(Z_{i} \oplus V_{i}\right) \times\left(Z_{i} \oplus V_{i}\right) \rightarrow Z_{i} \oplus V_{i}$ and the action maps $\trr: (Z_{0} \oplus V_{0}) \times (Z_{1} \oplus V_{1})\to Z_{1} \oplus V_{1},$ $\trl: (Z_{1} \oplus V_{1}) \times (Z_{0} \oplus V_{0})\to Z_{1} \oplus V_{1} $ is given by
$$
\left(x_{i}, u_{i}\right) \circ_{i} \left(y_{i}, v_{i}\right)
 =\big( x_{i}\cdot y_{i} + x_{i} \ppl_{i} v_{i} + u_{i}\ppr_{i} y_{i} + \omega_{i}(u_{i}, v_{i}),\  x_{i} \trr_{i} v_{i} + u_{i}\trl_{i} y_{i} + u_{i} \ast_{i} v_{i} \big),
$$
and
$$
\left(x_{0}, u_{0}\right) \trr \left( x_{1}, u_{1}\right)
 = \big( x_{0}\cdot x_{1} + x_{0} \ppl_{2} u_{1} + u_{0}\ppr_{2} x_{1} + \omega_{2}(u_{0}, u_{1}),\  x_{0} \trr_{2} u_{1} + u_{0}\trl_{2} x_{1} + u_{0} \ast_{2} u_{1} \big),
$$
$$
\left(x_{1}, u_{1}\right) \trl \left( x_{0}, u_{0}\right)
 = \big( x_{1}\cdot x_{0} + x_{1} \ppl_{3} u_{0} + u_{1}\ppr_{3} x_{0} + \omega_{3}(u_{1}, u_{0}),\  x_{1} \trr_{3} u_{0} + u_{1}\trl_{3} x_{0} + u_{1} \ast_{3} u_{0} \big).
$$
for $i = 0,1$ and all $x_{i}, y_{i} \in Z_{i}, u_i, v_i\in V_i$.
This Zinbiel 2-algebra is called the \emph{unified product}  of $(Z_1, Z_0,\varphi)$ and
$(V_1, V_0, d)$.  We will denoted it by  ${Z} \, \natural {V}$.
\end{definition}

\begin{theorem}\label{thm:unifyprod}
Suppose that $(Z_{1}, Z_{0}, \varphi)$ is  a Zinbiel 2-algebra and $(V_{1}, V_{0}, d)$ is a 2-vector space. Then $\Omega(Z, V)$ is an extending datum of $(Z_{1}, Z_{0}, \varphi)$ by $(V_{1}, V_{0}, d)$  such that $(E_1,E_0,\varphi_E)$ is a Zinbiel 2-algebra if and only if the following conditions hold for all $x_{i}, y_{i} \in Z_{i}$, $u_{i}, v_{i}\in V_{i}$, $i=0,1:$
\begin{enumerate}
\item[(Z1)] $(V_{i}, \trl_{i}, \trr_{i})$ is an ${Z_{i}}$-bimodule:
 \begin{eqnarray*}
            &&( x_{i}\cdot y_{i} )\trr_{i} w_{i}=x_{i}\trr_{i}( y_{i}\trr_{i} w_{i}+w_{i}\trl_{i} y_{i}),\\
           &&(x_{i}\trr_{i} v_{i})\trl_{i} z_{i}=x_{i}\trr_{i}(v_{i}\trl_{i} z_{i}+ z_{i}\trr_{i} v_{i}),\\
            &&(u_{i}\trl_{i} y_{i})\trl_{i} z_{i}=u_{i}\trl_{i} (y_{i}\cdot z_{i}+ z_{i}\cdot y_{i}),
            \end{eqnarray*}
\item[(Z2)]
$(x_{i}\ppl_{i} v_{i})\cdot y_{i}+(x_{i} \trr_{i} v_{i})\ppr_{i} y_{i}
=x_{i}\cdot (v_{i} \ppr_{i} y_{i}+y_{i} \ppl_{i} v_{i})+x_{i}\ppl_{i}( v_{i}\trl_{i} y_{i} + y_{i} \trr_{i} v_{i}),$
\item[(Z3)] $(u_{i} \ppr_{i} x_{i})\cdot y_{i} +( u_{i}\trl_{i} x_{i} )\ppr_{i} y_{i} = u_{i}\ppr_{i}( x_{i} \cdot  y_{i} + y_{i}\cdot  x_{i} ),$
\item[(Z4)]
$\omega_{i}(u_{i}, v_{i})\cdot x_{i}+ (u_{i} \ast_{i} v_{i} )\ppr_{i} x_{i}
= u_{i}\ppr_{i}( v_{i}\ppr_{i} x_{i} + x_{i} \ppl_{i} v_{i})+ \omega_{i}(u_{i},v_{i}\trl_{i} x_{i}+ x_{i} \trr_{i} v_{i}),$
\item[(Z5)] $(u_{i} \ast_{i} v_{i})\trl_{i} x_{i} = u_{i}\trl_{i}( v_{i}\ppr_{i} x_{i}+x_{i}\ppl_{i} v_{i})+u_{i} \ast_{i}( v_{i}\trl_{i} x_{i} + x_{i} \trr_{i} v_{i} ),$
\item[(Z6)]$( x_{i}\cdot y_{i} )\ppl_{i} w_{i}  = x_{i}\cdot(y_{i} \ppl_{i} w_{i}+w_{i}\ppr_{i} y_{i})+x_{i}\ppl_{i}(y_{i} \trr_{i} w_{i}+w_{i}\trl_{i} y_{i}),$
\item[(Z7)]
$( x_{i} \ppl_{i} v_{i})\ppl_{i} w_{i}+\omega_{i}(x_{i} \trr_{i} v_{i},w_{i})
=x_{i}\cdot\big(\omega_{i}(v_{i}, w_{i})+\omega_{i}(w_{i},v_{i})\big)+x_{i}\ppl_{i}(v_{i} \ast_{i} w_{i} +w_{i} \ast_{i} v_{i} ),$
\item[(Z8)] $(x_{i} \ppl_{i} v_{i}) \trr_{i} w_{i}+(x_{i}\trr_{i} v_{i})\ast_{i} w_{i} = x_{i}\trr_{i}( v_{i} \ast_{i} w_{i}+w_{i} \ast_{i} v_{i} ),$
\item[(Z9)]
$( u_{i}\ppr_{i} x_{i})\ppl_{i} w_{i}+\omega_{i}( u_{i}\trl_{i} x_{i},w_{i})
= u_{i}\ppr_{i}( x_{i} \ppl_{i} w_{i}+w_{i}\ppr_{i} x_{i} )+\omega_{i}(u_{i}, w_{i}\trl_{i} x_{i} + x_{i}\trr_{i} w_{i}),$
\item[(Z10)]
$(u_{i}\ppr_{i} x_{i}) \trr_{i} w_{i}+ (u_{i}\trl_{i} x_{i})\ast_{i} w_{i}
= u_{i}\trl_{i}(x_{i} \ppl_{i} w_{i} + w_{i}\ppr_{i} x_{i} )+u_{i} \ast_{i}(x_{i} \trr_{i} w_{i} + w_{i}\trl_{i} x_{i}),$
\item[(Z11)]
$\omega_{i}(u_{i}, v_{i})\ppl_{i} w_{i}+\omega_{i}( u_{i} \ast_{i} v_{i} ,w_{i} )
=u_{i}\ppr_{i}\big(\omega_{i}(v_{i}, w_{i})+\omega_{i}(w_{i},v_{i})\big)+ \omega_{i}(u_{i},v_{i} \ast_{i} w_{i}+ w_{i} \ast_{i} v_{i} ),$
\item[(Z12)]
$\omega_{i}(u_{i}, v_{i}) \trr_{i} w_{i}+(u_{i} \ast_{i} v_{i})\ast_{i} w_{i}
=u_{i}\trl_{i} \big(\omega_{i}(v_{i}, w_{i})+\omega_{i}(w_{i},v_{i})\big)+u_{i} \ast_{i}(v_{i} \ast_{i} w_{i}+w_{i} \ast_{i} v_{i}),$
\item[(Z13)]
$(x_{0}\cdot x_{1})\ppl_{1}u_{1}=x_{0}\cdot(x_{1}\ppl_{1}u_{1}+u_{1}\ppr_{1}x_{1})+ x_{0}\ppl_{2}(x_{1} \trr_{1}u_{1}+u_{1}\trl_{1}x_{1}),$
\item[(Z14)]
$(x_{0} \ppl_{2}u_{1})\cdot x_{1}+( x_{0}\trr_{2}u_{1})\ppr_{1}x_{1}=x_{0}\cdot(u_{1}\ppr_{1}x_{1}+x_{1} \ppl_{1}u_{1})+x_{0}\ppl_{2}(u_{1}\trl_{1} x_{1}+x_{1}\trr_{1} u_{1}),$
\item[(Z15)]
$(x_{0} \ppl_{2} u_{1}) \ppl_{1} v_{1} + \omega_{1}( x_{0} \trr_{2} u_{1}, v_{1})=x_{0}\cdot (\omega_{1}(u_{1}, v_{1})+\omega_{1}(v_{1}, u_{1}))+x_{0} \ppl_{2} (u_{1} \ast_{1} v_{1}+v_{1} \ast_{1} u_{1}),$
\item[(Z16)] $(u_{0}\ppr_{2} x_{1})\cdot y_{1} + (u_{0}\trl_{2} x_{1})\ppr_{1} y_{1}=u_{0}\ppr_{2} ( x_{1}\cdot y_{1}+y_{1}\cdot x_{1}),$
\item[(Z17)]
$(u_{0}\ppr_{2} x_{1}) \ppl_{1}u_{1} + \omega_{1}(u_{0}\trl_{2} x_{1},u_{1})=u_{0}\ppr_{2} (x_{1} \ppl_{1}u_{1}+u_{1}\ppr_{1} x_{1})+ \omega_{2}(u_{0}, x_{1} \trr_{1}u_{1}+u_{1}\trl_{1} x_{1}),$
\item[(Z18)]
$\omega_{2}(u_{0}, u_{1})\cdot x_{1} + (u_{0} \ast_{2} u_{1})\ppr_{1} x_{1}=u_{0}\ppr_{2} (u_{1}\ppr_{1} x_{1}+x_{1} \ppl_{1} u_{1}) + \omega_{2}(u_{0},u_{1}\trl_{1} x_{1}+x_{1} \trr_{1} u_{1}),$
\item[(Z19)]
$\omega_{2}(u_{0}, u_{1}) \ppl_{1} v_{1} + \omega_{1}(u_{0} \ast_{2} u_{1}, v_{1})=u_{0}\ppr_{2} (\omega_{1}(u_{1}, v_{1})+\omega_{1}(v_{1}, u_{1}))+ \omega_{2}(u_{0},u_{1} \ast_{1} v_{1}+ v_{1} \ast_{1} u_{1}),$
\item[(Z20)] $(x_{0}\cdot x_{1}) \trr_{1} u_{1}=x_{0} \trr_{2} (x_{1} \trr_{1} u_{1}+u_{1}\trl_{1} x_{1}),$
\item[(Z21)] $( x_{0} \trr_{2} u_{1})\trl_{1} y_{1} =x_{0} \trr_{2} (u_{1}\trl_{1} y_{1}+y_{1} \trr_{1} u_{1}),$
\item[(Z22)] $(x_{0} \ppl_{2} u_{1}) \trr_{1} v_{1}+( x_{0} \trr_{2} u_{1}) \ast_{1} v_{1}=x_{0} \trr_{2} (u_{1} \ast_{1} v_{1}+v_{1} \ast_{1} u_{1}),$
\item[(Z23)] $(u_{0}\trl_{2} x_{1})\trl_{1} y_{1}=u_{0}\trl_{2} (x_{1}\cdot y_{1}+y_{1}\cdot x_{1}),$
\item[(Z24)]
$(u_{0}\ppr_{2} x_{1}) \trr_{1} u_{1} + (u_{0}\trl_{2} x_{1}) \ast_{1} u_{1}=u_{0}\trl_{2} (x_{1} \ppl_{1} u_{1}+u_{1}\ppr_{1} x_{1}) + u_{0} \ast_{2} (x_{1} \trr_{1} u_{1}+u_{1}\trl_{1} x_{1}),$
\item[(Z25)]
$(u_{0} \ast_{2} u_{1})\trl_{1} x_{1}=u_{0}\trl_{2} (u_{1}\ppr_{1} x_{1}+x_{1} \ppl_{1} u_{1}) + u_{0} \ast_{2} (u_{1}\trl_{1} x_{1}+x_{1} \trr_{1} u_{1}),$
\item[(Z26)]
$\omega_{2}(u_{0}, u_{1}) \trr_{1} v_{1} + (u_{0} \ast_{2} u_{1}) \ast_{1} v_{1}=u_{0}\trl_{2} (\omega_{1}(u_{1}, v_{1})+\omega_{1}(v_{1}, u_{1})) + u_{0} \ast_{2} (u_{1} \ast_{1} v_{1}+v_{1} \ast_{1} u_{1}) \big),$
\item[(Z27)]
$(x_{0} \ppl_{0} u_{0})\cdot x_{1}+ (x_{0} \trr_{0} u_{0})\ppr_{2} x_{1}=x_{0}\cdot (u_{0}\ppr_{2} x_{1}+x_{1} \ppl_{3} u_{0})+ x_{0} \ppl_{2} (u_{0}\trl_{2} x_{1}+x_{1} \trr_{3} u_{0}),$
\item[(Z28)]
$(u_{0}\ppr_{0} x_{0})\cdot x_{1}+ (u_{0}\trl_{0} x_{0})\ppr_{2} x_{1}=u_{0}\ppr_{2} (x_{0}\cdot x_{1}+x_{1}\cdot x_{0}),$
\item[(Z29)]
$\omega_{0}(u_{0}, v_{0})\cdot x_{1}+ (u_{0} \ast_{0} v_{0})\ppr_{2} x_{1}= u_{0}\ppr_{2} (v_{0}\ppr_{2} x_{1}+x_{1} \ppl_{3} v_{0})+ \omega_{2}(u_{0},v_{0}\trl_{2} x_{1}+x_{1} \trr_{3} v_{0}),$
\item[(Z30)]
$(x_{0}\cdot y_{0}) \ppl_{2} u_{1}=x_{0}\cdot (y_{0} \ppl_{2} u_{1}+u_{1}\ppr_{3} y_{0})+x_{0} \ppl_{2} (y_{0} \trr_{2} u_{1}+u_{1}\trl_{3} y_{0}),$
\item[(Z31)]
$(x_{0} \ppl_{0} u_{0}) \ppl_{2} u_{1}+ \omega_{2}(x_{0} \trr_{0} u_{0}, u_{1})=x_{0}\cdot (\omega_{2}(u_{0}, u_{1})+ \omega_{3}(u_{1}, u_{0})) + x_{0} \ppl_{2} (u_{0} \ast_{2} u_{1}+u_{1} \ast_{3} u_{0}),$
\item[(Z32)]
$(u_{0}\ppr_{0} x_{0}) \ppl_{2} u_{1}+ \omega_{2}(u_{0}\trl_{0} x_{0}, u_{1})= u_{0}\ppr_{2} (x_{0} \ppl_{2} u_{1}+u_{1}\ppr_{3} x_{0}) + \omega_{2}(u_{0}, x_{0} \trr_{2} u_{1}+u_{1}\trl_{3} x_{0}),$
\item[(Z33)]
$\omega_{0}(u_{0}, v_{0}) \ppl_{2} u_{1}+ \omega_{2}(u_{0} \ast_{0} v_{0}, u_{1})= u_{0}\ppr_{2} (\omega_{2}(v_{0}, u_{1}) + \omega_{3}(u_{1}, v_{0}))+ \omega_{2}(u_{0},v_{0} \ast_{2} u_{1}+u_{1} \ast_{3} v_{0}),$
\item[(Z34)]
$(x_{0}\cdot y_{0}) \trr_{2} u_{1}=x_{0} \trr_{2} (y_{0} \trr_{2} u_{1}+u_{1}\trl_{3} y_{0}),$
\item[(Z35)]
$(x_{0} \ppl_{0} u_{0}) \trr_{2} u_{1} + (x_{0} \trr_{0} u_{0}) \ast_{2} u_{1}=x_{0} \trr_{2} (u_{0} \ast_{2} u_{1}+u_{1} \ast_{3} u_{0}),$
\item[(Z36)]
$(u_{0}\ppr_{0} x_{0}) \trr_{2} u_{1}+ (u_{0}\trl_{0} x_{0}) \ast_{2} u_{1}=u_{0}\trl_{2} (x_{0} \ppl_{2} u_{1}+u_{1}\ppr_{3} x_{0})+ u_{0} \ast_{2} (x_{0} \trr_{2} u_{1}+u_{1}\trl_{3} x_{0}),$
\item[(Z37)]
$\omega_{0}(u_{0}, v_{0}) \trr_{2} u_{1}+ (u_{0} \ast_{0} v_{0}) \ast_{2} u_{1}=u_{0}\trl_{2} (\omega_{2}(v_{0}, u_{1})+ \omega_{3}(u_{1}, v_{0})) + u_{0} \ast_{2} (v_{0} \ast_{2} u_{1}+u_{1} \ast_{3} v_{0}),$
\item[(Z38)]
$(x_{0} \trr_{0} u_{0})\trl_{2} x_{1}=x_{0} \trr_{2} (u_{0}\trl_{2} x_{1}+ x_{1} \trr_{3} u_{0}),$
\item[(Z39)]
$(u_{0}\trl_{0} x_{0})\trl_{2} x_{1}=u_{0}\trl_{2} (x_{0}\cdot x_{1}+x_{1}\cdot x_{0}),$
\item[(Z40)]
$(u_{0} \ast_{0} v_{0})\trl_{2} x_{1}=u_{0}\trl_{2} (v_{0}\ppr_{2} x_{1}+x_{1} \ppl_{3} v_{0})+ u_{0} \ast_{2} (v_{0}\trl_{2} x_{1}+x_{1} \trr_{3} v_{0}),$
\item[(Z41)]
$(x_{1} \ppl_{3} u_{0})\cdot x_{0}+ (x_{1} \trr_{3} u_{0})\ppr_{3} x_{0}=x_{1}\cdot (u_{0}\ppr_{0} x_{0}+x_{0} \ppl_{0} u_{0}) + x_{1} \ppl_{3} (u_{0}\trl_{0} x_{0}+x_{0} \trr_{0} u_{0}),$
\item[(Z42)]
$(u_{1}\ppr_{3} x_{0})\cdot y_{0} + (u_{1}\trl_{3} x_{0})\ppr_{3} y_{0}=u_{1}\ppr_{3} (x_{0}\cdot y_{0}+y_{0}\cdot x_{0}),$
\item[(Z43)]
$\omega_{3}(u_{1}, u_{0})\cdot x_{0} + (u_{1} \ast_{3} u_{0})\ppr_{3} x_{0}=u_{1}\ppr_{3} (u_{0}\ppr_{0} x_{0}+x_{0} \ppl_{0} u_{0})+ \omega_{3}(u_{1},u_{0}\trl_{0} x_{0}+x_{0} \trr_{0} u_{0}),$
\item[(Z44)]
$(x_{1}\cdot x_{0}) \ppl_{3} u_{0}=x_{1}\cdot (x_{0} \ppl_{0} u_{0}+u_{0}\ppr_{0} x_{0}) + x_{1} \ppl_{3} (x_{0} \trr_{0} u_{0}+u_{0}\trl_{0} x_{0}),$
\item[(Z45)]
$(x_{1} \ppl_{3} u_{0}) \ppl_{3} v_{0}+ \omega_{3}(x_{1} \trr_{3} u_{0}, v_{0})=x_{1}\cdot (\omega_{0}(u_{0}, v_{0})+\omega_{0}(v_{0}, u_{0}))+ x_{1} \ppl_{3} (u_{0} \ast_{0} v_{0}+v_{0} \ast_{0} u_{0}),$
\item[(Z46)]
$(u_{1}\ppr_{3} x_{0}) \ppl_{3} u_{0} + \omega_{3}(u_{1}\trl_{3} x_{0}, u_{0})=u_{1}\ppr_{3} (x_{0} \ppl_{0} u_{0}+u_{0}\ppr_{0} x_{0}) + \omega_{3}(u_{1}, x_{0} \trr_{0} u_{0}+ u_{0}\trl_{0} x_{0}),$
\item[(Z47)]
$\omega_{3}(u_{1}, u_{0}) \ppl_{3} v_{0} + \omega_{3}(u_{1} \ast_{3} u_{0}, v_{0})=u_{1}\ppr_{3} (\omega_{0}(u_{0}, v_{0})+ \omega_{0}(v_{0}, u_{0})) + \omega_{3}(u_{1},u_{0} \ast_{0} v_{0}+v_{0} \ast_{0} u_{0}),$
\item[(Z48)]
$(x_{1}\cdot x_{0}) \trr_{3} u_{0}=x_{1} \trr_{3} (x_{0} \trr_{0} u_{0}+u_{0}\trl_{0} x_{0}),$
\item[(Z49)]
$(x_{1} \ppl_{3} u_{0}) \trr_{3} v_{0}+ (x_{1} \trr_{3} u_{0}) \ast_{3} v_{0}=x_{1} \trr_{3} (u_{0} \ast_{0} v_{0}+v_{0} \ast_{0} u_{0}),$
\item[(Z50)]
$(u_{1}\ppr_{3} x_{0}) \trr_{3} u_{0}+ (u_{1}\trl_{3} x_{0}) \ast_{3} u_{0}=u_{1}\trl_{3} (x_{0} \ppl_{0} u_{0}+u_{0}\ppr_{0} x_{0})+ u_{1} \ast_{3} (x_{0} \trr_{0} u_{0}+u_{0}\trl_{0} x_{0}),$
\item[(Z51)]
$\omega_{3}(u_{1}, u_{0}) \trr_{3} v_{0} + (u_{1} \ast_{3} u_{0}) \ast_{3} v_{0}=u_{1}\trl_{3} (\omega_{0}(u_{0}, v_{0}) + \omega_{0}(v_{0}, u_{0}))+ u_{1} \ast_{3} (u_{0} \ast_{0} v_{0}+v_{0} \ast_{0} u_{0}),$
\item[(Z52)]
$(x_{1} \trr_{3} u_{0})\trl_{3} x_{0}=x_{1} \trr_{3} (u_{0}\trl_{0} x_{0}+x_{0} \trr_{0} u_{0}),$
\item[(Z53)]
$(u_{1}\trl_{3} x_{0})\trl_{3} y_{0}=u_{1}\trl_{3} (x_{0}\cdot y_{0}+y_{0}\cdot x_{0}),$
\item[(Z54)]
$(u_{1} \ast_{3} u_{0})\trl_{3} x_{0}=u_{1}\trl_{3} (u_{0}\ppr_{0} x_{0}+x_{0} \ppl_{0} u_{0})+ u_{1} \ast_{3} (u_{0}\trl_{0} x_{0}+x_{0} \trr_{0} u_{0}),$
\item[(Z55)]
$(x_{0} \ppl_{2} u_{1})\cdot y_{0}+ (x_{0} \trr_{2} u_{1})\ppr_{3} y_{0}=x_{0}\cdot (u_{1}\ppr_{3} y_{0}+y_{0} \ppl_{2} u_{1})+ x_{0} \ppl_{2} (u_{1}\trl_{3} y_{0}+y_{0} \trr_{2} u_{1}),$
\item[(Z56)]
$(u_{0}\ppr_{2} x_{1})\cdot y_{0}+ (u_{0}\trl_{2} x_{1})\ppr_{3} y_{0}= u_{0}\ppr_{2} (x_{1}\cdot y_{0}+y_{0}\cdot x_{1}),$
\item[(Z57)]
$\omega_{2}(u_{0}, u_{1})\cdot x_{0}+ (u_{0} \ast_{2} u_{1})\ppr_{3} x_{0}= u_{0}\ppr_{2} (u_{1}\ppr_{3} x_{0}+x_{0} \ppl_{2} u_{1})+ \omega_{2}(u_{0}, u_{1}\trl_{3} x_{0}+ x_{0} \trr_{2} u_{1}),$
\item[(Z58)]
$(x_{0}\cdot x_{1}) \ppl_{3} u_{0}=x_{0}\cdot (x_{1} \ppl_{3} u_{0}+u_{0}\ppr_{2} x_{1})+ x_{0} \ppl_{2} (x_{1} \trr_{3} u_{0}+u_{0}\trl_{2} x_{1}),$
\item[(Z59)]
$(x_{0} \ppl_{2} u_{1}) \ppl_{3} v_{0}+ \omega_{3}(x_{0} \trr_{2} u_{1}, v_{0})=x_{0}\cdot (\omega_{3}(u_{1}, v_{0})+ \omega_{2}(v_{0}, u_{1}))+ x_{0} \ppl_{2} (u_{1} \ast_{3} v_{0}+v_{0} \ast_{2} u_{1}),$
\item[(Z60)]
$(u_{0}\ppr_{2} x_{1}) \ppl_{3} v_{0}+ \omega_{3}(u_{0}\trl_{2} x_{1}, v_{0})= u_{0}\ppr_{2} (x_{1} \ppl_{3} v_{0}+v_{0}\ppr_{2} x_{1})+ \omega_{2}(u_{0}, x_{1} \trr_{3} v_{0}+v_{0}\trl_{2} x_{1}) ,$
\item[(Z61)]
$\omega_{2}(u_{0}, u_{1}) \ppl_{3} v_{0}+ \omega_{3}(u_{0} \ast_{2} u_{1}, v_{0})=u_{0}\ppr_{2} (\omega_{3}(u_{1}, v_{0})+\omega_{2}(v_{0}, u_{1}))+ \omega_{2}(u_{0}, u_{1} \ast_{3} v_{0}+ v_{0} \ast_{2} u_{1}),$
\item[(Z62)]
$(x_{0}\cdot x_{1}) \trr_{3} u_{0}=x_{0} \trr_{2} (x_{1} \trr_{3} u_{0}+u_{0}\trl_{2} x_{1}),$
\item[(Z63)]
$(x_{0} \ppl_{2} u_{1}) \trr_{3} v_{0}+ (x_{0} \trr_{2} u_{1}) \ast_{3} v_{0}=x_{0} \trr_{2} (u_{1} \ast_{3} v_{0}+v_{0} \ast_{2} u_{1}),$
\item[(Z64)]
$(u_{0}\ppr_{2} x_{1}) \trr_{3} v_{0}+ (u_{0}\trl_{2} x_{1}) \ast_{3} v_{0}=u_{0}\trl_{2} (x_{1} \ppl_{3} v_{0}+v_{0}\ppr_{2} x_{1})+ u_{0} \ast_{2} (x_{1} \trr_{3} v_{0}+v_{0}\trl_{2} x_{1}),$
\item[(Z65)]
$ \omega_{2}(u_{0}, u_{1}) \trr_{3} v_{0}+ (u_{0} \ast_{2} u_{1}) \ast_{3} v_{0}=u_{0}\trl_{2} (\omega_{3}(u_{1}, v_{0})+ \omega_{2}(v_{0}, u_{1}))+ u_{0} \ast_{2} (u_{1} \ast_{3} v_{0}+v_{0} \ast_{2} u_{1}),$
\item[(Z66)]
$(x_{0} \trr_{2} u_{1})\trl_{3} y_{0}=x_{0} \trr_{2} (u_{1}\trl_{3} y_{0}+y_{0} \trr_{2} u_{1}),$
\item[(Z67)]
$(u_{0}\trl_{2} x_{1})\trl_{3} y_{0}=u_{0}\trl_{2} (x_{1}\cdot y_{0}+y_{0}\cdot x_{1}),$
\item[(Z68)]
$(u_{0} \ast_{2} u_{1})\trl_{3} y_{0}=u_{0}\trl_{2} (u_{1}\ppr_{3} y_{0}+y_{0} \ppl_{2} u_{1})+ u_{0} \ast_{2} (u_{1}\trl_{3} y_{0}+y_{0} \trr_{2} u_{1}),$
\item[(Z69)]
$(x_{1} \ppl_{3} u_{0})\cdot y_{1}+ (x_{1} \trr_{3} u_{0})\ppr_{1} y_{1}=x_{1}\cdot (u_{0}\ppr_{2} y_{1}+y_{1} \ppl_{3} u_{0})+ x_{1} \ppl_{1} (u_{0}\trl_{2} y_{1}+y_{1} \trr_{3} u_{0}),$
\item[(Z70)]
$(u_{1}\ppr_{3} x_{0})\cdot x_{1} + (u_{1}\trl_{3} x_{0})\ppr_{1} x_{1}=u_{1}\ppr_{1} (x_{0}\cdot x_{1}+x_{1}\cdot x_{0}),$
\item[(Z71)]
$\omega_{3}(u_{1}, u_{0})\cdot x_{1}+ (u_{1} \ast_{3} u_{0})\ppr_{1} x_{1}=u_{1}\ppr_{1} (u_{0}\ppr_{2} x_{1}+x_{1} \ppl_{3} u_{0})+ \omega_{1}(u_{1}, u_{0}\trl_{2} x_{1}+ x_{1} \trr_{3} u_{0}),$
\item[(Z72)]
$(x_{1}\cdot x_{0}) \ppl_{1} u_{1}=x_{1}\cdot (x_{0} \ppl_{2} u_{1}+u_{1}\ppr_{3} x_{0})+x_{1} \ppl_{1} (x_{0} \trr_{2} u_{1}+u_{1}\trl_{3} x_{0}),$
\item[(Z73)]
$(x_{1} \ppl_{3} u_{0}) \ppl_{1} u_{1}+ \omega_{1}(x_{1} \trr_{3} u_{0}, u_{1})=x_{1}\cdot (\omega_{2}(u_{0}, u_{1})+ \omega_{3}(u_{1}, u_{0})) + x_{1} \ppl_{1} (u_{0} \ast_{2} u_{1}+u_{1} \ast_{3} u_{0}) ,$
\item[(Z74)]
$(u_{1}\ppr_{3} x_{0}) \ppl_{1} v_{1} + \omega_{1}(u_{1}\trl_{3} x_{0}, v_{1})=u_{1}\ppr_{1} (x_{0} \ppl_{2} v_{1}+v_{1}\ppr_{3} x_{0}) + \omega_{1}(u_{1}, x_{0} \trr_{2} v_{1}+ v_{1}\trl_{3} x_{0}),$
\item[(Z75)]
$\omega_{3}(u_{1}, u_{0}) \ppl_{1} v_{1}+ \omega_{1}(u_{1} \ast_{3} u_{0}, v_{1})=u_{1}\ppr_{1} (\omega_{2}(u_{0}, v_{1}) + \omega_{3}(v_{1}, u_{0})) + \omega_{1}(u_{1}, u_{0} \ast_{2} v_{1}+v_{1} \ast_{3} u_{0}),$
\item[(Z76)]
$(x_{1}\cdot x_{0}) \trr_{1} u_{1}=x_{1} \trr_{1} (x_{0} \trr_{2} u_{1}+u_{1}\trl_{3} x_{0}),$
\item[(Z77)]
$(x_{1} \ppl_{3} u_{0}) \trr_{1} u_{1}+ (x_{1} \trr_{3} u_{0}) \ast_{1} u_{1}=x_{1} \trr_{1} (u_{0} \ast_{2} u_{1}+u_{1} \ast_{3} u_{0}),$
\item[(Z78)]
$(u_{1}\ppr_{3} x_{0}) \trr_{1} v_{1} + (u_{1}\trl_{3} x_{0}) \ast_{1} v_{1}=u_{1}\trl_{1} (x_{0} \ppl_{2} v_{1}+v_{1}\ppr_{3} x_{0})+ u_{1} \ast_{1} (x_{0} \trr_{2} v_{1}+v_{1}\trl_{3} x_{0}),$
\item[(Z79)]
$\omega_{3}(u_{1}, u_{0}) \trr_{1} v_{1} + (u_{1} \ast_{3} u_{0}) \ast_{1} v_{1}= u_{1}\trl_{1} (\omega_{2}(u_{0}, v_{1})+ \omega_{3}(v_{1}, u_{0}))+ u_{1} \ast_{1} (u_{0} \ast_{2} v_{1}+v_{1} \ast_{3} u_{0}),$
\item[(Z80)]
$(x_{1} \trr_{3} u_{0})\trl_{1} y_{1}=x_{1} \trr_{1} (u_{0}\trl_{2} y_{1}+y_{1} \trr_{3} u_{0}),$
\item[(Z81)]
$(u_{1}\trl_{3} x_{0})\trl_{1} x_{1}=u_{1}\trl_{1} (x_{0}\cdot x_{1}+x_{1}\cdot x_{0}),$
\item[(Z82)]
$(u_{1} \ast_{3} u_{0})\trl_{1} x_{1}=u_{1}\trl_{1} (u_{0}\ppr_{2} x_{1}+x_{1} \ppl_{3} u_{0})+ u_{1} \ast_{1} (u_{0}\trl_{2} x_{1}+x_{1} \trr_{3} u_{0}) ,$

\item[(Z83)]
$(x_{1} \ppl_{1} u_{1})\cdot x_{0}+ (x_{1} \trr_{1} u_{1})\ppr_{3} x_{0}=x_{1}\cdot (u_{1}\ppr_{3} x_{0}+x_{0} \ppl_{2} u_{1})+ x_{1} \ppl_{1} (u_{1}\trl_{3} x_{0}+x_{0} \trr_{2} u_{1}),$
\item[(Z84)]
$(u_{1}\ppr_{1} x_{1})\cdot x_{0} + (u_{1}\trl_{1} x_{1})\ppr_{3} x_{0}= u_{1}\ppr_{1} (x_{1}\cdot x_{0}+x_{0}\cdot x_{1}),$
\item[(Z85)]
$\omega_{1}(u_{1}, v_{1})\cdot x_{0} + (u_{1} \ast_{1} v_{1})\ppr_{3} x_{0}=u_{1}\ppr_{1} (v_{1}\ppr_{3} x_{0}+x_{0} \ppl_{2} v_{1})+ \omega_{1}(u_{1},v_{1}\trl_{3} x_{0}+x_{0} \trr_{2} v_{1}),$
\item[(Z86)]
$(x_{1}\cdot y_{1}) \ppl_{3} u_{0}=x_{1}\cdot (y_{1} \ppl_{3} u_{0}+u_{0}\ppr_{2} y_{1})+ x_{1} \ppl_{1} (y_{1} \trr_{3} u_{0}+u_{0}\trl_{2} y_{1}),$
\item[(Z87)]
$(x_{1} \ppl_{1} u_{1}) \ppl_{3} u_{0}+ \omega_{3}(x_{1} \trr_{1} u_{1}, u_{0})=x_{1}\cdot (\omega_{3}(u_{1}, u_{0})+\omega_{2}(u_{0}, u_{1}))+ x_{1} \ppl_{1} (u_{1} \ast_{3} u_{0}+u_{0} \ast_{2} u_{1}),$
\item[(Z88)]
$(u_{1}\ppr_{1} x_{1}) \ppl_{3} u_{0} + \omega_{3}(u_{1}\trl_{1} x_{1}, u_{0})= u_{1}\ppr_{1} (x_{1} \ppl_{3} u_{0}+u_{0}\ppr_{2} x_{1}) + \omega_{1}(u_{1}, x_{1} \trr_{3} u_{0}+u_{0}\trl_{2} x_{1}),$
\item[(Z89)]
$\omega_{1}(u_{1}, v_{1}) \ppl_{3} u_{0}+ \omega_{3}(u_{1} \ast_{1} v_{1}, u_{0})=u_{1}\ppr_{1} (\omega_{3}(v_{1}, u_{0})+\omega_{2}(u_{0}, v_{1})) + \omega_{1}(u_{1},v_{1} \ast_{3} u_{0}+ u_{0} \ast_{2} v_{1}),$
\item[(Z90)]
$(x_{1}\cdot y_{1}) \trr_{3} u_{0}=x_{1} \trr_{1} (y_{1} \trr_{3} u_{0}+u_{0}\trl_{2} y_{1}),$
\item[(Z91)]
$(x_{1} \ppl_{1} u_{1}) \trr_{3} u_{0}+ (x_{1} \trr_{1} u_{1}) \ast_{3} u_{0}=x_{1} \trr_{1} (u_{1} \ast_{3} u_{0}+u_{0} \ast_{2} u_{1}),$
\item[(Z92)]
$(u_{1}\ppr_{1} x_{1}) \trr_{3} u_{0} + (u_{1}\trl_{1} x_{1}) \ast_{3} u_{0}=u_{1}\trl_{1} (x_{1} \ppl_{3} u_{0}+u_{0}\ppr_{2} x_{1})+ u_{1} \ast_{1} (x_{1} \trr_{3} u_{0}+u_{0}\trl_{2} x_{1}),$
\item[(Z93)]
$\omega_{1}(u_{1}, v_{1}) \trr_{3} u_{0} + (u_{1} \ast_{1} v_{1}) \ast_{3} u_{0}=u_{1}\trl_{1} (\omega_{3}(v_{1}, u_{0})+\omega_{2}(u_{0}, v_{1}))+ u_{1} \ast_{1} (v_{1} \ast_{3} u_{0}+u_{0} \ast_{2} v_{1}),$
\item[(Z94)]
$(x_{1} \trr_{1} u_{1})\trl_{3} x_{0}=x_{1} \trr_{1} (u_{1}\trl_{3} x_{0}+x_{0} \trr_{2} u_{1}),$
\item[(Z95)]
$(u_{1}\trl_{1} x_{1})\trl_{3} x_{0}=u_{1}\trl_{1} (x_{1}\cdot x_{0}+x_{0}\cdot x_{1}),$
\item[(Z96)]
$(u_{1} \ast_{1} v_{1})\trl_{3} x_{0}=u_{1}\trl_{1} (v_{1}\ppr_{3} x_{0}+x_{0} \ppl_{2} v_{1})+ u_{1} \ast_{1} (v_{1}\trl_{3} x_{0}+x_{0} \trr_{2} v_{1}),$
\item[(Z97)]
$\varphi(x_{0} \ppl_{2} u_{1})+\sigma(x_{0} \trr_{2} u_{1})=x_{0}\cdot \sigma(u_{1}) + x_{0} \ppl_{0} d(u_{1}),$
\item[(Z98)]
$\varphi(u_{0}\ppr_{2} x_{1})+ \sigma(u_{0}\trl_{2} x_{1})=u_{0}\ppr_{0} \varphi(x_{1}),$
\item[(Z99)]
$\varphi\omega_{2}(u_{0}, u_{1}) + \sigma(u_{0} \ast_{2} u_{1})=u_{0}\ppr_{0} \sigma(u_{1}) + \omega_{0}(u_{0}, d(u_{1})),$
\item[(Z100)]
$d(x_{0} \trr_{2} u_{1})=x_{0} \trr_{0} d(u_{1}),$
\item[(Z101)]
$d(u_{0}\trl_{2} x_{1})= u_{0}\trl_{0} \varphi(x_{1}),$
\item[(Z102)]
$d(u_{0} \ast_{2} u_{1})= u_{0}\trl_{0}\sigma(u_{1}) + u_{0} \ast_{0} d(u_{1}),$
\item[(Z103)]$\varphi(x_{1} \ppl_{3} u_{0})+\sigma(x_{1} \trr_{3} u_{0})=\varphi(x_{1})\ppl_{0} u_{0},$
\item[(Z104)]$\varphi(u_{1}\ppr_{3} x_{0}) + \sigma(u_{1}\trl_{3} x_{0})=\sigma(u_{1})\cdot x_{0}+ d(u_{1})\ppr_{0} x_{0},$
\item[(Z105)]$\varphi\omega_{3}(u_{1}, u_{0}) + \sigma(u_{1} \ast_{3} u_{0})=\sigma(u_{1}) \ppl_{0} u_{0}+ \omega_{0}(d(u_{1}), u_{0}),$
\item[(Z106)]$d(x_{1} \trr_{3} u_{0})=\varphi(x_{1})\trr_{0} u_{0},$
\item[(Z107)]$d(u_{1}\trl_{3} x_{0}) =d(u_{1})\trl_{0} x_{0},$
\item[(Z108)]$d(u_{1} \ast_{3} u_{0}))=\sigma(u_{1}) \trr_{0} u_{0} + d(u_{1}) \ast_{0} u_{0},$
\item[(Z109)]$\sigma(u_{1})\cdot x_{1}+ d(u_{1})\ppr_{2} x_{1}=u_{1}\ppr_{1} x_{1},$
\item[(Z110)]$\varphi(x_{1})\ppl_{2} u_{1}=x_{1} \ppl_{1} u_{1},$
\item[(Z111)]$\sigma(u_{1}) \ppl_{2} v_{1} + \omega_{2}(d(u_{1}), v_{1})=\omega_{1}(u_{1}, v_{1}),$
\item[(Z112)]$\varphi(x_{1})\trr_{2} u_{1}=x_{1} \trr_{1} u_{1},$
\item[(Z113)]$\sigma(u_{1}) \trr_{2} v_{1} + d(u_{1}) \ast_{2} v_{1}=u_{1} \ast_{1} v_{1},$
\item[(Z114)]$d(u_{1})\trl_{2} x_{1} =u_{1}\trl_{1} x_{1},$
\item[(Z115)]$x_{1}\cdot \sigma(u_{1})+x_{1} \ppl_{3} d(u_{1})=x_{1} \ppl_{1} u_{1},$
\item[(Z116)]$u_{1}\ppr_{3} \varphi(x_{1})=u_{1}\ppr_{1} x_{1},$
\item[(Z117)]$u_{1}\ppr_{3}\sigma(v_{1}) + \omega_{3}(u_{1}, d(v_{1}))=\omega_{1}(u_{1}, v_{1}),$
\item[(Z118)]$x_{1} \trr_{3} d(u_{1})=x_{1} \trr_{1} u_{1},$
\item[(Z119)]$u_{1}\trl_{3} \varphi(x_{1})=u_{1}\trl_{1} x_{1},$
\item[(Z120)]$u_{1}\trl_{3}\sigma(v_{1}) + u_{1} \ast_{3} d(v_{1})=u_{1} \ast_{1} v_{1}.$
\end{enumerate}
\end{theorem}

\begin{proof}
By direct computations, $Z_{i}\oplus V_{i}$ is a Zinbiel algebra if and only if (Z1)--(Z12) hold, see \cite{ZZ21} for details.
Thus, $(E_1,E_0,\varphi_E)$ is a Zinbiel 2-algebra if and only if
$\trr_{}: (Z_{0}\oplus V_{0})\oplus (Z_{1}\oplus V_{1})\rightarrow Z_{1}\oplus V_{1}$,
and
$\trl_{}: (Z_{1}\oplus V_{1})\oplus (Z_{0}\oplus V_{0})\rightarrow Z_{1}\oplus V_{1}$
is Zinbiel algebra action, and $\varphi_E: Z_{1}\oplus V_{1}\rightarrow Z_{0}\oplus V_{0}$ satisfying the following conditions:
\begin{enumerate}
\item[(1)] \begin{eqnarray*}
&&((x_{0},u_{0})\trr_{}(x_{1},u_{1}))\circ_{1}(y_{1},v_{1})\\
&=&(x_{0},u_{0})\trr_{}((x_{1},u_{1})\circ_{1}(y_{1},v_{1}) + (y_{1},v_{1})\circ_{1}(x_{1},u_{1}) ),
\end{eqnarray*}
\item[(2)] \begin{eqnarray*}
&&((x_{0},u_{0})\circ_{0}(y_{0},v_{0}))\trr_{}(x_{1},u_{1})\\
&=&(x_{0},u_{0})\trr_{}((y_{0},v_{0})\trr_{}(x_{1},u_{1})+(x_{1},u_{1})\trl_{}(y_{0},v_{0})),
\end{eqnarray*}
\item[(3)] \begin{eqnarray*}
&&((x_{1},u_{1})\trl_{}(x_{0},u_{0}))\trl_{}(y_{0},v_{0})\\
&=&(x_{1},u_{1})\trl_{}((x_{0},u_{0})\circ_{0}(y_{0},v_{0})+(y_{0},v_{0})\circ_{0}(x_{0},u_{0})),
\end{eqnarray*}
\item[(4)] \begin{eqnarray*}
&&((x_{0},u_{0})\trr_{}(x_{1},u_{1}))\trl_{}(y_{0},v_{0})\\
&=&(x_{0},u_{0})\trr_{}((x_{1},u_{1})\trl_{}(y_{0},v_{0})+((y_{0},v_{0})\trr_{}(x_{1},u_{1})),
\end{eqnarray*}
\item[(5)] \begin{eqnarray*}
&&((x_{1},u_{1})\trl_{}(x_{0},u_{0}))\circ_{1}(y_{1},v_{1})\\
&=&(x_{1},u_{1})\circ_{1}((x_{0},u_{0})\trr_{}(y_{1},v_{1})+(y_{1},v_{1})\trl_{}(x_{0},u_{0})),
\end{eqnarray*}
\item[(6)] \begin{eqnarray*}
&&((x_{1},u_{1})\circ_{1}(y_{1},v_{1}))\trl_{}(x_{0},u_{0})\\
&=&(x_{1},u_{1})\circ_{1}((y_{1},v_{1})\trl_{}(x_{0},u_{0})+(x_{0},u_{0})\trr_{}(y_{1},v_{1})),
\end{eqnarray*}
\item[(7)] $$\varphi_E((x_{0},u_{0})\trr (x_{1},u_{1})) = (x_{0},u_{0})\circ_{0}\varphi_E(x_{1},u_{1}),$$
\item[(8)] $$\varphi_E((x_{1},u_{1})\trl (x_{0},u_{0})) = \varphi_E(x_{1},u_{1})\circ_{0}(x_{0},u_{0}),$$
\item[(9)] $$\varphi_E(x_{1},u_{1})\trr (y_{1},v_{1}) = (x_{1},u_{1})\circ_{1}(y_{1},v_{1})
=(x_{1},u_{1})\trl \varphi_E(y_{1},v_{1}),$$
\end{enumerate}
for all  $x_{i}, y_{i} \in Z_{i},$ $u_{i}, v_{i} \in V_{i}$, $i = 0,1$.
%

For equation $(1)$, the left hand side is equal to
\begin{eqnarray*}
&& ((x_{0},u_{0})\trr_{}(x_{1},u_{1}))\circ_{1}(y_{1},v_{1})\\
&=&\big( x_{0}\cdot x_{1} + x_{0} \ppl_{2} u_{1} + u_{0}\ppr_{2} x_{1} + \omega_{2}(u_{0}, u_{1}),\ \\
&& x_{0} \trr_{2} u_{1} + u_{0}\trl_{2} x_{1} + u_{0} \ast_{2} u_{1} \big)\circ_{1}(y_{1},v_{1})\\
&=&\big( (x_{0}\cdot x_{1})\cdot y_{1} + (x_{0} \ppl_{2} u_{1})\cdot y_{1} + (u_{0}\ppr_{2} x_{1})\cdot y_{1} + \omega_{2}(u_{0}, u_{1})\cdot y_{1} \\
&&+ (x_{0}\cdot x_{1}) \ppl_{1} v_{1} + (x_{0} \ppl_{2} u_{1}) \ppl_{1} v_{1} + (u_{0}\ppr_{2} x_{1}) \ppl_{1} v_{1} + \omega_{2}(u_{0}, u_{1}) \ppl_{1} v_{1} \\
&&+ ( x_{0} \trr_{2} u_{1})\ppr_{1} y_{1} + (u_{0}\trl_{2} x_{1})\ppr_{1} y_{1} + (u_{0} \ast_{2} u_{1})\ppr_{1} y_{1} \\
&&+ \omega_{1}( x_{0} \trr_{2} u_{1}, v_{1}) + \omega_{1}(u_{0}\trl_{2} x_{1}, v_{1}) + \omega_{1}(u_{0} \ast_{2} u_{1}, v_{1}),\ \\
&& (x_{0}\cdot x_{1}) \trr_{1} v_{1} + (x_{0} \ppl_{2} u_{1}) \trr_{1} v_{1} + (u_{0}\ppr_{2} x_{1}) \trr_{1} v_{1} + \omega_{2}(u_{0}, u_{1}) \trr_{1} v_{1}\\
&& + ( x_{0} \trr_{2} u_{1})\trl_{1} y_{1} + (u_{0}\trl_{2} x_{1})\trl_{1} y_{1} + (u_{0} \ast_{2} u_{1})\trl_{1} y_{1} \\
&& + ( x_{0} \trr_{2} u_{1}) \ast_{1} v_{1} + (u_{0}\trl_{2} x_{1}) \ast_{1} v_{1} + (u_{0} \ast_{2} u_{1}) \ast_{1} v_{1} \big),
\end{eqnarray*}

the right hand side is equal to
\begin{eqnarray*}
&&(x_{0},u_{0})\trr_{}((x_{1},u_{1})\circ_{1}(y_{1},v_{1}))\\
&=&(x_{0},u_{0})\trr_{}\big( x_{1}\cdot y_{1} + x_{1} \ppl_{1} v_{1} + u_{1}\ppr_{1} y_{1} + \omega_{1}(u_{1}, v_{1}),\  x_{1} \trr_{1} v_{1} + u_{1}\trl_{1} y_{1} + u_{1} \ast_{1} v_{1} \big)\\
&=&\big( x_{0}\cdot (x_{1}\cdot y_{1}) + x_{0}\cdot (x_{1} \ppl_{1} v_{1}) + x_{0}\cdot (u_{1}\ppr_{1} y_{1}) + x_{0}\cdot \omega_{1}(u_{1}, v_{1})\\
&&+ x_{0} \ppl_{2} (x_{1} \trr_{1} v_{1}) + x_{0} \ppl_{2} (u_{1}\trl_{1} y_{1}) + x_{0} \ppl_{2} (u_{1} \ast_{1} v_{1}) + u_{0}\ppr_{2} ( x_{1}\cdot y_{1})\\
&&+ u_{0}\ppr_{2} (x_{1} \ppl_{1} v_{1}) + u_{0}\ppr_{2} (u_{1}\ppr_{1} y_{1}) + u_{0}\ppr_{2} \omega_{1}(u_{1}, v_{1}) + \omega_{2}(u_{0}, x_{1} \trr_{1} v_{1}) \\
&&+ \omega_{2}(u_{0},u_{1}\trl_{1} y_{1}) + \omega_{2}(u_{0},u_{1} \ast_{1} v_{1}),\ \\
&&x_{0} \trr_{2} (x_{1} \trr_{1} v_{1}) + x_{0} \trr_{2} (u_{1}\trl_{1} y_{1}) + x_{0} \trr_{2} (u_{1} \ast_{1} v_{1}) + u_{0}\trl_{2} (x_{1}\cdot y_{1})\\
&&+ u_{0}\trl_{2} (x_{1} \ppl_{1} v_{1}) + u_{0}\trl_{2} (u_{1}\ppr_{1} y_{1}) + u_{0}\trl_{2} \omega_{1}(u_{1}, v_{1}) + u_{0} \ast_{2} (x_{1} \trr_{1} v_{1})\\
&&+  u_{0} \ast_{2} (u_{1}\trl_{1} y_{1}) +  u_{0} \ast_{2} (u_{1} \ast_{1} v_{1}) \big),
\end{eqnarray*}

\begin{eqnarray*}
&&(x_{0},u_{0})\trr_{}((y_{1},v_{1})\circ_{1}(x_{1},u_{1}) )\\
&=&(x_{0},u_{0})\trr_{}\big( y_{1}\cdot x_{1} + y_{1} \ppl_{1} u_{1} + v_{1}\ppr_{1} x_{1} + \omega_{1}(v_{1}, u_{1}),\ \\
&&  y_{1} \trr_{1} u_{1} + v_{1}\trl_{1} x_{1} + v_{1} \ast_{1} u_{1} \big)\\
&=&\big( x_{0}\cdot (y_{1}\cdot x_{1}) + x_{0}\cdot (y_{1} \ppl_{1} u_{1}) + x_{0}\cdot (v_{1}\ppr_{1} x_{1}) + x_{0}\cdot \omega_{1}(v_{1}, u_{1}))\\
&&+ x_{0} \ppl_{2} (y_{1} \trr_{1} u_{1}) + x_{0} \ppl_{2} (v_{1}\trl_{1} x_{1}) + x_{0} \ppl_{2} (v_{1} \ast_{1} u_{1}) \\
&&+ u_{0}\ppr_{2} (y_{1}\cdot x_{1}) + u_{0}\ppr_{2} (y_{1} \ppl_{1} u_{1}) + u_{0}\ppr_{2} (v_{1}\ppr_{1} x_{1}) + u_{0}\ppr_{2} \omega_{1}(v_{1}, u_{1}) \\
&&+ \omega_{2}(u_{0}, y_{1} \trr_{1} u_{1}) + \omega_{2}(u_{0}, v_{1}\trl_{1} x_{1}) + \omega_{2}(u_{0}, v_{1} \ast_{1} u_{1}),\ \\
&& x_{0} \trr_{2} (y_{1} \trr_{1} u_{1}) + x_{0} \trr_{2} (v_{1}\trl_{1} x_{1}) + x_{0} \trr_{2} (v_{1} \ast_{1} u_{1}) \\
&& + u_{0}\trl_{2} (y_{1}\cdot x_{1}) + u_{0}\trl_{2} (y_{1} \ppl_{1} u_{1}) + u_{0}\trl_{2} (v_{1}\ppr_{1} x_{1}) + u_{0}\trl_{2} \omega_{1}(v_{1}, u_{1})\\
&& + u_{0} \ast_{2} (y_{1} \trr_{1} u_{1}) + u_{0} \ast_{2} (v_{1}\trl_{1} x_{1}) + u_{0} \ast_{2} (v_{1} \ast_{1} u_{1}) \big).
\end{eqnarray*}
Equation (1) hold if and only if (Z13)--(Z26) hold.

For Equation $(2)$, the left hand side is equal to
\begin{eqnarray*}
&&((x_{0},u_{0})\circ_{0}(y_{0},v_{0}))\trr_{}(x_{1},u_{1})\\
&=&\big( x_{0}\cdot y_{0} + x_{0} \ppl_{0} v_{0} + u_{0}\ppr_{0} y_{0} + \omega_{0}(u_{0}, v_{0}),\ \\
&&  x_{0} \trr_{0} v_{0} + u_{0}\trl_{0} y_{0} + u_{0} \ast_{0} v_{0} \big)\trr_{}(x_{1},u_{1})\\
&=&\big( (x_{0}\cdot y_{0})\cdot x_{1} + (x_{0} \ppl_{0} v_{0})\cdot x_{1} + (u_{0}\ppr_{0} y_{0})\cdot x_{1} + \omega_{0}(u_{0}, v_{0})\cdot x_{1} \\
&&+ (x_{0}\cdot y_{0}) \ppl_{2} u_{1} + (x_{0} \ppl_{0} v_{0}) \ppl_{2} u_{1} + (u_{0}\ppr_{0} y_{0}) \ppl_{2} u_{1} + \omega_{0}(u_{0}, v_{0}) \ppl_{2} u_{1} \\
&&+ (x_{0} \trr_{0} v_{0})\ppr_{2} x_{1} + (u_{0}\trl_{0} y_{0})\ppr_{2} x_{1} + (u_{0} \ast_{0} v_{0})\ppr_{2} x_{1} \\
&&+ \omega_{2}(x_{0} \trr_{0} v_{0}, u_{1}) + \omega_{2}(u_{0}\trl_{0} y_{0}, u_{1}) + \omega_{2}(u_{0} \ast_{0} v_{0}, u_{1}),\ \\
&& (x_{0}\cdot y_{0}) \trr_{2} u_{1} + (x_{0} \ppl_{0} v_{0}) \trr_{2} u_{1} + (u_{0}\ppr_{0} y_{0}) \trr_{2} u_{1} + \omega_{0}(u_{0}, v_{0}) \trr_{2} u_{1} \\
&& + (x_{0} \trr_{0} v_{0})\trl_{2} x_{1} + (u_{0}\trl_{0} y_{0})\trl_{2} x_{1} + (u_{0} \ast_{0} v_{0})\trl_{2} x_{1} \\
&& + (x_{0} \trr_{0} v_{0}) \ast_{2} u_{1} + (u_{0}\trl_{0} y_{0}) \ast_{2} u_{1} + (u_{0} \ast_{0} v_{0}) \ast_{2} u_{1} \big),
\end{eqnarray*}

the right hand side is equal to
\begin{eqnarray*}
&&(x_{0},u_{0})\trr_{}((y_{0},v_{0})\trr_{}(x_{1},u_{1}))\\
&=&(x_{0},u_{0})\trr_{}\big( y_{0}\cdot x_{1} + y_{0} \ppl_{2} u_{1} + v_{0}\ppr_{2} x_{1} + \omega_{2}(v_{0}, u_{1}),\ \\
&& y_{0} \trr_{2} u_{1} + v_{0}\trl_{2} x_{1} + v_{0} \ast_{2} u_{1} \big)\\
&=&\big( x_{0}\cdot (y_{0}\cdot x_{1}) + x_{0}\cdot (y_{0} \ppl_{2} u_{1}) + x_{0}\cdot (v_{0}\ppr_{2} x_{1}) + x_{0}\cdot (\omega_{2}(v_{0}, u_{1})\\
&& + x_{0} \ppl_{2} (y_{0} \trr_{2} u_{1}) + x_{0} \ppl_{2} (v_{0}\trl_{2} x_{1}) + x_{0} \ppl_{2} (v_{0} \ast_{2} u_{1}) \\
&& + u_{0}\ppr_{2} (y_{0}\cdot x_{1}) + u_{0}\ppr_{2} (y_{0} \ppl_{2} u_{1}) + u_{0}\ppr_{2} (v_{0}\ppr_{2} x_{1}) + u_{0}\ppr_{2} \omega_{2}(v_{0}, u_{1})\\
&&+ \omega_{2}(u_{0}, y_{0} \trr_{2} u_{1}) + \omega_{2}(u_{0},v_{0}\trl_{2} x_{1}) + \omega_{2}(u_{0},v_{0} \ast_{2} u_{1}),\ \\
&& x_{0} \trr_{2} (y_{0} \trr_{2} u_{1}) + x_{0} \trr_{2} (v_{0}\trl_{2} x_{1}) + x_{0} \trr_{2} (v_{0} \ast_{2} u_{1})\\
&& + u_{0}\trl_{2} (y_{0}\cdot x_{1}) + u_{0}\trl_{2} (y_{0} \ppl_{2} u_{1}) + u_{0}\trl_{2} (v_{0}\ppr_{2} x_{1}) + u_{0}\trl_{2} \omega_{2}(v_{0}, u_{1})\\
&& + u_{0} \ast_{2} (y_{0} \trr_{2} u_{1}) + u_{0} \ast_{2} (v_{0}\trl_{2} x_{1}) + u_{0} \ast_{2} (v_{0} \ast_{2} u_{1}) \big),
\end{eqnarray*}

\begin{eqnarray*}
&&(x_{0},u_{0})\trr_{}((x_{1},u_{1})\trl_{}(y_{0},v_{0}))\\
&=&(x_{0},u_{0})\trr_{}\big( x_{1}\cdot y_{0} + x_{1} \ppl_{3} v_{0} + u_{1}\ppr_{3} y_{0} + \omega_{3}(u_{1}, v_{0}),\ \\
&& x_{1} \trr_{3} v_{0} + u_{1}\trl_{3} y_{0} + u_{1} \ast_{3} v_{0} \big)\\
&=&\big( x_{0}\cdot (x_{1}\cdot y_{0}) + x_{0}\cdot (x_{1} \ppl_{3} v_{0}) + x_{0}\cdot (u_{1}\ppr_{3} y_{0}) + x_{0}\cdot \omega_{3}(u_{1}, v_{0}) \\
&&+ x_{0} \ppl_{2} ( x_{1} \trr_{3} v_{0}) + x_{0} \ppl_{2} (u_{1}\trl_{3} y_{0}) + x_{0} \ppl_{2} (u_{1} \ast_{3} v_{0})\\
&& + u_{0}\ppr_{2} (x_{1}\cdot y_{0}) + u_{0}\ppr_{2} (x_{1} \ppl_{3} v_{0}) + u_{0}\ppr_{2} (u_{1}\ppr_{3} y_{0}) + u_{0}\ppr_{2} \omega_{3}(u_{1}, v_{0}) \\
&& + \omega_{2}(u_{0},  x_{1} \trr_{3} v_{0}) + \omega_{2}(u_{0},u_{1}\trl_{3} y_{0}) + \omega_{2}(u_{0},u_{1} \ast_{3} v_{0}),\ \\
&& x_{0} \trr_{2} ( x_{1} \trr_{3} v_{0}) + x_{0} \trr_{2} (u_{1}\trl_{3} y_{0}) + x_{0} \trr_{2} (u_{1} \ast_{3} v_{0})\\
&& + u_{0}\trl_{2} (x_{1}\cdot y_{0}) + u_{0}\trl_{2} (x_{1} \ppl_{3} v_{0}) + u_{0}\trl_{2} (u_{1}\ppr_{3} y_{0}) + u_{0}\trl_{2} \omega_{3}(u_{1}, v_{0}) \\
&& + u_{0} \ast_{2} ( x_{1} \trr_{3} v_{0}) + u_{0} \ast_{2} (u_{1}\trl_{3} y_{0}) + u_{0} \ast_{2} (u_{1} \ast_{3} v_{0}) \big).
\end{eqnarray*}
Equation (2) hold if and only if (Z27)--(Z40) hold.

For equation $(3)$, the left hand side is equal to
\begin{eqnarray*}
&&((x_{1},u_{1})\trl_{}(x_{0},u_{0}))\trl_{}(y_{0},v_{0})\\
&=&\big( x_{1}\cdot x_{0} + x_{1} \ppl_{3} u_{0} + u_{1}\ppr_{3} x_{0} + \omega_{3}(u_{1}, u_{0}),\  x_{1} \trr_{3} u_{0} + u_{1}\trl_{3} x_{0} + u_{1} \ast_{3} u_{0} \big)\trl_{}(y_{0},v_{0})\\
&=& \big( (x_{1}\cdot x_{0})\cdot y_{0} + (x_{1} \ppl_{3} u_{0})\cdot y_{0} + (u_{1}\ppr_{3} x_{0})\cdot y_{0} + \omega_{3}(u_{1}, u_{0})\cdot y_{0}\\
&&+ (x_{1}\cdot x_{0}) \ppl_{3} v_{0} + (x_{1} \ppl_{3} u_{0}) \ppl_{3} v_{0} + (u_{1}\ppr_{3} x_{0}) \ppl_{3} v_{0} + \omega_{3}(u_{1}, u_{0}) \ppl_{3} v_{0}\\
&&+ (x_{1} \trr_{3} u_{0})\ppr_{3} y_{0} + (u_{1}\trl_{3} x_{0})\ppr_{3} y_{0} + (u_{1} \ast_{3} u_{0})\ppr_{3} y_{0}\\
&&+ \omega_{3}(x_{1} \trr_{3} u_{0}, v_{0}) + \omega_{3}(u_{1}\trl_{3} x_{0}, v_{0}) + \omega_{3}(u_{1} \ast_{3} u_{0}, v_{0}),\ \\
&&(x_{1}\cdot x_{0}) \trr_{3} v_{0} + (x_{1} \ppl_{3} u_{0}) \trr_{3} v_{0} + (u_{1}\ppr_{3} x_{0}) \trr_{3} v_{0} + \omega_{3}(u_{1}, u_{0}) \trr_{3} v_{0}\\
&&+ (x_{1} \trr_{3} u_{0})\trl_{3} y_{0} + (u_{1}\trl_{3} x_{0})\trl_{3} y_{0} + (u_{1} \ast_{3} u_{0})\trl_{3} y_{0}\\
&&+ (x_{1} \trr_{3} u_{0}) \ast_{3} v_{0} + (u_{1}\trl_{3} x_{0}) \ast_{3} v_{0} + (u_{1} \ast_{3} u_{0}) \ast_{3} v_{0} \big),
\end{eqnarray*}

the right hand side is equal to
\begin{eqnarray*}
&&(x_{1},u_{1})\trl_{}((x_{0},u_{0})\circ_{0}(y_{0},v_{0}))\\
&=&(x_{1},u_{1})\trl_{}\big( x_{0}\cdot y_{0} + x_{0} \ppl_{0} v_{0} + u_{0}\ppr_{0} y_{0} + \omega_{0}(u_{0}, v_{0}),\  x_{0} \trr_{0} v_{0} + u_{0}\trl_{0} y_{0} + u_{0} \ast_{0} v_{0} \big)\\
&=&\big( x_{1}\cdot (x_{0}\cdot y_{0}) + x_{1}\cdot (x_{0} \ppl_{0} v_{0}) + x_{1}\cdot (u_{0}\ppr_{0} y_{0}) + x_{1}\cdot \omega_{0}(u_{0}, v_{0})\\
&&+ x_{1} \ppl_{3} (x_{0} \trr_{0} v_{0}) + x_{1} \ppl_{3} (u_{0}\trl_{0} y_{0}) + x_{1} \ppl_{3} (u_{0} \ast_{0} v_{0})\\
&&+ u_{1}\ppr_{3} (x_{0}\cdot y_{0}) + u_{1}\ppr_{3} (x_{0} \ppl_{0} v_{0}) + u_{1}\ppr_{3} (u_{0}\ppr_{0} y_{0}) + u_{1}\ppr_{3} \omega_{0}(u_{0}, v_{0})\\
&&+ \omega_{3}(u_{1}, x_{0} \trr_{0} v_{0}) + \omega_{3}(u_{1},u_{0}\trl_{0} y_{0}) + \omega_{3}(u_{1},u_{0} \ast_{0} v_{0}),\ \\
&&x_{1} \trr_{3} (x_{0} \trr_{0} v_{0}) + x_{1} \trr_{3} (u_{0}\trl_{0} y_{0}) + x_{1} \trr_{3} (u_{0} \ast_{0} v_{0})\\
&&+ u_{1}\trl_{3} (x_{0}\cdot y_{0}) + u_{1}\trl_{3} (x_{0} \ppl_{0} v_{0}) + u_{1}\trl_{3} (u_{0}\ppr_{0} y_{0}) + u_{1}\trl_{3} \omega_{0}(u_{0}, v_{0})\\
&&+ u_{1} \ast_{3} (x_{0} \trr_{0} v_{0}) + u_{1} \ast_{3} (u_{0}\trl_{0} y_{0}) + u_{1} \ast_{3} (u_{0} \ast_{0} v_{0}) \big),
\end{eqnarray*}

\begin{eqnarray*}
&&(x_{1},u_{1})\trl_{}((y_{0},v_{0})\circ_{0}(x_{0},u_{0}))\\
&=&(x_{1},u_{1})\trl_{}\big( y_{0}\cdot x_{0} + y_{0} \ppl_{0} u_{0} + v_{0}\ppr_{0} x_{0} + \omega_{0}(v_{0}, u_{0}),\  y_{0} \trr_{0} u_{0} + v_{0}\trl_{0} x_{0} + v_{0} \ast_{0} u_{0} \big)\\
&=&\big( x_{1}\cdot (y_{0}\cdot x_{0}) + x_{1}\cdot (y_{0} \ppl_{0} u_{0}) + x_{1}\cdot (v_{0}\ppr_{0} x_{0}) + x_{1}\cdot \omega_{0}(v_{0}, u_{0})\\
&&+ x_{1} \ppl_{3} (y_{0} \trr_{0} u_{0}) + x_{1} \ppl_{3} (v_{0}\trl_{0} x_{0}) + x_{1} \ppl_{3} (v_{0} \ast_{0} u_{0})\\
&&+ u_{1}\ppr_{3} (y_{0}\cdot x_{0}) + u_{1}\ppr_{3} (y_{0} \ppl_{0} u_{0}) + u_{1}\ppr_{3} (v_{0}\ppr_{0} x_{0}) + u_{1}\ppr_{3} \omega_{0}(v_{0}, u_{0})\\
&&+ \omega_{3}(u_{1}, y_{0} \trr_{0} u_{0}) + \omega_{3}(u_{1}, v_{0}\trl_{0} x_{0}) + \omega_{3}(u_{1},v_{0} \ast_{0} u_{0}),\\
&&x_{1} \trr_{3} (y_{0} \trr_{0} u_{0}) + x_{1} \trr_{3} (v_{0}\trl_{0} x_{0}) + x_{1} \trr_{3} (v_{0} \ast_{0} u_{0})\\
&&+ u_{1}\trl_{3} (y_{0}\cdot x_{0}) + u_{1}\trl_{3} (y_{0} \ppl_{0} u_{0}) + u_{1}\trl_{3} (v_{0}\ppr_{0} x_{0}) + u_{1}\trl_{3} \omega_{0}(v_{0}, u_{0})\\
&&+ u_{1} \ast_{3} (y_{0} \trr_{0} u_{0}) + u_{1} \ast_{3} (v_{0}\trl_{0} x_{0}) + u_{1} \ast_{3} (v_{0} \ast_{0} u_{0}) \big).
\end{eqnarray*}
Equation (3) hold if and only if (Z41)--(Z54) hold.

For equation $(4)$, the left hand side is equal to
\begin{eqnarray*}
&&((x_{0},u_{0})\trr_{}(x_{1},u_{1}))\trl_{}(y_{0},v_{0})\\
&=&\big( x_{0}\cdot x_{1} + x_{0} \ppl_{2} u_{1} + u_{0}\ppr_{2} x_{1} + \omega_{2}(u_{0}, u_{1}),\  x_{0} \trr_{2} u_{1} + u_{0}\trl_{2} x_{1} + u_{0} \ast_{2} u_{1} \big)\trl_{}(y_{0},v_{0})\\
&=&\big( (x_{0}\cdot x_{1})\cdot y_{0} + (x_{0} \ppl_{2} u_{1})\cdot y_{0} + (u_{0}\ppr_{2} x_{1})\cdot y_{0} + \omega_{2}(u_{0}, u_{1})\cdot y_{0}\\
&&+ (x_{0}\cdot x_{1}) \ppl_{3} v_{0} + (x_{0} \ppl_{2} u_{1}) \ppl_{3} v_{0} + (u_{0}\ppr_{2} x_{1}) \ppl_{3} v_{0} + \omega_{2}(u_{0}, u_{1}) \ppl_{3} v_{0}\\
&&+ (x_{0} \trr_{2} u_{1})\ppr_{3} y_{0} + (u_{0}\trl_{2} x_{1})\ppr_{3} y_{0} + (u_{0} \ast_{2} u_{1})\ppr_{3} y_{0}\\
&&+ \omega_{3}(x_{0} \trr_{2} u_{1}, v_{0}) + \omega_{3}(u_{0}\trl_{2} x_{1}, v_{0}) + \omega_{3}(u_{0} \ast_{2} u_{1}, v_{0}),\ \\
&&(x_{0}\cdot x_{1}) \trr_{3} v_{0} + (x_{0} \ppl_{2} u_{1}) \trr_{3} v_{0} + (u_{0}\ppr_{2} x_{1}) \trr_{3} v_{0} + \omega_{2}(u_{0}, u_{1}) \trr_{3} v_{0}\\
&&+ (x_{0} \trr_{2} u_{1})\trl_{3} y_{0} + (u_{0}\trl_{2} x_{1})\trl_{3} y_{0} + (u_{0} \ast_{2} u_{1})\trl_{3} y_{0}\\
&&+ (x_{0} \trr_{2} u_{1}) \ast_{3} v_{0} + (u_{0}\trl_{2} x_{1}) \ast_{3} v_{0} + (u_{0} \ast_{2} u_{1}) \ast_{3} v_{0} \big),
\end{eqnarray*}

the right hand side is equal to
\begin{eqnarray*}
&&(x_{0},u_{0})\trr_{}((x_{1},u_{1})\trl_{}(y_{0},v_{0}))\\
&=&(x_{0},u_{0})\trr_{}\big( x_{1}\cdot y_{0} + x_{1} \ppl_{3} v_{0} + u_{1}\ppr_{3} y_{0} + \omega_{3}(u_{1}, v_{0}),\  x_{1} \trr_{3} v_{0} + u_{1}\trl_{3} y_{0} + u_{1} \ast_{3} v_{0} \big)\\
&=&\big( x_{0}\cdot (x_{1}\cdot y_{0}) + x_{0}\cdot (x_{1} \ppl_{3} v_{0}) + x_{0}\cdot (u_{1}\ppr_{3} y_{0}) + x_{0}\cdot \omega_{3}(u_{1}, v_{0})\\
&&+ x_{0} \ppl_{2} (x_{1} \trr_{3} v_{0}) + x_{0} \ppl_{2} (u_{1}\trl_{3} y_{0}) + x_{0} \ppl_{2} (u_{1} \ast_{3} v_{0})\\
&&+ u_{0}\ppr_{2} (x_{1}\cdot y_{0}) + u_{0}\ppr_{2} (x_{1} \ppl_{3} v_{0}) + u_{0}\ppr_{2} (u_{1}\ppr_{3} y_{0}) + u_{0}\ppr_{2} \omega_{3}(u_{1}, v_{0})\\
&&+ \omega_{2}(u_{0}, x_{1} \trr_{3} v_{0}) + \omega_{2}(u_{0}, u_{1}\trl_{3} y_{0}) + \omega_{2}(u_{0}, u_{1} \ast_{3} v_{0}),\ \\
&&x_{0} \trr_{2} (x_{1} \trr_{3} v_{0}) + x_{0} \trr_{2} (u_{1}\trl_{3} y_{0}) + x_{0} \trr_{2} (u_{1} \ast_{3} v_{0})\\
&&+ u_{0}\trl_{2} (x_{1}\cdot y_{0}) + u_{0}\trl_{2} (x_{1} \ppl_{3} v_{0}) + u_{0}\trl_{2} (u_{1}\ppr_{3} y_{0}) + u_{0}\trl_{2} \omega_{3}(u_{1}, v_{0})\\
&&+ u_{0} \ast_{2} (x_{1} \trr_{3} v_{0}) + u_{0} \ast_{2} (u_{1}\trl_{3} y_{0}) + u_{0} \ast_{2} (u_{1} \ast_{3} v_{0}) \big),
\end{eqnarray*}

\begin{eqnarray*}
&&(x_{0},u_{0})\trr_{}((y_{0},v_{0})\trr_{}(x_{1},u_{1}))\\
&=&(x_{0},u_{0})\trr_{}\big( y_{0}\cdot x_{1} + y_{0} \ppl_{2} u_{1} + v_{0}\ppr_{2} x_{1} + \omega_{2}(v_{0}, u_{1}),\  y_{0} \trr_{2} u_{1} + v_{0}\trl_{2} x_{1} + v_{0} \ast_{2} u_{1} \big)\\
&=&\big( x_{0}\cdot (y_{0}\cdot x_{1}) + x_{0}\cdot (y_{0} \ppl_{2} u_{1}) + x_{0}\cdot (v_{0}\ppr_{2} x_{1}) + x_{0}\cdot \omega_{2}(v_{0}, u_{1})\\
&&+ x_{0} \ppl_{2} (y_{0} \trr_{2} u_{1}) + x_{0} \ppl_{2} (v_{0}\trl_{2} x_{1}) + x_{0} \ppl_{2} (v_{0} \ast_{2} u_{1})\\
&&+ u_{0}\ppr_{2} (y_{0}\cdot x_{1}) + u_{0}\ppr_{2} (y_{0} \ppl_{2} u_{1}) + u_{0}\ppr_{2} (v_{0}\ppr_{2} x_{1}) + u_{0}\ppr_{2} \omega_{2}(v_{0}, u_{1})\\
&&+ \omega_{2}(u_{0}, y_{0} \trr_{2} u_{1}) + \omega_{2}(u_{0}, v_{0}\trl_{2} x_{1}) + \omega_{2}(u_{0}, v_{0} \ast_{2} u_{1}),\ \\
&&x_{0} \trr_{2} (y_{0} \trr_{2} u_{1}) + x_{0} \trr_{2} (v_{0}\trl_{2} x_{1}) + x_{0} \trr_{2} (v_{0} \ast_{2} u_{1})\\
&&+ u_{0}\trl_{2} (y_{0}\cdot x_{1}) + u_{0}\trl_{2} (y_{0} \ppl_{2} u_{1}) + u_{0}\trl_{2} (v_{0}\ppr_{2} x_{1}) + u_{0}\trl_{2} \omega_{2}(v_{0}, u_{1})\\
&&+ u_{0} \ast_{2} (y_{0} \trr_{2} u_{1}) + u_{0} \ast_{2} (v_{0}\trl_{2} x_{1}) + u_{0} \ast_{2} (v_{0} \ast_{2} u_{1}) \big).
\end{eqnarray*}
Equation (4) hold if and only if (Z55)--(Z68) hold.

For equation $(5)$, the left hand side is equal to
\begin{eqnarray*}
&&((x_{1},u_{1})\trl_{}(x_{0},u_{0}))\circ_{1}(y_{1},v_{1})\\
&=&\big( x_{1}\cdot x_{0} + x_{1} \ppl_{3} u_{0} + u_{1}\ppr_{3} x_{0} + \omega_{3}(u_{1}, u_{0}),\, \, x_{1} \trr_{3} u_{0} + u_{1}\trl_{3} x_{0} + u_{1} \ast_{3} u_{0} \big)\circ_{1}(y_{1},v_{1})\\
&=&\big( (x_{1}\cdot x_{0})\cdot y_{1} + (x_{1} \ppl_{3} u_{0})\cdot y_{1} + (u_{1}\ppr_{3} x_{0})\cdot y_{1} + \omega_{3}(u_{1}, u_{0})\cdot y_{1}\\
&&+ (x_{1}\cdot x_{0}) \ppl_{1} v_{1} + (x_{1} \ppl_{3} u_{0}) \ppl_{1} v_{1} + (u_{1}\ppr_{3} x_{0}) \ppl_{1} v_{1} + \omega_{3}(u_{1}, u_{0}) \ppl_{1} v_{1}\\
&&+ (x_{1} \trr_{3} u_{0})\ppr_{1} y_{1} + (u_{1}\trl_{3} x_{0})\ppr_{1} y_{1} + (u_{1} \ast_{3} u_{0})\ppr_{1} y_{1} \\
&&+ \omega_{1}(x_{1} \trr_{3} u_{0}, v_{1}) + \omega_{1}(u_{1}\trl_{3} x_{0}, v_{1}) + \omega_{1}(u_{1} \ast_{3} u_{0}, v_{1}),\, \, \\
&&(x_{1}\cdot x_{0}) \trr_{1} v_{1} + (x_{1} \ppl_{3} u_{0}) \trr_{1} v_{1} + (u_{1}\ppr_{3} x_{0}) \trr_{1} v_{1} + \omega_{3}(u_{1}, u_{0}) \trr_{1} v_{1} \\
&&+ (x_{1} \trr_{3} u_{0})\trl_{1} y_{1} + (u_{1}\trl_{3} x_{0})\trl_{1} y_{1} + (u_{1} \ast_{3} u_{0})\trl_{1} y_{1}\\
&&+ (x_{1} \trr_{3} u_{0}) \ast_{1} v_{1} + (u_{1}\trl_{3} x_{0}) \ast_{1} v_{1} + (u_{1} \ast_{3} u_{0}) \ast_{1} v_{1} \big),
\end{eqnarray*}

the right hand side is equal to
\begin{eqnarray*}
&&(x_{1},u_{1})\circ_{1}((x_{0},u_{0})\trr_{}(y_{1},v_{1}))\\
&=&(x_{1},u_{1})\circ_{1}\big( x_{0}\cdot y_{1} + x_{0} \ppl_{2} v_{1} + u_{0}\ppr_{2} y_{1} + \omega_{2}(u_{0}, v_{1}),\, \, x_{0} \trr_{2} v_{1} + u_{0}\trl_{2} y_{1} + u_{0} \ast_{2} v_{1} \big)\\
&=&\big( x_{1}\cdot (x_{0}\cdot y_{1}) + x_{1}\cdot (x_{0} \ppl_{2} v_{1}) + x_{1}\cdot (u_{0}\ppr_{2} y_{1}) + x_{1}\cdot \omega_{2}(u_{0}, v_{1})\\
&&+ x_{1} \ppl_{1} (x_{0} \trr_{2} v_{1}) + x_{1} \ppl_{1} (u_{0}\trl_{2} y_{1}) + x_{1} \ppl_{1} (u_{0} \ast_{2} v_{1})\\
&&+ u_{1}\ppr_{1} (x_{0}\cdot y_{1}) + u_{1}\ppr_{1} (x_{0} \ppl_{2} v_{1}) + u_{1}\ppr_{1} (u_{0}\ppr_{2} y_{1}) + u_{1}\ppr_{1} \omega_{2}(u_{0}, v_{1})\\
&&+ \omega_{1}(u_{1}, x_{0} \trr_{2} v_{1}) + \omega_{1}(u_{1}, u_{0}\trl_{2} y_{1}) + \omega_{1}(u_{1}, u_{0} \ast_{2} v_{1}),\, \,\\
&&x_{1} \trr_{1} (x_{0} \trr_{2} v_{1}) + x_{1} \trr_{1} (u_{0}\trl_{2} y_{1}) + x_{1} \trr_{1} (u_{0} \ast_{2} v_{1})\\
&&+ u_{1}\trl_{1} (x_{0}\cdot y_{1}) + u_{1}\trl_{1} (x_{0} \ppl_{2} v_{1}) + u_{1}\trl_{1} (u_{0}\ppr_{2} y_{1}) + u_{1}\trl_{1} \omega_{2}(u_{0}, v_{1})\\
&&+ u_{1} \ast_{1} (x_{0} \trr_{2} v_{1}) + u_{1} \ast_{1} (u_{0}\trl_{2} y_{1}) + u_{1} \ast_{1} (u_{0} \ast_{2} v_{1}) \big),
\end{eqnarray*}
\begin{eqnarray*}
&&(x_{1},u_{1})\circ_{1}((y_{1},v_{1})\trl_{}(x_{0},u_{0}))\\
&=&(x_{1},u_{1})\circ_{1}\big( y_{1}\cdot x_{0} + y_{1} \ppl_{3} u_{0} + v_{1}\ppr_{3} x_{0} + \omega_{3}(v_{1}, u_{0}),\  y_{1} \trr_{3} u_{0} + v_{1}\trl_{3} x_{0} + v_{1} \ast_{3} u_{0} \big)\\
&=&\big( x_{1}\cdot (y_{1}\cdot x_{0}) + x_{1}\cdot (y_{1} \ppl_{3} u_{0}) + x_{1}\cdot (v_{1}\ppr_{3} x_{0}) + x_{1}\cdot \omega_{3}(v_{1}, u_{0}) \\
&&+ x_{1} \ppl_{1} (y_{1} \trr_{3} u_{0}) + x_{1} \ppl_{1} (v_{1}\trl_{3} x_{0}) + x_{1} \ppl_{1} (v_{1} \ast_{3} u_{0}) \\
&&+ u_{1}\ppr_{1} (y_{1}\cdot x_{0}) + u_{1}\ppr_{1} (y_{1} \ppl_{3} u_{0}) + u_{1}\ppr_{1} (v_{1}\ppr_{3} x_{0}) + u_{1}\ppr_{1} \omega_{3}(v_{1}, u_{0}) \\
&&+ \omega_{1}(u_{1}, y_{1} \trr_{3} u_{0}) + \omega_{1}(u_{1}, v_{1}\trl_{3} x_{0}) + \omega_{1}(u_{1}, v_{1} \ast_{3} u_{0}),\, \, \\
&&x_{1} \trr_{1} (y_{1} \trr_{3} u_{0}) + x_{1} \trr_{1} (v_{1}\trl_{3} x_{0}) + x_{1} \trr_{1} (v_{1} \ast_{3} u_{0}) \\
&&+ u_{1}\trl_{1} (y_{1}\cdot x_{0}) + u_{1}\trl_{1} (y_{1} \ppl_{3} u_{0}) + u_{1}\trl_{1} (v_{1}\ppr_{3} x_{0}) + u_{1}\trl_{1} \omega_{3}(v_{1}, u_{0})\\
&&+ u_{1} \ast_{1} (y_{1} \trr_{3} u_{0}) + u_{1} \ast_{1} (v_{1}\trl_{3} x_{0}) + u_{1} \ast_{1} (v_{1} \ast_{3} u_{0})\big).
\end{eqnarray*}
Equation (5) hold if and only if (Z69)--(Z82) hold.

For Equation $(6)$, the left hand side is equal to
\begin{eqnarray*}
&&((x_{1},u_{1})\circ_{1}(y_{1},v_{1}))\trl_{}(x_{0},u_{0})\\
&=&\big( x_{1}\cdot y_{1} + x_{1} \ppl_{1} v_{1} + u_{1}\ppr_{1} y_{1} + \omega_{1}(u_{1}, v_{1}),\, \, x_{1} \trr_{1} v_{1} + u_{1}\trl_{1} y_{1} + u_{1} \ast_{1} v_{1}\big)\trl_{}(x_{0},u_{0})\\
&=&\big( (x_{1}\cdot y_{1})\cdot x_{0} + (x_{1} \ppl_{1} v_{1})\cdot x_{0} + (u_{1}\ppr_{1} y_{1})\cdot x_{0} + \omega_{1}(u_{1}, v_{1})\cdot x_{0}\\
&&+ (x_{1}\cdot y_{1}) \ppl_{3} u_{0} + (x_{1} \ppl_{1} v_{1}) \ppl_{3} u_{0} + (u_{1}\ppr_{1} y_{1}) \ppl_{3} u_{0} + \omega_{1}(u_{1}, v_{1}) \ppl_{3} u_{0}\\
&&+ (x_{1} \trr_{1} v_{1})\ppr_{3} x_{0} + (u_{1}\trl_{1} y_{1})\ppr_{3} x_{0} + (u_{1} \ast_{1} v_{1})\ppr_{3} x_{0}\\
&&+ \omega_{3}(x_{1} \trr_{1} v_{1}, u_{0}) + \omega_{3}(u_{1}\trl_{1} y_{1}, u_{0}) + \omega_{3}(u_{1} \ast_{1} v_{1}, u_{0}),\, \,\\
&&(x_{1}\cdot y_{1}) \trr_{3} u_{0} + (x_{1} \ppl_{1} v_{1}) \trr_{3} u_{0} + (u_{1}\ppr_{1} y_{1}) \trr_{3} u_{0} + \omega_{1}(u_{1}, v_{1}) \trr_{3} u_{0}\\
&&+ (x_{1} \trr_{1} v_{1})\trl_{3} x_{0} + (u_{1}\trl_{1} y_{1})\trl_{3} x_{0} + (u_{1} \ast_{1} v_{1})\trl_{3} x_{0}\\
&&+ (x_{1} \trr_{1} v_{1}) \ast_{3} u_{0} + (u_{1}\trl_{1} y_{1}) \ast_{3} u_{0} + (u_{1} \ast_{1} v_{1}) \ast_{3} u_{0} \big),
\end{eqnarray*}

the right hand side is equal to
\begin{eqnarray*}
&&(x_{1},u_{1})\circ_{1}((y_{1},v_{1})\trl_{}(x_{0},u_{0}))\\
&=&(x_{1},u_{1})\circ_{1}\big( y_{1}\cdot x_{0} + y_{1} \ppl_{3} u_{0} + v_{1}\ppr_{3} x_{0} + \omega_{3}(v_{1}, u_{0}),\, \, y_{1} \trr_{3} u_{0} + v_{1}\trl_{3} x_{0} + v_{1} \ast_{3} u_{0} \big)\\
&=&\big( x_{1}\cdot (y_{1}\cdot x_{0}) + x_{1}\cdot (y_{1} \ppl_{3} u_{0}) + x_{1}\cdot (v_{1}\ppr_{3} x_{0}) + x_{1}\cdot \omega_{3}(v_{1}, u_{0})\\
&&+ x_{1} \ppl_{1} (y_{1} \trr_{3} u_{0}) + x_{1} \ppl_{1} (v_{1}\trl_{3} x_{0}) + x_{1} \ppl_{1} (v_{1} \ast_{3} u_{0})\\
&&+ u_{1}\ppr_{1} (y_{1}\cdot x_{0}) + u_{1}\ppr_{1} (y_{1} \ppl_{3} u_{0}) + u_{1}\ppr_{1} (v_{1}\ppr_{3} x_{0}) + u_{1}\ppr_{1} \omega_{3}(v_{1}, u_{0})\\
&&+ \omega_{1}(u_{1}, y_{1} \trr_{3} u_{0}) + \omega_{1}(u_{1},v_{1}\trl_{3} x_{0}) + \omega_{1}(u_{1},v_{1} \ast_{3} u_{0}),\, \,\\
&&x_{1} \trr_{1} (y_{1} \trr_{3} u_{0}) + x_{1} \trr_{1} (v_{1}\trl_{3} x_{0}) + x_{1} \trr_{1} (v_{1} \ast_{3} u_{0})\\
&&+ u_{1}\trl_{1} (y_{1}\cdot x_{0}) + u_{1}\trl_{1} (y_{1} \ppl_{3} u_{0}) + u_{1}\trl_{1} (v_{1}\ppr_{3} x_{0}) + u_{1}\trl_{1} \omega_{3}(v_{1}, u_{0})\\
&&+ u_{1} \ast_{1} (y_{1} \trr_{3} u_{0}) +  u_{1} \ast_{1} (v_{1}\trl_{3} x_{0}) +  u_{1} \ast_{1} (v_{1} \ast_{3} u_{0})\big),
\end{eqnarray*}

\begin{eqnarray*}
&&(x_{1},u_{1})\circ_{1}((x_{0},u_{0})\trr_{}(y_{1},v_{1}))\\
&=&(x_{1},u_{1})\circ_{1}\big( x_{0}\cdot y_{1} + x_{0} \ppl_{2} v_{1} + u_{0}\ppr_{2} y_{1} + \omega_{2}(u_{0}, v_{1}),\  x_{0} \trr_{2} v_{1} + u_{0}\trl_{2} y_{1} + u_{0} \ast_{2} v_{1} \big)\\
&=&\big( x_{1}\cdot (x_{0}\cdot y_{1}) + x_{1}\cdot (x_{0} \ppl_{2} v_{1}) + x_{1}\cdot (u_{0}\ppr_{2} y_{1}) + x_{1}\cdot \omega_{2}(u_{0}, v_{1}) \\
&&+ x_{1} \ppl_{1} (x_{0} \trr_{2} v_{1}) + x_{1} \ppl_{1} (u_{0}\trl_{2} y_{1}) + x_{1} \ppl_{1} (u_{0} \ast_{2} v_{1}) \\
&&+ u_{1}\ppr_{1} (x_{0}\cdot y_{1}) + u_{1}\ppr_{1} (x_{0} \ppl_{2} v_{1}) + u_{1}\ppr_{1} (u_{0}\ppr_{2} y_{1}) + u_{1}\ppr_{1} \omega_{2}(u_{0}, v_{1}) \\
&&+ \omega_{1}(u_{1}, x_{0} \trr_{2} v_{1}) + \omega_{1}(u_{1}, u_{0}\trl_{2} y_{1}) + \omega_{1}(u_{1}, u_{0} \ast_{2} v_{1}),\, \, \\
&&x_{1} \trr_{1} (x_{0} \trr_{2} v_{1}) + x_{1} \trr_{1} (u_{0}\trl_{2} y_{1}) + x_{1} \trr_{1} (u_{0} \ast_{2} v_{1}) \\
&&+ u_{1}\trl_{1} (x_{0}\cdot y_{1}) + u_{1}\trl_{1} (x_{0} \ppl_{2} v_{1}) + u_{1}\trl_{1} (u_{0}\ppr_{2} y_{1}) + u_{1}\trl_{1} \omega_{2}(u_{0}, v_{1}) \\
&&+ u_{1} \ast_{1} (x_{0} \trr_{2} v_{1}) + u_{1} \ast_{1} (u_{0}\trl_{2} y_{1}) + u_{1} \ast_{1} (u_{0} \ast_{2} v_{1})\big).
\end{eqnarray*}
Equation (6) hold if and only if (Z83)--(Z96) hold.

For equation $(7)$, the left hand side is equal to
\begin{eqnarray*}
&&\varphi_E((x_{0},u_{0})\trr_{}(x_{1},u_{1}))\\
&=&\varphi_E\big( x_{0}\cdot x_{1} + x_{0} \ppl_{2} u_{1} + u_{0}\ppr_{2} x_{1} + \omega_{2}(u_{0}, u_{1}),\
 x_{0} \trr_{2} u_{1} + u_{0}\trl_{2} x_{1} + u_{0} \ast_{2} u_{1} \big)\\
&=&\big(\varphi(x_{0}\cdot x_{1}) + \varphi(x_{0} \ppl_{2} u_{1}) + \varphi(u_{0}\ppr_{2} x_{1}) + \varphi\omega_{2}(u_{0}, u_{1})+\sigma(x_{0} \trr_{2} u_{1})\\
&&+ \sigma(u_{0}\trl_{2} x_{1}) + \sigma(u_{0} \ast_{2} u_{1}), d(x_{0} \trr_{2} u_{1}) + d(u_{0}\trl_{2} x_{1}) + d(u_{0} \ast_{2} u_{1})\big),
\end{eqnarray*}

the right hand side is equal to
\begin{eqnarray*}
&&(x_{0},u_{0})\circ_{0}\varphi_E(x_{1},u_{1})\\
&=&(x_{0},u_{0})\circ_{0}(\varphi(x_{1})+\sigma(u_{1}), d(u_{1}))\\
&=&\big( x_{0}\cdot \varphi(x_{1})+x_{0}\cdot \sigma(u_{1}) + x_{0} \ppl_{0} d(u_{1})
 + u_{0}\ppr_{0} \varphi(x_{1})+u_{0}\ppr_{0} \sigma(u_{1}) \\
&& + \omega_{0}(u_{0}, d(u_{1})),\  x_{0} \trr_{0} d(u_{1}) + u_{0}\trl_{0} \varphi(x_{1})+u_{0}\trl_{0}\sigma(u_{1}) + u_{0} \ast_{0} d(u_{1}) \big).
\end{eqnarray*}
Equation (7) hold if and only if (Z97)--(Z102) hold.

For equation $(8)$, the left hand side is equal to
\begin{eqnarray*}
&&\varphi_E((x_{1},u_{1})\trl_{}(x_{0},u_{0}))\\
&=&\varphi_E\big( x_{1}\cdot x_{0} + x_{1} \ppl_{3} u_{0} + u_{1}\ppr_{3} x_{0} + \omega_{3}(u_{1}, u_{0}),\  x_{1} \trr_{3} u_{0} + u_{1}\trl_{3} x_{0} + u_{1} \ast_{3} u_{0} \big)\\
&=&(\varphi(x_{1}\cdot x_{0}) + \varphi(x_{1} \ppl_{3} u_{0}) + \varphi(u_{1}\ppr_{3} x_{0}) + \varphi\omega_{3}(u_{1}, u_{0})+\sigma(x_{1} \trr_{3} u_{0})\\
&& + \sigma(u_{1}\trl_{3} x_{0}) + \sigma(u_{1} \ast_{3} u_{0}), d(x_{1} \trr_{3} u_{0}) + d(u_{1}\trl_{3} x_{0}) + d(u_{1} \ast_{3} u_{0})),
\end{eqnarray*}

the right hand side is equal to
\begin{eqnarray*}
&&\varphi_E(x_{1},u_{1})\circ_{0}(x_{0},u_{0})\\
&=&(\varphi(x_{1})+\sigma(u_{1}), d(u_{1}))\circ_{0}(x_{0},u_{0})\\
&=&\big( \varphi(x_{1})\cdot x_{0}+\sigma(u_{1})\cdot x_{0}+ \varphi(x_{1})\ppl_{0} u_{0}+\sigma(u_{1}) \ppl_{0} u_{0}+ d(u_{1})\ppr_{0} x_{0}\\
&&+ \omega_{0}(d(u_{1}), u_{0}),\ \varphi(x_{1})\trr_{0} u_{0}+\sigma(u_{1}) \trr_{0} u_{0}+ d(u_{1})\trl_{0} x_{0}+ d(u_{1}) \ast_{0} u_{0} \big).
\end{eqnarray*}
Equation (8) hold if and only if (Z103)--(Z108) hold.

For equation $(9)$, the left hand side is equal to
$$
\begin{aligned}
&\varphi_E(x_{1},u_{1})\trr_{}(y_{1},v_{1})\\
&=(\varphi(x_{1})+\sigma(u_{1}), d(u_{1}))\trr_{}(y_{1},v_{1})\\
&=\big( \varphi(x_{1})\cdot y_{1}+\sigma(u_{1})\cdot y_{1} + \varphi(x_{1})\ppl_{2} v_{1}+\sigma(u_{1}) \ppl_{2} v_{1}+ d(u_{1})\ppr_{2} y_{1} \\
&+ \omega_{2}(d(u_{1}), v_{1}),\  \varphi(x_{1})\trr_{2} v_{1}+\sigma(u_{1}) \trr_{2} v_{1}+ d(u_{1})\trl_{2} y_{1} + d(u_{1}) \ast_{2} v_{1} \big),
\end{aligned}
$$
the right hand side is equal to
$$
\begin{aligned}
&(x_{1},u_{1})\circ_{1}(y_{1},v_{1})\\
&=\big( x_{1}\cdot y_{1} + x_{1} \ppl_{1} v_{1} + u_{1}\ppr_{1} y_{1} + \omega_{1}(u_{1}, v_{1}),\
x_{1} \trr_{1} v_{1} + u_{1}\trl_{1} y_{1} + u_{1} \ast_{1} v_{1} \big),\\
&(x_{1},u_{1})\trl_{}\varphi_E(y_{1},v_{1})\\
&=(x_{1},u_{1})\trl_{}(\varphi(y_{1})+\sigma(v_{1}), d(v_{1}))\\
&=\big( x_{1}\cdot \varphi(y_{1})+x_{1}\cdot \sigma(v_{1}) + x_{1} \ppl_{3} d(v_{1})  + u_{1}\ppr_{3} \varphi(y_{1}) + u_{1}\ppr_{3}\sigma(v_{1})\\
& + \omega_{3}(u_{1}, d(v_{1})),\ x_{1} \trr_{3} d(v_{1})+ u_{1}\trl_{3} \varphi(y_{1})+u_{1}\trl_{3}\sigma(v_{1}) + u_{1} \ast_{3} d(v_{1}) \big).
\end{aligned}
$$
Equation (9) hold if and only if (Z109)--(Z120) hold.
Now the proof is completed.
\end{proof}

\begin{examples}
The conditions in the above Theorem \ref{thm:unifyprod} seems a little complicated. In fact, for a Zinbiel algebra, we can assume $(0, Z, 0)$ to be a trivial example of Zinbiel 2-algebra as in example \ref{exam01},  we also let $(V_{1}, V_{0}, d)$ be a 2-vector space.
The extending datum of $(0, Z_0,0)$ by $(V_{1}, V_{0}, d)$ is a system consist of the following maps:
\begin{eqnarray*}
\rightharpoonup_{0}: V_{0} \times Z_{0} \rightarrow Z_{0}, \quad\leftharpoonup_{0}: Z_{0} \times V_{0} \rightarrow Z_{0}, \quad \triangleright_{0}: Z_{0} \times V_{0} \rightarrow V_{0},
\quad \triangleleft_{0}: V_{0} \times Z_{0} \rightarrow V_{0},\\
\sigma: V_{1} \rightarrow Z_{0}, \quad
 \triangleright_{2}: Z_{0} \times V_{1} \rightarrow V_{1},
\quad \triangleleft_{3}: V_{1} \times Z_{0} \rightarrow V_{1},\quad \omega_{0}: V_{0} \times V_{0} \rightarrow Z_{0},\\
\quad *_{0}: V_{0} \times V_{0} \rightarrow V_{0},
\quad *_{1}: V_{1} \times V_{1} \rightarrow V_{1}\quad *_{2}: V_{0} \times V_{1} \rightarrow V_{1},
\quad *_{3}: V_{1} \times V_{0} \rightarrow V_{1}.
\end{eqnarray*}

As a Zinbiel 2-algebra, $(E_1,E_0,\varphi_E)$ where $E_1=0\oplus V_{1} $, $E_0=Z_{0}\oplus V_{0}$ and $\varphi_E:
0\oplus V_{1} \rightarrow Z_{0}\oplus V_{0}$ is defined as follows:
 \begin{eqnarray}
 \varphi_E(0, u_1)=(\sigma(u_1), d(u_1)),
 \end{eqnarray}
for all $u \in V_{1}$. The multiplication maps $\circ_{i}:\left(Z_{i} \oplus V_{i}\right) \times\left(Z_{i} \oplus V_{i}\right) \rightarrow Z_{i} \oplus V_{i}$ and the action maps $\trr: (Z_{0}\oplus V_{0}) \times (0\oplus V_{1})\to 0 \oplus V_{1},$ $\trl: (0 \oplus V_{1}) \times (Z_{0}\oplus V_{0})\to 0 \oplus V_{1} $ are given by
$$
\left(x_{0}, u_{0}\right) \circ_{0} \left(y_{0}, v_{0}\right)
 =\big( x_{0}\cdot y_{0} + x_{0} \ppl_{0} v_{0} + u_{0}\ppr_{0} y_{0} + \omega_{0}(u_{0}, v_{0}),\  x_{0} \trr_{0} v_{0} + u_{0}\trl_{0} y_{0} + u_{0} \ast_{0} v_{0} \big),
$$
$$
\left(0, u_{1}\right) \circ_{1} \left(0, v_{1}\right)
 =\big(0,\  u_{1} \ast_{1} v_{1} \big),
$$
and
$$
\left(x_{0}, u_{0}\right) \trr \left(0, u_{1}\right)
 = \big(0,\  x_{0} \trr_{2} u_{1} + u_{0} \ast_{2} u_{1} \big),
$$
$$
\left(0, u_{1}\right) \trl \left( x_{0}, u_{0}\right)
 = \big(0,\   u_{1}\trl_{3} x_{0} + u_{1} \ast_{3} u_{0} \big).
$$
Then one obtain a nontrivial  example of Zinbiel 2-algebra by choosing $\sigma$ to satisfying the following conditions:
\begin{enumerate}
\item[(ZZ1)] $( x_{0}\cdot y_{0} )\trr_{0} w_{0}=x_{0}\trr_{0}( y_{0}\trr_{0} w_{0}+w_{0}\trl_{0} y_{0}),$\\
           $(x_{0}\trr_{0} v_{0})\trl_{0} z_{0}=x_{0}\trr_{0}(v_{0}\trl_{0} z_{0}+ z_{0}\trr_{0} v_{0}),$\\
           $(u_{0}\trl_{0} y_{0})\trl_{0} z_{0}=u_{0}\trl_{0} (y_{0}\cdot z_{0}+ z_{0}\cdot y_{0}),$
\item[(ZZ2)]
$(x_{0}\ppl_{0} v_{0})\cdot y_{0}+(x_{0} \trr_{0} v_{0})\ppr_{0} y_{0}
=x_{0}\cdot (v_{0} \ppr_{0} y_{0}+y_{0} \ppl_{0} v_{0})+x_{0}\ppl_{0}( v_{0}\trl_{0} y_{0} + y_{0} \trr_{0} v_{0}),$
\item[(ZZ3)] $(u_{0} \ppr_{0} x_{0})\cdot y_{0} +( u_{0}\trl_{0} x_{0} )\ppr_{0} y_{0} = u_{0}\ppr_{0}( x_{0} \cdot  y_{0} + y_{0}\cdot  x_{0} ),$
\item[(ZZ4)]
$\omega_{0}(u_{0}, v_{0})\cdot x_{0}+ (u_{0} \ast_{0} v_{0} )\ppr_{0} x_{0}
= u_{0}\ppr_{0}( v_{0}\ppr_{0} x_{0} + x_{0} \ppl_{0} v_{0})+ \omega_{0}(u_{0},v_{0}\trl_{0} x_{0}+ x_{0} \trr_{0} v_{0}),$
\item[(ZZ5)] $(u_{0} \ast_{0} v_{0})\trl_{0} x_{0} = u_{0}\trl_{0}( v_{0}\ppr_{0} x_{0}+x_{0}\ppl_{0} v_{0})+u_{0} \ast_{0}( v_{0}\trl_{0} x_{0} + x_{0} \trr_{0} v_{0} ),$
\item[(ZZ6)]$( x_{0}\cdot y_{0} )\ppl_{0} w_{0}  = x_{0}\cdot(y_{0} \ppl_{0} w_{0}+w_{0}\ppr_{0} y_{0})+x_{0}\ppl_{0}(y_{0} \trr_{0} w_{0}+w_{0}\trl_{0} y_{0}),$
\item[(ZZ7)]
$( x_{0} \ppl_{0} v_{0})\ppl_{0} w_{0}+\omega_{0}(x_{0} \trr_{0} v_{0},w_{0})
=x_{0}\cdot\big(\omega_{0}(v_{0}, w_{0})+\omega_{0}(w_{0},v_{0})\big)+x_{0}\ppl_{0}(v_{0} \ast_{0} w_{0} +w_{0} \ast_{0} v_{0} ),$
\item[(ZZ8)] $(x_{0} \ppl_{0} v_{0}) \trr_{0} w_{0}+(x_{0}\trr_{0} v_{0})\ast_{0} w_{0} = x_{0}\trr_{0}( v_{0} \ast_{0} w_{0}+w_{0} \ast_{0} v_{0} ),$
\item[(ZZ9)]
$( u_{0}\ppr_{0} x_{0})\ppl_{0} w_{0}+\omega_{0}( u_{0}\trl_{0} x_{0},w_{0})
= u_{0}\ppr_{0}( x_{0} \ppl_{0} w_{0}+w_{0}\ppr_{0} x_{0} )+\omega_{0}(u_{0}, w_{0}\trl_{0} x_{0} + x_{0}\trr_{0} w_{0}),$
\item[(ZZ10)]
$(u_{0}\ppr_{0} x_{0}) \trr_{0} w_{0}+ (u_{0}\trl_{0} x_{0})\ast_{0} w_{0}
= u_{0}\trl_{0}(x_{0} \ppl_{0} w_{0} + w_{0}\ppr_{0} x_{0} )+u_{0} \ast_{0}(x_{0} \trr_{0} w_{0} + w_{0}\trl_{0} x_{0}),$
\item[(ZZ11)]
$\omega_{0}(u_{0}, v_{0})\ppl_{0} w_{0}+\omega_{0}( u_{0} \ast_{0} v_{0} ,w_{0} )
=u_{0}\ppr_{0}\big(\omega_{0}(v_{0}, w_{0})+\omega_{0}(w_{0},v_{0})\big)+ \omega_{0}(u_{0},v_{0} \ast_{0} w_{0}+ w_{0} \ast_{0} v_{0} ),$
\item[(ZZ12)]
$\omega_{0}(u_{0}, v_{0}) \trr_{0} w_{0}+(u_{0} \ast_{0} v_{0})\ast_{0} w_{0}
=u_{0}\trl_{0} \b0g(\omega_{0}(v_{0}, w_{0})+\omega_{0}(w_{0},v_{0})\b0g)+u_{0} \ast_{0}(v_{0} \ast_{0} w_{0}+w_{0} \ast_{0} v_{0}),$
\item[(ZZ13)] $( x_{0} \trr_{2} u_{1}) \ast_{1} v_{1}=x_{0} \trr_{2} (u_{1} \ast_{1} v_{1}+v_{1} \ast_{1} u_{1}),$
\item[(ZZ14)]
$(u_{0} \ast_{2} u_{1}) \ast_{1} v_{1}=u_{0} \ast_{2} (u_{1} \ast_{1} v_{1}+v_{1} \ast_{1} u_{1}) \big),$
\item[(ZZ15)]
$(x_{0}\cdot y_{0}) \trr_{2} u_{1}=x_{0} \trr_{2} (y_{0} \trr_{2} u_{1}+u_{1}\trl_{3} y_{0}),$
\item[(ZZ16)]
$(x_{0} \ppl_{0} u_{0}) \trr_{2} u_{1} + (x_{0} \trr_{0} u_{0}) \ast_{2} u_{1}=x_{0} \trr_{2} (u_{0} \ast_{2} u_{1}+u_{1} \ast_{3} u_{0}),$
\item[(ZZ17)]
$(u_{0}\ppr_{0} x_{0}) \trr_{2} u_{1}+ (u_{0}\trl_{0} x_{0}) \ast_{2} u_{1}=u_{0} \ast_{2} (x_{0} \trr_{2} u_{1}+u_{1}\trl_{3} x_{0}),$
\item[(ZZ18)]
$\omega_{0}(u_{0}, v_{0}) \trr_{2} u_{1}+ (u_{0} \ast_{0} v_{0}) \ast_{2} u_{1}=u_{0} \ast_{2} (v_{0} \ast_{2} u_{1}+u_{1} \ast_{3} v_{0}),$
\item[(ZZ19)]
$(u_{1}\trl_{3} x_{0}) \ast_{3} u_{0}=u_{1} \ast_{3} (x_{0} \trr_{0} u_{0}+u_{0}\trl_{0} x_{0}),$
\item[(ZZ20)]
$(u_{1} \ast_{3} u_{0}) \ast_{3} v_{0}=u_{1}\trl_{3} (\omega_{0}(u_{0}, v_{0}) + \omega_{0}(v_{0}, u_{0}))+ u_{1} \ast_{3} (u_{0} \ast_{0} v_{0}+v_{0} \ast_{0} u_{0}),$
\item[(ZZ21)]
$(u_{1}\trl_{3} x_{0})\trl_{3} y_{0}=u_{1}\trl_{3} (x_{0}\cdot y_{0}+y_{0}\cdot x_{0}),$
\item[(ZZ22)]
$(u_{1} \ast_{3} u_{0})\trl_{3} x_{0}=u_{1}\trl_{3} (u_{0}\ppr_{0} x_{0}+x_{0} \ppl_{0} u_{0})+ u_{1} \ast_{3} (u_{0}\trl_{0} x_{0}+x_{0} \trr_{0} u_{0}),$
\item[(ZZ23)]
$(x_{0} \trr_{2} u_{1}) \ast_{3} v_{0}=x_{0} \trr_{2} (u_{1} \ast_{3} v_{0}+v_{0} \ast_{2} u_{1}),$
\item[(ZZ24)]
$ (u_{0} \ast_{2} u_{1}) \ast_{3} v_{0}= u_{0} \ast_{2} (u_{1} \ast_{3} v_{0}+v_{0} \ast_{2} u_{1}),$
\item[(ZZ25)]
$(x_{0} \trr_{2} u_{1})\trl_{3} y_{0}=x_{0} \trr_{2} (u_{1}\trl_{3} y_{0}+y_{0} \trr_{2} u_{1}),$
\item[(ZZ26)]
$(u_{0} \ast_{2} u_{1})\trl_{3} y_{0}=u_{0} \ast_{2} (u_{1}\trl_{3} y_{0}+y_{0} \trr_{2} u_{1}),$

\item[(ZZ27)]
$(u_{1}\trl_{3} x_{0}) \ast_{1} v_{1}=u_{1} \ast_{1} (x_{0} \trr_{2} v_{1}+v_{1}\trl_{3} x_{0}),$
\item[(ZZ28)]
$(u_{1} \ast_{3} u_{0}) \ast_{1} v_{1}= u_{1} \ast_{1} (u_{0} \ast_{2} v_{1}+v_{1} \ast_{3} u_{0}),$
\item[(ZZ29)]
$(u_{1} \ast_{1} v_{1}) \ast_{3} u_{0}=u_{1} \ast_{1} (v_{1} \ast_{3} u_{0}+u_{0} \ast_{2} v_{1}),$
\item[(ZZ30)]
$(u_{1} \ast_{1} v_{1})\trl_{3} x_{0}=u_{1} \ast_{1} (v_{1}\trl_{3} x_{0}+x_{0} \trr_{2} v_{1}),$
\item[(ZZ31)]
$\sigma(x_{0} \trr_{2} u_{1})=x_{0}\cdot \sigma(u_{1}) + x_{0} \ppl_{0} d(u_{1}),$
\item[(ZZ32)]
$\sigma(u_{0} \ast_{2} u_{1})=u_{0}\ppr_{0} \sigma(u_{1}) + \omega_{0}(u_{0}, d(u_{1})),$
\item[(ZZ33)]
$d(x_{0} \trr_{2} u_{1})=x_{0} \trr_{0} d(u_{1}),$
\item[(ZZ34)]
$d(u_{0} \ast_{2} u_{1})= u_{0}\trl_{0}\sigma(u_{1}) + u_{0} \ast_{0} d(u_{1}),$
\item[(ZZ35)]$\sigma(u_{1}\trl_{3} x_{0})=\sigma(u_{1})\cdot x_{0}+ d(u_{1})\ppr_{0} x_{0},$
\item[(ZZ36)]$\sigma(u_{1} \ast_{3} u_{0})=\sigma(u_{1}) \ppl_{0} u_{0}+ \omega_{0}(d(u_{1}), u_{0}),$
\item[(ZZ37)]$d(u_{1}\trl_{3} x_{0}) =d(u_{1})\trl_{0} x_{0},$
\item[(ZZ38)]$d(u_{1} \ast_{3} u_{0}))=\sigma(u_{1}) \trr_{0} u_{0} + d(u_{1}) \ast_{0} u_{0},$
\item[(ZZ39)]$\sigma(u_{1}) \trr_{2} v_{1} + d(u_{1}) \ast_{2} v_{1}=u_{1} \ast_{1} v_{1},$
\item[(ZZ40)]$u_{1}\trl_{3}\sigma(v_{1}) + u_{1} \ast_{3} d(v_{1})=u_{1} \ast_{1} v_{1}.$
\end{enumerate}
This is the canonical way to obtain a  Zinbiel 2-algebra from a Zinbiel algebra.
\end{examples}

\begin{theorem}
 Let $(Z_{1},Z_{0}, \varphi) $ be a Zinbiel 2-algebra and $(E_{1},E_{0}, \varphi_E)$ a 2-vector space containing $(Z_{1},Z_{0}, \varphi)$ as a 2-vector subspace. Suppose that we have a Zinbiel 2-algebra structure on $(E_{1},E_{0}, \varphi_E)$ such that $(Z_{1},Z_{0}, \varphi)$ is a sub-Zinbiel 2-algebra. Then there is an isomorphism of Zinbiel 2-algebras $(E_{1},E_{0}, \varphi_E) \cong Z \natural V$ which stabilizes $(Z_{1},Z_{0}, \varphi)$ and co-stabilizes $(V_{1},V_{0},d)$.
\end{theorem}

\begin{proof} Let $p_{i}: E_{i} \rightarrow Z_{i}$ be linear maps such that $p_{i}\left(x_{i}\right)=x_{i}$ for $i=0,1$ and $x_{i} \in Z_{i} .$ Then $V_{i}:=\operatorname{Ker}\left(p_{i}\right)$ is a subspace of $E_{i}$ and a complement of $Z_{i}$ in $E_{i}$. Define the extending datum of $(Z_{1}, Z_{0}, \varphi)$ by $(V_{1}, V_{0}, d)$  by the following formulas:
$$
\begin{aligned}
\rightharpoonup_{i}: V_{i} \oplus Z_{i} \rightarrow Z_{i}, \quad u_{i} \rightharpoonup_{i} x_{i} &:=p_{i}\left(u_{i}\circ_{i} x_{i}\right), \\
\leftharpoonup_{i}: Z_{i} \oplus V_{i} \rightarrow Z_{i},\quad x_{i} \leftharpoonup_{i} u_{i} &:= p_{i}\bigl(x_{i}\circ_{i} u_{i}\bigl),\\
\triangleright_{i}: Z_{i} \oplus V_{i} \rightarrow V_{i},\quad x_{i}\triangleright_{i} u_{i} &:= x_{i}\circ_{i} u_{i} - p_{i}\bigl(x_{i}\circ_{i} u_{i}\bigl),\\
\triangleleft_{i}: V_{i} \oplus Z_{i} \rightarrow V_{i},u_{i} \triangleleft_{i} x_{i} &:= u_{i}\circ_{i} x_{i} - p_{i} \bigl(u_{i}\circ_{i} x_{i}\bigl),\\
\omega_{i}: V_{i} \oplus V_{i} \rightarrow Z_{i},\quad\omega_{i}(u_{i}, v_{i}) &:= p_{i} \bigl(u_{i}\circ_{i} v_{i}\bigl),\\
*_{i}: V_{i} \oplus V_{i} \rightarrow V_{i},\quad u_{i}\ast_{i} v_{i} &:= u_{i}\circ_{i} v_{i} - p_{i} \bigl(u_{i}\circ_{i} v_{i}\bigl)\\
\end{aligned}
$$
$$
\begin{aligned}
\rightharpoonup_{2}: V_{0} \oplus Z_{1} \rightarrow Z_{1},  \quad u_{0} \rightharpoonup_{2} x_{1} &:=p_{1}\left(u_{0} \trr_{} x_{1}\right), \\
\leftharpoonup_{2}: Z_{0} \oplus V_{1} \rightarrow Z_{1},  \quad x_{0} \leftharpoonup_{2} u_{1} &:=p_{1}\left(x_{0}\trr_{} u_{1}\right), \\
\triangleright_{2}: Z_{0} \oplus V_{1} \rightarrow V_{1},\quad x_{0} \triangleright_{2} u_{1} &:=x_{0}\trr_{} u_{1}-p_{1}\left(x_{0}\trr_{} u_{1}\right), \\
\triangleleft_{2}: V_{0} \oplus Z_{1} \rightarrow V_{1},\quad u_{0} \triangleleft_{2} x_{1} &:=u_{0}\trr_{} x_{1}-p_{1}\left(u_{0}\trr_{} x_{1}\right), \\
\omega_{2}: V_{0} \oplus V_{1} \rightarrow  Z_{1},\quad \omega_{2}\left(u_{0}, u_{1}\right) &:=p_{1}\left(u_{0}\trr_{} u_{1}\right), \\
*_{2}: V_{0} \oplus V_{1} \rightarrow V_{1},\quad u_{0} *_{2} u_{1} &:=u_{0}\trr_{} u_{1}-p_{1}\left(u_{0}\trr_{}u_{1}\right), \\
\end{aligned}
$$
$$
\begin{aligned}
\rightharpoonup_{3}: V_{1} \oplus Z_{0} \rightarrow Z_{1},\quad u_{1} \rightharpoonup_{3} x_{0} &:=-p_{1}\left(x_{0} \trl_{} u_{1}\right), \\
 \leftharpoonup_{3}: Z_{1} \oplus V_{0} \rightarrow Z_{1},\quad x_{1} \leftharpoonup_{3} u_{0} &:=-u_{0} \trl_{} x_{1}+p_{1}\left(u_{0} \trl_{} x_{1}\right), \\
 \triangleright_{3}: Z_{1} \oplus V_{0} \rightarrow V_{1},\quad x_{1} \triangleright_{3} u_{0} &:=-u_{0}\trl_{}x_{1}+p_{1}\left(u_{0}\trl_{} x_{1}\right), \\
\triangleleft_{3}: V_{1} \oplus Z_{0} \rightarrow V_{1},\quad u_{1} \triangleleft_{3} x_{0} &:=-x_{0}\trl_{}u_{1}+p_{1}\left(x_{0}\trl_{} u_{1}\right), \\
\omega_{3}: V_{1} \oplus V_{0} \rightarrow Z_{1}, \quad \omega_{3}\left(u_{1}, u_{0}\right) &:=p_{1}\left(u_{1}\trl_{} u_{0}\right), \\
 *_{3}: V_{1} \oplus V_{0} \rightarrow V_{1},\quad u_{1} *_{2} u_{0} &:=u_{1}\trl_{} u_{0}-p_{1}\left(u_{1}\trl_{}u_{0}\right), \\
\sigma: V_{1} \rightarrow Z_{0}, \quad \sigma\left(u_{1}\right) &:=p_{0}\left(\varphi_E\left(u_{1}\right)\right),
\end{aligned}
$$
for any $i=0,1, x_{i} \in Z_{i}$ and $u_{i}, v_{i} \in V_{i}$. First, the above maps are all well defined. We shall prove that $\Omega(Z, V)=\left(\leftharpoonup_{j}, \rightharpoonup_{j}, \triangleleft_{j}, \triangleright_{j}, \omega_{j}, *_{j},\sigma ; j=0,1,2,3\right)$ is a Zinbiel 2-extending structure of $(Z_{1}, Z_{0}, \varphi)$ by $(V_{1}, V_{0}, d)$  and
$$
\psi=\left(\psi_{0}, \psi_{1}\right): Z \natural V \rightarrow (E_{1},E_{0}, \varphi_E), \quad \psi_{i}\left(x_{i}, u_{i}\right):=x_{i}+u_{i}
$$
is an isomorphism of Zinbiel 2-algebras that stabilizes $(Z_{1},Z_{0}, \varphi)$ and co-stabilizes $(V_{1},V_{0}, d)$. It is easy to verify that for $i=0,1$ and $x_{i} \in E_{i}, \psi^{-1}=\left(\psi_{0}^{-1}, \psi_{1}^{-1}\right): (E_{1},E_{0}, \varphi_E) \rightarrow Z \natural V, \quad \psi_{i}^{-1}\left(x_{i}\right):=\left(p_{i}\left(x_{i}\right), x_{i}-p_{i}\left(x_{i}\right)\right)$ is an inverse of $\psi: Z \natural V \rightarrow (E_{1},E_{0}, \varphi_E)$ as 2-vector space. Therefore, there is a unique strict Zinbiel 2-algebra structure on $Z \natural V$ such that $\psi$ is an isomorphism of strict Zinbiel 2-algebras and this unique Zinbiel 2-algebra structure is given by
$$
\begin{aligned}
(x_{i}, u_{i})\circ_{i} (y_{i}, v_{i})&:=\psi_{i}^{-1}(\psi_{i}(x_{i}, u_{i})\circ_{i} \psi_{i}(y_{i}, v_{i})), \\
(x_{0}, u_{0})\trr_{}(x_{1}, u_{1})&:=\psi_{1}^{-1}\left(\psi_{0}\left(x_{0}, u_{0}\right)\trr_{} \psi_{1}\left(x_{1}, u_{1}\right)\right),\\
(x_{1}, u_{1})\trl_{}(x_{0}, u_{0})&:=\psi_{1}^{-1}\left(\psi_{1}\left(x_{1}, u_{1}\right)\trl_{} \psi_{0}\left(x_{0}, u_{0}\right)\right),
\end{aligned}
$$
for all $x_{i}, y_{i} \in Z_{i}$ and $u_{i}, v_{i} \in V_{i} .$ Then it is sufficient to prove that this Zinbiel 2-algebra structure coincides with the one defined by Definition \ref{def:01} associated with the system $\left(\leftharpoonup_{j}, \rightharpoonup_{j}, \triangleleft_{j}, \triangleright_{j}, \omega_{j}, *_{j}, \sigma ; j=0,1,2,3\right)$. Indeed, for any $x_{i}, y_{i} \in Z_{i}$ and $u_{i}, v_{i} \in V_{i}$
\begin{eqnarray*}
&&(x_{i}, u_{i})\circ_{i}(y_{i}, v_{i}) \\
&=& \psi_{i}^{-1}(\psi_{i}(x_{i}, u_{i})\circ_{i} \psi_{i}(y_{i}, v_{i}))\\
&=& \psi_{i}^{-1}(x_{i}y_{i}+x_{i}\circ_{i} v_{i}+u_{i}\circ_{i} y_{i}+u_{i}\circ_{i} v_{i}) \\
&=&( x_{i}y_{i}+p_{i}(x_{i}\circ_{i} v_{i})+p_{i}(u_{i}\circ_{i} y_{i})+p_{i}\left(u_{i}\circ_{i} v_{i}\right),\\
&&x_{i}\circ_{i} v_{i}-p_{i}(x_{i}\circ_{i} v_{i})+u_{i}\circ_{i} y_{i}-p_{i}(u_{i}\circ_{i} y_{i})+u_{i}\circ_{i} v_{i}-p_{i}\left(u_{i}\circ_{i} v_{i}\right))\\
&=&( x_{i}\cdot y_{i} + x_{i} \ppl_{i} v_{i} + u_{i}\ppr_{i} y_{i} + \omega_{i}(u_{i}, v_{i}),\  x_{i} \trr_{i} v_{i} + u_{i}\trl_{i} y_{i} + u_{i} \ast_{i} v_{i} ),
\end{eqnarray*}
\begin{eqnarray*}
&&(x_{0}, u_{0})\circ_{1}(x_{1}, u_{1})\\
&=&\psi_{1}^{-1}\left(x_{0}\trr_{} x_{1}+x_{0}\trr_{}u_{1}+u_{0}\trr_{} x_{1}+u_{0}\trr_{} u_{1}\right)\\
&=&(x_{0} x_{1}+p_{1}(x_{0}\trr_{}u_{1})+p_{1}(u_{0}\trr_{} x_{1})+p_{1}(u_{0}\trr_{} u_{1}), \\
&&x_{0}\trr_{}u_{1}-p_{1}(x_{0}\trr_{}u_{1})+u_{0}\trr_{} x_{1}-p_{1}(u_{0}\trr_{} x_{1})+u_{0}\trr_{} u_{1}-p_{1}(u_{0}\trr_{} u_{1}))\\
&=&\big( x_{0}\cdot x_{1} + x_{0} \ppl_{2} u_{1} + u_{0}\ppr_{2} x_{1} + \omega_{2}(u_{0}, u_{1}),\  x_{0} \trr_{2} u_{1} + u_{0}\trl_{2} x_{1} + u_{0} \ast_{2} u_{1} \big),
\end{eqnarray*}
\begin{eqnarray*}
&&(x_{1}, u_{1})\trl_{}(x_{0}, u_{0})\\
&=&\psi_{1}^{-1}\left(x_{1}\trl_{} x_{0}+x_{1}\trl_{}u_{0}+u_{1}\trl_{} x_{0}+u_{1}\trl_{} u_{0}\right)\\
&=&(x_{1} x_{0}+p_{1}(x_{1}\trl_{}u_{0})+p_{1}(u_{1}\trl_{} x_{0})+p_{1}(u_{1}\trl_{} u_{0}), \\
&&x_{1}\trl_{}u_{0}-p_{1}(x_{1}\trl_{}u_{0})+u_{1}\trl_{} x_{0}-p_{1}(u_{1}\trl_{} x_{0})+u_{1}\trl_{} u_{0}-p_{1}(u_{1}\trl_{} u_{0}))\\
&=&\big( x_{1}\cdot x_{0} + x_{1} \ppl_{3} u_{0} + u_{1}\ppr_{3} x_{0} + \omega_{3}(u_{1}, u_{0}),\  x_{1} \trr_{3} u_{0} + u_{1}\trl_{3} x_{0} + u_{1} \ast_{3} u_{0} \big).
\end{eqnarray*}
The proof is completed.
\end{proof}

By the above theorem, the classification of all crossed modules of Zinbiel 2-algebra structure on $(E_{1},E_{0}, \varphi_E)$ that containing $(Z_{1},Z_{0}, \varphi)$ as a sub-Zinbiel 2-algebra reduces to the classification of all unified products $(E_1,E_0,\varphi_E)$ associated with all Zinbiel 2-extending structures $\Omega(Z, V)=\left(\leftharpoonup_{j}, \rightharpoonup_{j}, \triangleleft_{j}, \triangleright_{j}, \omega_{j}, *_{j}, \sigma ; j=0,1,2,3\right)$ for a given 2-vector space $(V_{1},V_{0}, d)$ such that $V_{i}$ is a complement of $Z_{i}$ in $E_{i}$.

\begin{lemma}
Suppose that $\Omega(Z, V)=\left(\leftharpoonup_{j}, \rightharpoonup_{j}, \triangleleft_{j}, \triangleright_{j}, \omega_{j}, *_{j}, \sigma ; j=0,1,2,3\right)$ and $\Omega^{\prime}(Z, V)=\left(\leftharpoonup^{\prime}_{j}, \rightharpoonup^{\prime}_{j}, \triangleleft^{\prime}_{j}, \triangleright^{\prime}_{j}, \omega^{\prime}_{j}, *^{\prime}_{j}, \sigma^{\prime} ; j=0,1,2,3\right)$ are two Zinbiel 2-extending structures of $(Z_{1}, Z_{0}, \varphi)$ by $(V_{1}, V_{0}, d)$  and $Z\natural V, Z\natural^{\prime} V$ are the associated unified products.
Then there exists a bijection between the set of all morphisms of
Zinbiel 2-algebras $\phi: Z\natural V \rightarrow Z\natural^{\prime} V$ which stabilizes $Z$ and the set of $\left(r_{i}, s_{i} ; i=0,1\right)$, where
$r_{i}: V_{i} \rightarrow Z_{i}$ and $s_{i}: V_{i} \rightarrow V_{i}$
are linear maps satisfying the following compatibility conditions
for $i=0,1$ and any $x_{i} \in Z_{i}, u_{i}, v_{i} \in V_{i}$ :

\begin{enumerate}
\item[(H1)] $x_{i} \ppl_{i} v_{i}+r_{i}(x_{i} \trr_{i} v_{i})=x_{i}\cdot^{\prime} r_{i}(v_{i})+ x_{i}\ppl_{i}^{\prime} s_{i}(v_{i}),$
\item[(H2)] $u_{i}\ppr_{i} y_{i}+ r_{i}(u_{i}\trl_{i} y_{i})=r_{i}(u_{i})\cdot^{\prime} y_{i}+ s_{i}(u_{i})\ppr_{i}^{\prime} y_{i},$
\item[(H3)] $\omega_{i}(u_{i}, v_{i})+ r_{i}(u_{i} \ast_{i} v_{i})=r_{i}(u_{i})\cdot^{\prime} r_{i}(v_{i})+r_{i}(u_{i}) \ppl_{i}^{\prime} s_{i}(v_{i})+s_{i}(u_{i})\ppr_{i}^{\prime}r_{i}(v_{i})+ \omega_{i}^{\prime}(s_{i}(u_{i}), s_{i}(v_{i})),$
\item[(H4)]$s_{i}(x_{i} \trr_{i} v_{i})=x_{i}\trr_{i}^{\prime} s_{i}(v_{i})$,
\item[(H5)]$s_{i}(u_{i}\trl_{i} y_{i})=s_{i}(u_{i})\trl_{i}^{\prime} y_{i},$
\item[(H6)]$s_{i}(u_{i} \ast_{i} v_{i}))=r_{i}(u_{i}) \trr_{i}^{\prime} s_{i}(v_{i})+s_{i}(u_{i})\trl_{i}^{\prime}r_{i}(v_{i})+ s_{i}(u_{i}) \ast_{i}^{\prime} s_{i}(v_{i}),$
\item[(H7)] $\varphi(r_{1}(u_{1}))+\sigma^{\prime}(x_{1})+\sigma^{\prime}r_{1}(u_{1})=\sigma(u_{1})+r_{0}d(u_{1}),$
\item[(H8)]
$ds_{1}(u_{1})=s_{0}d(u_{1}),$
\item[(H9)] $x_{0} \ppl_{2} u_{1}+r_{1}(x_{0} \trr_{2} u_{1})-x_{0}\cdot^{\prime}r_{1}(u_{1})-x_{0}\ppl^{\prime}_{2} s_{1}(u_{1})=0,$
\item[(H10)] $u_{0}\ppr_{2} x_{1}+ r_{1}(u_{0}\trl_{2} x_{1})-r_{0}(u_{0})\cdot^{\prime} x_{1}-s_{0}(u_{0})\ppr^{\prime}_{2} x_{1}=0,$
 \item[(H11)]$\omega_{2}(u_{0}, u_{1})+ r_{1}(u_{0} \ast_{2} u_{1})-r_{0}(u_{0})\cdot^{\prime} r_{1}(u_{1})-r_{0}(u_{0}) \ppl^{\prime}_{2} s_{1}(u_{1})-s_{0}(u_{0})\ppr^{\prime}_{2}r_{1}(u_{1})-\omega^{\prime}_{2}(s_{0}(u_{0}), s_{1}(u_{1}))=0,$
\item[(H12)] $s_{1}(x_{0} \trr_{2} u_{1})-x_{0}\trr^{\prime}_{2} s_{1}(u_{1})=0,$
\item[(H13)] $s_{1}(u_{0}\trl_{2} x_{1})-s_{0}(u_{0})\trl^{\prime}_{2} x_{1}=0,$
\item[(H14)] $s_{1}(u_{0} \ast_{2} u_{1}))-r_{0}(u_{0}) \trr^{\prime}_{2} s_{1}(u_{1})-s_{0}(u_{0})\trl^{\prime}_{2}r_{1}(u_{1})
-s_{0}(u_{0}) \ast^{\prime}_{2} s_{1}(u_{1})=0,$
\item[(H15)] $x_{1} \ppl_{3} u_{0}+r_{1}(x_{1} \trr_{3} u_{0})=x_{1}\cdot^{\prime}r_{0}(u_{0}) + x_{1}\ppl^{\prime}_{3} s_{0}(u_{0}),$
\item[(H16)] $u_{1}\ppr_{3} x_{0} + r_{1}(u_{1}\trl_{3} x_{0})=r_{1}(u_{1})\cdot^{\prime} x_{0}+ s_{1}(u_{1})\ppr^{\prime}_{3} x_{0},$
\item[(H17)] $\omega_{3}(u_{1}, u_{0}) + r_{1}(u_{1} \ast_{3} u_{0})=r_{1}(u_{1})\cdot^{\prime}r_{0}(u_{0})+r_{1}(u_{1}) \ppl^{\prime}_{3} s_{0}(u_{0})+s_{1}(u_{1})\ppr^{\prime}_{3}r_{0}(u_{0})
 + \omega^{\prime}_{3}(s_{1}(u_{1}), s_{0}(u_{0})),$
\item[(H18)] $ s_{1}(x_{1} \trr_{3} u_{0})=x_{1}\trr^{\prime}_{3} s_{0}(u_{0}),$
\item[(H19)] $ s_{1}(u_{1}\trl_{3} x_{0})=s_{1}(u_{1})\trl^{\prime}_{3} x_{0},$
\item[(H20)] $s_{1}(u_{1} \ast_{3} u_{0}))=r_{1}(u_{1}) \trr^{\prime}_{3} s_{0}(u_{0})+s_{1}(u_{1})\trl^{\prime}_{3}r_{0}(u_{0})
 + s_{1}(u_{1}) \ast^{\prime}_{3} s_{0}(u_{0}).$
\end{enumerate}
\end{lemma}
Under the above bijection the morphism of Zinbiel 2-algebras $\phi=\phi_{\left(r_{i}, s_{i} ; i=0,1\right)}=\left(\phi_{0}, \phi_{1}\right): Z \natural V \rightarrow Z \natural^{\prime} V$ corresponding to $\left(r_{i}, s_{i} ; i=0,1\right)$ is given by:
$$
\phi_{i}\left(x_{i}, u_{i}\right)=\left(x_{i}+r_{i}\left(u_{i}\right), s_{i}\left(u_{i}\right)\right),
$$
for any $x_{i} \in Z_{i}, u_{i} \in V_{i}$ and $i=0,1 .$ Moreover, $\phi=\phi_{\left(r_{i}, s_{i} ; i=0,1\right)}$ is an isomorphism if and only if $s_{i}: V_{i} \rightarrow V_{i}$ is an isomorphism and ${\phi}={\phi}_{\left(r_{i}, s_{i} ; i=0,1\right)}$ co-stabilizes $(V_{1},V_{0}, d)$ if and only if $s_{i}=\id_{V_{i}}$.

\begin{proof}
Suppose that $\phi=\left(\phi_{0}, \phi_{1}\right): Z \natural V \rightarrow Z \natural^{\prime} V$ is a linear functor.
Then it is uniquely determined by linear map $r_{i}: V_{i} \rightarrow Z_{i}$ and $s_{i}: V_{i} \rightarrow V_{i}$ such that $\phi_{i}\left(x_{i}, u_{i}\right)=\left(x_{i}+r_{i}\left(u_{i}\right), s_{i}\left(u_{i}\right)\right)$ for $i=0,1$ and all $x_{i} \in Z_{i}, v_{i} \in V_{i}.$ In fact, let $\phi_{i}\left(0, u_{i}\right)=\left(r_{i}\left(u_{i}\right), s_{i}\left(u_{i}\right)\right) \in Z_{i} \oplus V_{i}$ for all $u_{i} \in V_{i}$. Then $\phi_{i}\left(x_{i}, u_{i}\right)=(x_{i}+r_{i}(u_{i}), s_{i}(u_{i})) \in Z_{i} \oplus V_{i}$. Next, the paper proves that $\phi=\phi_{\left(r_{i}, s_{i} i=0,1\right)}$ is a morphism of crossed modules of Zinbiel 2-algebras if and only if the compatibility conditions $(H1)-(H12)$ hold. It is sufficient to prove the equations
\begin{eqnarray}
&&\phi_{i}((x_{i}, u_{i})\circ_{i}(y_{i}, v_{i})) =\phi_{i}\left(x_{i}, u_{i}\right)\circ_{i} \phi_{i}\left(y_{i}, v_{i}\right),  i=0,1,\\
&&\left(\varphi+\sigma^{\prime}+d\right) \phi_{1}\left(x_{1}, u_{1}\right)=\phi_{0}\left(\varphi\left(x_{1}\right)+\sigma\left(u_{1}\right), d\left(u_{1}\right)\right), \\
&&\phi_{1}((x_{0}, u_{0})\trr_{}(x_{1}, u_{1})) =\phi_{0}(x_{0},u_{0})\circ^{\prime}_{2} \phi_{1}(x_{1},u_{1}),\\
&&\phi_{1}((x_{1}, u_{1})\trl_{}(x_{0}, u_{0})) =\phi_{1}(x_{1},u_{1})\circ^{\prime}_{3} \phi_{0}(x_{0},u_{0}).
\end{eqnarray}

First, consider Equation $(1)$, by direct computations,
\begin{eqnarray*}
&&\phi_{i}((x_{i}, u_{i})\circ_{i}(y_{i}, v_{i})) \\
&=&\phi_{i}\big( x_{i}\cdot y_{i} + x_{i} \ppl_{i} v_{i} + u_{i}\ppr_{i} y_{i} + \omega_{i}(u_{i}, v_{i}),\  x_{i} \trr_{i} v_{i} + u_{i}\trl_{i} y_{i} + u_{i} \ast_{i} v_{i} \big)\\
&=&(x_{i}\cdot y_{i} + x_{i} \ppl_{i} v_{i} + u_{i}\ppr_{i} y_{i} + \omega_{i}(u_{i}, v_{i})
+r_{i}(x_{i} \trr_{i} v_{i}) + r_{i}(u_{i}\trl_{i} y_{i}) + r_{i}(u_{i} \ast_{i} v_{i}),\\
&& s_{i}(x_{i} \trr_{i} v_{i}) + s_{i}(u_{i}\trl_{i} y_{i}) + s_{i}(u_{i} \ast_{i} v_{i})),
\end{eqnarray*}

\begin{eqnarray*}
&&\big(x_{i}\cdot y_{i} + x_{i}\cdot^{\prime} r_{i}(v_{i})+ r_{i}(u_{i})\cdot^{\prime} y_{i}+r_{i}(u_{i})\cdot^{\prime} r_{i}(v_{i})
+ x_{i}\ppl_{i}^{\prime} s_{i}(v_{i})+r_{i}(u_{i}) \ppl_{i}^{\prime} s_{i}(v_{i})
\\
&&+ s_{i}(u_{i})\ppr_{i}^{\prime} y_{i}+s_{i}(u_{i})\ppr_{i}^{\prime}r_{i}(v_{i})+ \omega_{i}^{\prime}(s_{i}(u_{i}), s_{i}(v_{i})),\ x_{i}\trr_{i}^{\prime} s_{i}(v_{i})+r_{i}(u_{i}) \trr_{i}^{\prime} s_{i}(v_{i})\\
&&+ s_{i}(u_{i})\trl_{i}^{\prime} y_{i}+s_{i}(u_{i})\trl_{i}^{\prime}r_{i}(v_{i})
+ s_{i}(u_{i}) \ast_{i}^{\prime} s_{i}(v_{i}) \big),
\end{eqnarray*}
Equation $(1)$ holds if and only if conditions (H1)--(H6) hold.

\begin{eqnarray*}
&&(\varphi+\sigma^{\prime}+d) \phi_{1}(x_{1}, u_{1})-\phi_{0}(\varphi(x_{1})+\sigma(u_{1}), d(u_{1})) \\
&=&(\varphi+\sigma^{\prime}+d)(x_{1}+r_{1}(u_{1}), s_{1}(u_{1}))-(\varphi(x_{1})+\sigma(u_{1})+r_{0}(d(u_{1})), s_{0}(d(u_{1})))\\
&=&(\varphi(r_{1}(u_{1}))+\sigma^{\prime}(x_{1})+\sigma^{\prime}r_{1}(u_{1})
-\sigma(u_{1})-r_{0}d(u_{1}),ds_{1}(u_{1})-s_{0}d(u_{1})),
\end{eqnarray*}
Equation (2) holds if and only if (H7)--(H8) hold.

Now consider Equation (3),
\begin{eqnarray*}
&&\phi_{1}((x_{0}, u_{0})\trr_{}(x_{1}, u_{1})) -\phi_{0}(x_{0},u_{0})\circ^{\prime}_{2} \phi_{1}(x_{1},u_{1})\\
&=&\phi_{1}\big( x_{0}\cdot x_{1} + x_{0} \ppl_{2} u_{1} + u_{0}\ppr_{2} x_{1} + \omega_{2}(u_{0}, u_{1}),\  x_{0} \trr_{2} u_{1} + u_{0}\trl_{2} x_{1} + u_{0} \ast_{2} u_{1} \big)\\
&&-(x_{0}+r_{0}(u_{0}), s_{0}(u_{0}))\circ^{\prime}_{2}(x_{1}+r_{1}(u_{1}), s_{1}(u_{1}))\\
 &=&(x_{0}\cdot x_{1} + x_{0} \ppl_{2} u_{1} + u_{0}\ppr_{2} x_{1} + \omega_{2}(u_{0}, u_{1})+r_{1}(x_{0} \trr_{2} u_{1}) + r_{1}(u_{0}\trl_{2} x_{1})\\
 &&+ r_{1}(u_{0} \ast_{2} u_{1}),s_{1}(x_{0} \trr_{2} u_{1}) + s_{1}(u_{0}\trl_{2} x_{1}) + s_{1}(u_{0} \ast_{2} u_{1}))\\
&& -\big(x_{0}\cdot^{\prime}x_{1}+ x_{0}\cdot^{\prime}r_{1}(u_{1})+r_{0}(u_{0})\cdot^{\prime} x_{1}+r_{0}(u_{0})\cdot^{\prime} r_{1}(u_{1})
 + x_{0}\ppl^{\prime}_{2} s_{1}(u_{1}) \\
&&+r_{0}(u_{0}) \ppl^{\prime}_{2} s_{1}(u_{1}) + s_{0}(u_{0})\ppr^{\prime}_{2} x_{1}+s_{0}(u_{0})\ppr^{\prime}_{2}r_{1}(u_{1})
 + \omega^{\prime}_{2}(s_{0}(u_{0}), s_{1}(u_{1})),\ \\
&& x_{0}\trr^{\prime}_{2} s_{1}(u_{1})+r_{0}(u_{0}) \trr^{\prime}_{2} s_{1}(u_{1})
 + s_{0}(u_{0})\trl^{\prime}_{2} x_{1}+s_{0}(u_{0})\trl^{\prime}_{2}r_{1}(u_{1})
 + s_{0}(u_{0}) \ast^{\prime}_{2} s_{1}(u_{1}) \big),
\end{eqnarray*}
Equation (3) holds if and only if (H9)--(H14) hold.

 \begin{eqnarray*}
&&\phi_{1}((x_{1}, u_{1})\trl_{}(x_{0}, u_{0})) - \phi_{1}(x_{1},u_{1})\circ^{\prime}_{3} \phi_{0}(x_{0},u_{0})\\
&=&\phi_{1}\big( x_{1}\cdot x_{0} + x_{1} \ppl_{3} u_{0} + u_{1}\ppr_{3} x_{0} + \omega_{3}(u_{1}, u_{0}),\  x_{1} \trr_{3} u_{0} + u_{1}\trl_{3} x_{0} + u_{1} \ast_{3} u_{0} \big)\\
&&-(x_{1}+r_{1}(u_{1}), s_{1}(u_{1})) \circ^{\prime}_{3}(x_{0}+r_{0}(u_{0}), s_{0}(u_{0}))\\
&=&( x_{1}\cdot x_{0} + x_{1} \ppl_{3} u_{0} + u_{1}\ppr_{3} x_{0} + \omega_{3}(u_{1}, u_{0})+r_{1}(x_{1} \trr_{3} u_{0}) + r_{1}(u_{1}\trl_{3} x_{0})\\
 && + r_{1}(u_{1} \ast_{3} u_{0}),s_{1}(x_{1} \trr_{3} u_{0}) + s_{1}(u_{1}\trl_{3} x_{0}) + s_{1}(u_{1} \ast_{3} u_{0}))\\
 &&-\big(x_{1}\cdot^{\prime}x_{0}+ x_{1}\cdot^{\prime}r_{0}(u_{0})+r_{1}(u_{1})\cdot^{\prime} x_{0}+r_{1}(u_{1})\cdot^{\prime}r_{0}(u_{0})
 + x_{1}\ppl^{\prime}_{3} s_{0}(u_{0}) \\
 &&+r_{1}(u_{1}) \ppl^{\prime}_{3} s_{0}(u_{0})+ s_{1}(u_{1})\ppr^{\prime}_{3} x_{0}+s_{1}(u_{1})\ppr^{\prime}_{3}r_{0}(u_{0})
 + \omega^{\prime}_{3}(s_{1}(u_{1}), s_{0}(u_{0})),\ \\
 &&x_{1}\trr^{\prime}_{3} s_{0}(u_{0})+r_{1}(u_{1}) \trr^{\prime}_{3} s_{0}(u_{0})
 + s_{1}(u_{1})\trl^{\prime}_{3} x_{0}+s_{1}(u_{1})\trl^{\prime}_{3}r_{0}(u_{0})
 + s_{1}(u_{1}) \ast^{\prime}_{3} s_{0}(u_{0}) \big).
\end{eqnarray*}
Equation (4) holds if and only if (H15)--(H20) hold.

Assume that $s_{i}: V_{i} \rightarrow V_{i}$ is bijective. Then ${\phi}_{\left(r_{i}, s_{i} ; i=0,1\right)}$ is an isomorphism of Zinbiel 2-algebras with the inverse given by ${\phi}_{\left(r_{i}, s_{i} i=0,1\right)}^{-1}=\left({\phi}_{0}^{-1}, {\phi}_{1}^{-1}\right)$, where ${\phi}_{i}^{-1}\left(x_{i}, v_{i}\right)=\left(x_{i}-r_{i}\left(s^{-1}\left(v_{i}\right)\right), s^{-1}\left(v_{i}\right)\right)$ for $x_{i} \in Z_{i}$ and $v_{i} \in V_{i}$. Conversely, assume that ${\phi}_{\left(r_{i}, s_{i} i=0,1\right)}=\left({\phi}_{0}, {\phi}_{1}\right)$ is isomorphic. Then ${\phi}_{i}$ is an isomorphism of Zinbiel algebras for $i=0,1$. The proof is completed now.
\end{proof}

We denote by ${\mathcal Z} ({Z}, V)$ the set of extending structures $\Omega({Z}, V)$.
It is easy to see that  $\equiv$  and $\approx$ are equivalence relations on the set   ${\mathcal Z} ({Z}, V)$.
We obtain the following theorem, which provides an answer to the
extending structures problem of strict Zinbiel 2-algebras.

\begin{theorem}\thlabel{main1}
Let $(Z_{1},Z_{0}, \varphi)$ be a Zinbiel 2-algebra, $(E_{1},E_{0}, \varphi_E)$ a 2-vector space that
contains $(Z_{1},Z_{0}, \varphi)$ as a subspace and $(V_{1},V_{0}, d)$ a complement of
$(Z_{1},Z_{0}, \varphi)$ in $(E_{1},E_{0}, \varphi_E)$. Then we get:

$(1)$ Denote ${\mathcal
H}{\mathcal E}^{2}_{} \, (V, \, {Z} ) :=
{\mathcal Z}  ({Z}, V)/ \equiv $, then the
map
$$
{\mathcal H}{\mathcal E}^{2}_{} \, (V, \, {Z}) \to {\rm Extd} \, (E, {Z}), \qquad
\overline{(\leftharpoonup_{j}, \rightharpoonup_{j}, \triangleleft_{j}, \triangleright_{j}, \omega_{j}, *_{j}; j=0,1,2,3)} \mapsto {Z}\natural V
$$
is bijective, where $\overline{(\trl, \trr,
\ppl, \, \ppr, \omega,\, \ast)}$ is the
equivalence class of $(\trl, \trr,
\ppl, \, \ppr, \omega,\, \ast)$ via $\equiv$.

$(2)$ Denote ${\mathcal
H}{\mathcal C}^{2} \, (V, \, {Z} ) := {\mathcal Z} ({Z}, V)/ \approx $, then the map
$$
{\mathcal H}{\mathcal C}^{2} \, (V, \, {Z} ) \to {\rm Extd}' \, (E, {Z}), \qquad
\overline{\overline{(\leftharpoonup_{j}, \rightharpoonup_{j}, \triangleleft_{j}, \triangleright_{j}, \omega_{j}, *_{j}; j=0,1,2,3)}} \mapsto{Z} \natural V
$$
is bijective, where $\overline{\overline{(\trl,
\trr, \ppl, \, \ppr, \omega,\, \ast)}}$ is the equivalence class of $(\trl,
\trr, \ppl, \, \ppr,  \omega,\, \ast)$ via $\approx$.
\end{theorem}

\section{Nonabelian extension and matched pair of the Zinbiel 2-algebras}\selabel{cazurispeciale}

In this section, we show some special cases of unified products and extending structures.

\subsection{Crossed products and nonabelian extension problem}

Now we give a first special case of the unified product, namely the crossed product of Zinbiel 2-algebras
which is related to the study of the extension problem.

Let $(Z_1, Z_0,\varphi)$ and $(V_1, V_0,d)$ be two given Zinbiel 2-algebras. The extension problem asks for the classification of all
extensions of $(V_1, V_0,d)$ by $(Z_{1},Z_{0}, \varphi)$, i.e. all
Zinbiel 2-algebras in $(E_1, E_0,\varphi_E)$ that fit into an exact sequence
\begin{eqnarray} \label{diag03}
\xymatrix {0 \ar[r]^{}  & {Z_1} \ar[r]^{i_1} \ar[d]_{\varphi_Z} & {E_1}  \ar[r]^{\pi_1}\ar[d]^{\varphi_E}
& {V_1}\ar[d]_{}\ar[r]^{} \ar[d]^{\varphi_V} & 0 \\
0 \ar[r]^{}& {Z_0} \ar[r]^{i_0} & {E_0}\ar[r]^{\pi_0} & V_0  \ar[r]^{} & 0}
\end{eqnarray}
The classification is up to an isomorphism of Zinbiel 2-algebras
that stabilizes $(Z_1, Z_0,\varphi)$ and co-stabilizes $(V_1, V_0,d)$
and we denote by ${\mathcal E} {\mathcal P} ({V}, \,
{Z})$ the isomorphism classes of all extensions of
$(V_1, V_0,d)$ by $(Z_1, Z_0,\varphi)$ up to this equivalence relation.

\begin{definition}\label{cpLei}
A system $\Omega({Z}, V)=\left(\leftharpoonup_{j}, \rightharpoonup_{j}, \omega_{j}, \sigma ; j=0,1,2,3\right)$ consisting of the following bilinear maps
 \begin{eqnarray*}
\rightharpoonup_{0}: V_{0} \times  Z_{0} \rightarrow Z_{0}, \quad\leftharpoonup_{0}: Z_{0} \times  V_{0} \rightarrow Z_{0},
\quad\rightharpoonup_{1}: V_{1} \times  Z_{1} \rightarrow Z_{1}, \quad\leftharpoonup_{0}: Z_{1} \times  V_{1} \rightarrow Z_{1}, \\
\quad\rightharpoonup_{2}: V_{0} \times  Z_{1} \rightarrow Z_{1},\quad \leftharpoonup_{2}: Z_{0} \times  V_{1} \rightarrow Z_{1},
\quad\rightharpoonup_{3}: V_{1} \times  Z_{0} \rightarrow Z_{1},\quad \leftharpoonup_{3}: Z_{1} \times  V_{0} \rightarrow Z_{1},\\
\quad\omega_{0}: V_{0} \times  V_{0} \rightarrow Z_{0},\quad \omega_{2}: V_{0} \times  V_{1} \rightarrow  Z_{1},
\quad *_{0}: V_{0} \times  V_{0} \rightarrow V_{0},\quad *_{2}: V_{0} \times  V_{1} \rightarrow V_{1},\\
\quad \omega_{1}: V_{1} \times  V_{1} \rightarrow Z_{1},\quad \omega_{3}: V_{0} \times  V_{0} \rightarrow Z_{1},
\quad *_{1}: V_{1} \times  V_{1} \rightarrow V_{1},\quad *_{3}: V_{1} \times  V_{0} \rightarrow V_{1}
\end{eqnarray*}
and one linear map $\sigma: V_{1} \rightarrow Z_{0}$ will be called a crossed system of Zinbiel 2-algebras. The crossed product associated with the crossed system $\left(Z, V, \leftharpoonup_{j}, \rightharpoonup_{j}, \omega_{j}, \sigma ; j=0,1,2,3\right)$ is the Zinbiel 2-algebra  with the multiplications given by:
$$
\left(x_{i}, u_{i}\right) \circ_{i} \left(y_{i}, v_{i}\right)
 =\big( x_{i}\cdot y_{i} + x_{i} \ppl_{i} v_{i} + u_{i}\ppr_{i} y_{i} + \omega_{i}(u_{i}, v_{i}),\  u_{i} \ast_{i} v_{i} \big), i = 0,1,
$$
and
$$
\left(x_{0}, u_{0}\right) \trr_{} \left( x_{1}, u_{1}\right)
 = \big( x_{0}\cdot x_{1} + x_{0} \ppl_{2} u_{1} + u_{0}\ppr_{2} x_{1} + \omega_{2}(u_{0}, u_{1}),\ u_{0} \ast_{2} u_{1} \big),
$$
$$
\left(x_{1}, u_{1}\right) \trl_{} \left( x_{0}, u_{o}\right)
 = \big( x_{1}\cdot x_{0} + x_{1} \ppl_{3} u_{0} + u_{1}\ppr_{3} x_{0} + \omega_{3}(u_{1}, u_{0}),\ u_{1} \ast_{3} u_{0} \big),
$$
for all $x_{i}, y_{i} \in Z_{i}.$ Then $Z \cong Z \times\{0\}$ is an ideal of the Zinbiel 2-algebra $Z\#_{ \leftharpoonup,\rightharpoonup}^{\omega} V$ since $\left(x_{i}, 0\right)\circ_{i}\left(y_{i}, u_{i}\right):=(x_{i}\cdot y_{i} + x_{i} \ppl_{i} u_{i},\  0), (x_{0}, 0)\trr_{}\left(x_{1}, u_{1}\right):=( x_{0}\cdot x_{1} + x_{0} \ppl_{2} u_{1},\ 0 )$, and $\left(x_{0}, u_{0}\right)\trl_{}\left(x_{1}, 0\right):= ( x_{1}\cdot x_{0} + x_{1}\ppl_{3}u_{0},\ 0 \big).$ This type of Zinbiel 2-algebra $(Z_1, Z_0,\varphi)$ is called the cross product of $(Z_1, Z_0,\varphi)$ and $(V_1, V_0,d)$. We will denote it by $Z \#_{}^{\omega} V$.
\end{definition}

\begin{theorem}
The crossed system $\Omega({Z}, V)=\left(\leftharpoonup_{j}, \rightharpoonup_{j}, \omega_{j}, \sigma ; j=0,1,2,3\right)$ is a Zinbiel 2-algebra if and only if $\left(V, d, *_{j} ; j=0,1,2,3\right)$ is a strict Zinbiel 2-algebra and the following compatibilities hold for $i=0,1$ and any $x_{i}, y_{i} \in Z_{i}, u_{i}, v_{i} \in V_{i}$ :
\begin{enumerate}
\item[(CZ1)]
$(x_{i}\ppl_{i} v_{i})\cdot y_{i}=x_{i}\cdot (v_{i} \ppr_{i} y_{i}+y_{i} \ppl_{i} v_{i}),$
\item[(CZ2)] $(u_{i} \ppr_{i} x_{i})\cdot y_{i}= u_{i}\ppr_{i}( x_{i} \cdot  y_{i} + y_{i}\cdot  x_{i} ),$
\item[(CZ3)]
$\omega_{i}(u_{i}, v_{i})\cdot x_{i}+ (u_{i} \ast_{i} v_{i} )\ppr_{i} x_{i}
= u_{i}\ppr_{i}( v_{i}\ppr_{i} x_{i} + x_{i} \ppl_{i} v_{i}),$
\item[(CZ4)]$( x_{i}\cdot y_{i} )\ppl_{i} w_{i}  = x_{i}\cdot(y_{i} \ppl_{i} w_{i}+w_{i}\ppr_{i} y_{i}),$
\item[(CZ5)]
$( x_{i} \ppl_{i} v_{i})\ppl_{i} w_{i}
=x_{i}\cdot\big(\omega_{i}(v_{i}, w_{i})+\omega_{i}(w_{i},v_{i})\big)+x_{i}\ppl_{i}(v_{i} \ast_{i} w_{i} +w_{i} \ast_{i} v_{i} ),$
\item[(CZ6)]
$( u_{i}\ppr_{i} x_{i})\ppl_{i} w_{i}
= u_{i}\ppr_{i}( x_{i} \ppl_{i} w_{i}+w_{i}\ppr_{i} x_{i} ),$
\item[(CZ7)]
$\omega_{i}(u_{i}, v_{i})\ppl_{i} w_{i}+\omega_{i}( u_{i} \ast_{i} v_{i} ,w_{i} )
=u_{i}\ppr_{i}\big(\omega_{i}(v_{i}, w_{i})+\omega_{i}(w_{i},v_{i})\big)+ \omega_{i}(u_{i},v_{i} \ast_{i} w_{i}+ w_{i} \ast_{i} v_{i} ),$
\item[(CZ8)]
$(x_{0}\cdot x_{1}) \ppl_{1} u_{1}=x_{0}\cdot (x_{1} \ppl_{1}u_{1}+u_{1}\ppr_{1} x_{1}),$
\item[(CZ9)]
$(x_{0} \ppl_{2} u_{1})\cdot x_{1} =x_{0}\cdot (u_{1}\ppr_{1}x_{1}+x_{1} \ppl_{1}u_{1}),$
\item[(CZ10)]
$(x_{0} \ppl_{2} u_{1}) \ppl_{1} v_{1}=x_{0}\cdot (\omega_{1}(u_{1}, v_{1})+\omega_{1}(v_{1}, u_{1}))+x_{0} \ppl_{2} (u_{1} \ast_{1} v_{1}+v_{1} \ast_{1} u_{1}),$
\item[(CZ11)] $(u_{0}\ppr_{2} x_{1})\cdot y_{1}=u_{0}\ppr_{2} ( x_{1}\cdot y_{1}+y_{1}\cdot x_{1}),$
\item[(CZ12)]
$(u_{0}\ppr_{2} x_{1}) \ppl_{1}u_{1}=u_{0}\ppr_{2} (x_{1} \ppl_{1}u_{1}+u_{1}\ppr_{1} x_{1}),$
\item[(CZ13)]
$\omega_{2}(u_{0}, u_{1})\cdot x_{1} + (u_{0} \ast_{2} u_{1})\ppr_{1} x_{1}=u_{0}\ppr_{2} (u_{1}\ppr_{1} x_{1}+x_{1} \ppl_{1} u_{1}),$
\item[(CZ14)]
$\omega_{2}(u_{0}, u_{1}) \ppl_{1} v_{1} + \omega_{1}(u_{0} \ast_{2} u_{1}, v_{1})=u_{0}\ppr_{2} (\omega_{1}(u_{1}, v_{1})+\omega_{1}(v_{1}, u_{1}))+ \omega_{2}(u_{0},u_{1} \ast_{1} v_{1}+ v_{1} \ast_{1} u_{1}),$
\item[(CZ15)]
$(x_{0} \ppl_{0} u_{0})\cdot x_{1}=x_{0}\cdot (u_{0}\ppr_{2} x_{1}+x_{1} \ppl_{3} u_{0}),$
\item[(CZ16)]
$(u_{0}\ppr_{0} x_{0})\cdot x_{1}=u_{0}\ppr_{2} (x_{0}\cdot x_{1}+x_{1}\cdot x_{0}),$
\item[(CZ17)]
$\omega_{0}(u_{0}, v_{0})\cdot x_{1}+ (u_{0} \ast_{0} v_{0})\ppr_{2} x_{1}= u_{0}\ppr_{2} (v_{0}\ppr_{2} x_{1}+x_{1} \ppl_{3} v_{0}),$
\item[(CZ18)]
$(x_{0}\cdot y_{0}) \ppl_{2} u_{1}=x_{0}\cdot (y_{0} \ppl_{2} u_{1}+u_{1}\ppr_{3} y_{0}),$
\item[(CZ19)]
$(x_{0} \ppl_{0} u_{0}) \ppl_{2} u_{1}=x_{0}\cdot (\omega_{2}(u_{0}, u_{1})+ \omega_{3}(u_{1}, u_{0})) + x_{0} \ppl_{2} (u_{0} \ast_{2} u_{1}+u_{1} \ast_{3} u_{0}),$
\item[(CZ20)]
$(u_{0}\ppr_{0} x_{0}) \ppl_{2} u_{1}= u_{0}\ppr_{2} (x_{0} \ppl_{2} u_{1}+u_{1}\ppr_{3} x_{0}),$
\item[(CZ21)]
$\omega_{0}(u_{0}, v_{0}) \ppl_{2} u_{1}+ \omega_{2}(u_{0} \ast_{0} v_{0}, u_{1})= u_{0}\ppr_{2} (\omega_{2}(v_{0}, u_{1}) + \omega_{3}(u_{1}, v_{0}))+ \omega_{2}(u_{0},v_{0} \ast_{2} u_{1}+u_{1} \ast_{3} v_{0}),$
\item[(CZ22)]
$(x_{1} \ppl_{3} u_{0})\cdot x_{0}=x_{1}\cdot (u_{0}\ppr_{0} x_{0}+x_{0} \ppl_{0} u_{0}) ,$
\item[(CZ23)]
$(u_{1}\ppr_{3} x_{0})\cdot y_{0} =u_{1}\ppr_{3} (x_{0}\cdot y_{0}+y_{0}\cdot x_{0}),$
\item[(CZ24)]
$\omega_{3}(u_{1}, u_{0})\cdot x_{0} + (u_{1} \ast_{3} u_{0})\ppr_{3} x_{0}=u_{1}\ppr_{3} (u_{0}\ppr_{0} x_{0}+x_{0} \ppl_{0} u_{0}),$
\item[(CZ25)]
$(x_{1}\cdot x_{0}) \ppl_{3} u_{0}=x_{1}\cdot (x_{0} \ppl_{0} u_{0}+u_{0}\ppr_{0} x_{0}),$
\item[(CZ26)]
$(x_{1} \ppl_{3} u_{0}) \ppl_{3} v_{0}=x_{1}\cdot (\omega_{0}(u_{0}, v_{0})+\omega_{0}(v_{0}, u_{0}))+ x_{1} \ppl_{3} (u_{0} \ast_{0} v_{0}+v_{0} \ast_{0} u_{0}),$
\item[(CZ27)]
$(u_{1}\ppr_{3} x_{0}) \ppl_{3} u_{0} =u_{1}\ppr_{3} (x_{0} \ppl_{0} u_{0}+u_{0}\ppr_{0} x_{0}),$
\item[(CZ28)]
$\omega_{3}(u_{1}, u_{0}) \ppl_{3} v_{0} + \omega_{3}(u_{1} \ast_{3} u_{0}, v_{0})=u_{1}\ppr_{3} (\omega_{0}(u_{0}, v_{0})+ \omega_{0}(v_{0}, u_{0})) + \omega_{3}(u_{1},u_{0} \ast_{0} v_{0}+v_{0} \ast_{0} u_{0}),$
\item[(CZ29)]
$(x_{0} \ppl_{2} u_{1})\cdot y_{0}=x_{0}\cdot (u_{1}\ppr_{3} y_{0}+y_{0} \ppl_{2} u_{1}),$
\item[(CZ30)]
$(u_{0}\ppr_{2} x_{1})\cdot y_{0}= u_{0}\ppr_{2} (x_{1}\cdot y_{0}+y_{0}\cdot x_{1}),$
\item[(CZ31)]
$\omega_{2}(u_{0}, u_{1})\cdot x_{0}+ (u_{0} \ast_{2} u_{1})\ppr_{3} x_{0}= u_{0}\ppr_{2} (u_{1}\ppr_{3} x_{0}+x_{0} \ppl_{2} u_{1}),$
\item[(CZ32)]
$(x_{0}\cdot x_{1}) \ppl_{3} u_{0}=x_{0}\cdot (x_{1} \ppl_{3} u_{0}+u_{0}\ppr_{2} x_{1}),$
\item[(CZ33)]
$(x_{0} \ppl_{2} u_{1}) \ppl_{3} v_{0}=x_{0}\cdot (\omega_{3}(u_{1}, v_{0})+ \omega_{2}(v_{0}, u_{1}))+ x_{0} \ppl_{2} (u_{1} \ast_{3} v_{0}+v_{0} \ast_{2} u_{1}),$
\item[(CZ34)]
$(u_{0}\ppr_{2} x_{1}) \ppl_{3} v_{0}= u_{0}\ppr_{2} (x_{1} \ppl_{3} v_{0}+v_{0}\ppr_{2} x_{1}) ,$
\item[(CZ35)]
$\omega_{2}(u_{0}, u_{1}) \ppl_{3} v_{0}+ \omega_{3}(u_{0} \ast_{2} u_{1}, v_{0})=u_{0}\ppr_{2} (\omega_{3}(u_{1}, v_{0})+\omega_{2}(v_{0}, u_{1}))+ \omega_{2}(u_{0}, u_{1} \ast_{3} v_{0}+ v_{0} \ast_{2} u_{1}),$
\item[(CZ36)]
$(x_{1} \ppl_{3} u_{0})\cdot y_{1}=x_{1}\cdot (u_{0}\ppr_{2} y_{1}+y_{1} \ppl_{3} u_{0}),$
\item[(CZ37)]
$(u_{1}\ppr_{3} x_{0})\cdot x_{1}=u_{1}\ppr_{1} (x_{0}\cdot x_{1}+x_{1}\cdot x_{0}),$
\item[(CZ38)]
$\omega_{3}(u_{1}, u_{0})\cdot x_{1}+ (u_{1} \ast_{3} u_{0})\ppr_{1} x_{1}=u_{1}\ppr_{1} (u_{0}\ppr_{2} x_{1}+x_{1} \ppl_{3} u_{0}),$
\item[(CZ39)]
$(x_{1}\cdot x_{0}) \ppl_{1} u_{1}=x_{1}\cdot (x_{0} \ppl_{2} u_{1}+u_{1}\ppr_{3} x_{0}),$
\item[(CZ40)]
$(x_{1} \ppl_{3} u_{0}) \ppl_{1} u_{1}=x_{1}\cdot (\omega_{2}(u_{0}, u_{1})+ \omega_{3}(u_{1}, u_{0})) + x_{1} \ppl_{1} (u_{0} \ast_{2} u_{1}+u_{1} \ast_{3} u_{0}) ,$
\item[(CZ41)]
$(u_{1}\ppr_{3} x_{0}) \ppl_{1} v_{1}=u_{1}\ppr_{1} (x_{0} \ppl_{2} v_{1}+v_{1}\ppr_{3} x_{0}) ,$
\item[(CZ42)]
$\omega_{3}(u_{1}, u_{0}) \ppl_{1} v_{1}+ \omega_{1}(u_{1} \ast_{3} u_{0}, v_{1})=u_{1}\ppr_{1} (\omega_{2}(u_{0}, v_{1}) + \omega_{3}(v_{1}, u_{0})) + \omega_{1}(u_{1}, u_{0} \ast_{2} v_{1}+v_{1} \ast_{3} u_{0}),$
\item[(CZ43)]
$(x_{1} \ppl_{1} u_{1})\cdot x_{0}=x_{1}\cdot (u_{1}\ppr_{3} x_{0}+x_{0} \ppl_{2} u_{1}),$
\item[(CZ44)]
$(u_{1}\ppr_{1} x_{1})\cdot x_{0} = u_{1}\ppr_{1} (x_{1}\cdot x_{0}+x_{0}\cdot x_{1}),$
\item[(CZ45)]
$\omega_{1}(u_{1}, v_{1})\cdot x_{0} + (u_{1} \ast_{1} v_{1})\ppr_{3} x_{0}=u_{1}\ppr_{1} (v_{1}\ppr_{3} x_{0}+x_{0} \ppl_{2} v_{1}),$
\item[(CZ46)]
$(x_{1}\cdot y_{1}) \ppl_{3} u_{0}=x_{1}\cdot (y_{1} \ppl_{3} u_{0}+u_{0}\ppr_{2} y_{1}),$
\item[(CZ47)]
$(x_{1} \ppl_{1} u_{1}) \ppl_{3} u_{0}=x_{1}\cdot (\omega_{3}(u_{1}, u_{0})+\omega_{2}(u_{0}, u_{1}))+ x_{1} \ppl_{1} (u_{1} \ast_{3} u_{0}+u_{0} \ast_{2} u_{1}),$
\item[(CZ48)]
$(u_{1}\ppr_{1} x_{1}) \ppl_{3} u_{0} = u_{1}\ppr_{1} (x_{1} \ppl_{3} u_{0}+u_{0}\ppr_{2} x_{1}),$
\item[(CZ49)]
$\omega_{1}(u_{1}, v_{1}) \ppl_{3} u_{0}+ \omega_{3}(u_{1} \ast_{1} v_{1}, u_{0})=u_{1}\ppr_{1} (\omega_{3}(v_{1}, u_{0})+\omega_{2}(u_{0}, v_{1})) + \omega_{1}(u_{1},v_{1} \ast_{3} u_{0}+ u_{0} \ast_{2} v_{1}),$
\item[(CZ50)]
$\varphi(x_{0} \ppl_{2} u_{1})=x_{0}\cdot \sigma(u_{1}) + x_{0} \ppl_{0} d(u_{1}),$
\item[(CZ51)]
$\varphi(u_{0}\ppr_{2} x_{1})=u_{0}\ppr_{0} \varphi(x_{1}),$
\item[(CZ52)]
$\varphi\omega_{2}(u_{0}, u_{1}) + \sigma(u_{0} \ast_{2} u_{1})=u_{0}\ppr_{0} \sigma(u_{1}) + \omega_{0}(u_{0}, d(u_{1})),$
\item[(CZ53)]$\varphi(x_{1} \ppl_{3} u_{0})=\varphi(x_{1})\ppl_{0} u_{0},$
\item[(CZ54)]$\varphi(u_{1}\ppr_{3} x_{0})=\sigma(u_{1})\cdot x_{0}+ d(u_{1})\ppr_{0} x_{0},$
\item[(CZ55)]$\varphi\omega_{3}(u_{1}, u_{0}) + \sigma(u_{1} \ast_{3} u_{0})=\sigma(u_{1}) \ppl_{0} u_{0}+ \omega_{0}(d(u_{1}), u_{0}),$
\item[(CZ56)]$\sigma(u_{1})\cdot x_{1}+ d(u_{1})\ppr_{2} x_{1}=u_{1}\ppr_{1} x_{1},$
\item[(CZ57)]$\varphi(x_{1})\ppl_{2} u_{1}=x_{1} \ppl_{1} u_{1},$
\item[(CZ58)]$\sigma(u_{1}) \ppl_{2} v_{1} + \omega_{2}(d(u_{1}), v_{1})=\omega_{1}(u_{1}, v_{1}),$
\item[(CZ59)]$x_{1}\cdot \sigma(u_{1})+x_{1} \ppl_{3} d(u_{1})=x_{1} \ppl_{1} u_{1},$
\item[(CZ60)]$u_{1}\ppr_{3} \varphi(x_{1})=u_{1}\ppr_{1} x_{1},$
\item[(CZ61)]$u_{1}\ppr_{3}\sigma(v_{1}) + \omega_{3}(u_{1}, d(v_{1}))=\omega_{1}(u_{1}, v_{1}).$
\end{enumerate}
\end{theorem}

\begin{corollary}
Let  $(Z_1, Z_0,\varphi_Z) $ be a Zinbiel 2-algebra and $(E_1, E_0,\varphi_E) $  2-vector space.
Then any Zinbiel 2-algebra structure on  $(E_1, E_0,\varphi_E) $ such that $(E_1, E_0,\varphi_E) $  contains $(Z_1, Z_0,\varphi_Z) $ as an ideal is isomorphic to a crossed product $Z\#_{}^{\omega} V$.
\end{corollary}

\subsection{Matched pair of Zinbiel 2-algebras and factorization problem}

The {factorization problem} is the following:
Let $(Z_1, Z_0,\varphi)$ and $(V_1, V_0,d)$ be two given Zinbiel 2-algebras. Describe and classify all Zinbiel 2-algebras $(E_1, E_0,\varphi_E)$ that contains $(Z_1, Z_0,\varphi)$ and $(V_1, V_0,d)$ as a subalgebras such that ${E}  = {Z} \oplus {V}$ as 2-vector spaces.

\begin{definition}\label{mpLei}
Let $(Z_{1}, Z_{0}, \varphi)$ and $(V_{1}, V_{0}, d)$  be two  Zinbiel 2-algebras. They form a matched pair if there exist
 bilinear maps:
 \begin{eqnarray*}
\rightharpoonup_{0}: V_{0} \times  Z_{0} \rightarrow Z_{0}, \quad\leftharpoonup_{0}: Z_{0} \times  V_{0} \rightarrow Z_{0}, \quad \triangleright_{0}: Z_{0} \times  V_{0} \rightarrow V_{0},
\quad \triangleleft_{0}: V_{0} \times  Z_{0} \rightarrow V_{0},\\
\rightharpoonup_{1}: V_{1} \times  Z_{1} \rightarrow Z_{1},\quad \leftharpoonup_{1}: Z_{1} \times  V_{1} \rightarrow Z_{1}, \quad\triangleright_{1}: Z_{1} \times  V_{1} \rightarrow V_{1},
\quad \triangleleft_{1}: V_{1} \times  Z_{1} \rightarrow V_{1}, \\
\rightharpoonup_{2}: V_{0} \times  Z_{1} \rightarrow Z_{1},\quad \leftharpoonup_{2}: Z_{0} \times  V_{1} \rightarrow Z_{1},\quad \triangleright_{2}: Z_{0} \times  V_{1} \rightarrow V_{1},
\quad \triangleleft_{2}: V_{0} \times  Z_{1} \rightarrow V_{1},\\
\rightharpoonup_{3}: V_{1} \times  Z_{0} \rightarrow Z_{1}, \quad\leftharpoonup_{3}: Z_{1} \times  V_{0} \rightarrow Z_{1},\quad\triangleright_{3}: Z_{1} \times  V_{0} \rightarrow V_{1},
\quad \triangleleft_{3}: V_{1} \times  Z_{0} \rightarrow V_{1},
\end{eqnarray*}
such that the direct sum space ${Z}\oplus {V}$ form a Zinbiel 2-algebra
under the following multiplication:
$$
\left(x_{i}, u_{i}\right) \circ_{i} \left(y_{i}, v_{i}\right)
 =\big( x_{i}\cdot y_{i} + x_{i} \ppl_{i} v_{i} + u_{i}\ppr_{i} y_{i},\  x_{i} \trr_{i} v_{i} + u_{i}\trl_{i} y_{i} + u_{i} \ast_{i} v_{i} \big),
$$
and
$$
\left(x_{0}, u_{0}\right) \trr_{} \left( x_{1}, u_{1}\right)
 = \big( x_{0}\cdot x_{1} + x_{0} \ppl_{2} u_{1} + u_{0}\ppr_{2} x_{1},\  x_{0} \trr_{2} u_{1} + u_{0}\trl_{2} x_{1} + u_{0} \ast_{2} u_{1} \big),
$$
$$
\left(x_{1}, u_{1}\right) \trl_{} \left( x_{0}, u_{o}\right)
 = \big( x_{1}\cdot x_{0} + x_{1} \ppl_{3} u_{0} + u_{1}\ppr_{3} x_{0},\  x_{1} \trr_{3} u_{0} + u_{1}\trl_{3} x_{0} + u_{1} \ast_{3} u_{0} \big),
$$
for $i = 0,1$ and all $x_{i}, y_{i} \in Z_{i},$ $u_{i}, v_{i} \in V_{i},$. This Zinbiel 2-algebra is
 called the \emph{bicrossed product}  of $(Z_1, Z_0,\varphi)$ and
$(V_1, V_0, d)$.  We will denoted it by  ${Z} \, \bowtie {V}$.
\end{definition}

\begin{theorem}
Let $\Omega(Z, V)=\left(\leftharpoonup_{j}, \rightharpoonup_{j}, \triangleleft_{j}, \triangleright_{j}, \omega_{j}, *_{j},\sigma ; j=0,1,2,3\right)$ be the extending datum of $(Z_{1}, Z_{0}, \varphi)$ by $(V_{1}, V_{0}, d)$  such that $\sigma$ and $\omega_{j}$ are trivial maps for $j=0,1,2,3.$ Then $\Omega(Z, V)=\left(\leftharpoonup_{j}, \rightharpoonup_{j}, \triangleleft_{j}, \triangleright_{j}, \omega_{j}, *_{j},\sigma ; j=0,1,2,3\right)=\left(\leftharpoonup_{j}, \rightharpoonup_{j}, \triangleleft_{j}, \triangleright_{j}, *_{j}; j=0,1,2,3\right)$ is a Zinbiel 2-extending structure of $(Z_{1}, Z_{0}, \varphi)$ by $(V_{1}, V_{0}, d)$  if and only if $\left(V, d, *_{j} ; j=0,1,2,3\right)$ is a strict Zinbiel 2-algebra, the following compatibilities hold for $i=0,1$ and any $x_{i}, y_{i} \in Z_{i}, u_{i}, v_{i} \in V_{i}$ :
\begin{enumerate}
\item[(BZ1)] $(V_{i}, \trl_{i}, \trr_{i})$ is an ${Z_{i}}$-bimodule:
 \begin{eqnarray*}
            &&( x_{i}\cdot y_{i} )\trr_{i} w_{i}=x_{i}\trr_{i}( y_{i}\trr_{i} w_{i}+w_{i}\trl_{i} y_{i}),\\
           &&(x_{i}\trr_{i} v_{i})\trl_{i} z_{i}=x_{i}\trr_{i}(v_{i}\trl_{i} z_{i}+ z_{i}\trr_{i} v_{i}),\\
            &&(u_{i}\trl_{i} y_{i})\trl_{i} z_{i}=u_{i}\trl_{i} (y_{i}\cdot z_{i}+ z_{i}\cdot y_{i}),
            \end{eqnarray*}
\item[(BZ2)]
$(x_{i}\ppl_{i} v_{i})\cdot y_{i}+(x_{i} \trr_{i} v_{i})\ppr_{i} y_{i}
=x_{i}\cdot (v_{i} \ppr_{i} y_{i}+y_{i} \ppl_{i} v_{i})+x_{i}\ppl_{i}( v_{i}\trl_{i} y_{i} + y_{i} \trr_{i} v_{i}),$
\item[(BZ3)] $(u_{i} \ppr_{i} x_{i})\cdot y_{i} +( u_{i}\trl_{i} x_{i} )\ppr_{i} y_{i} = u_{i}\ppr_{i}( x_{i} \cdot  y_{i} + y_{i}\cdot  x_{i} ),$
\item[(BZ4)]
$(u_{i} \ast_{i} v_{i} )\ppr_{i} x_{i}
= u_{i}\ppr_{i}( v_{i}\ppr_{i} x_{i} + x_{i} \ppl_{i} v_{i}),$
\item[(BZ5)] $(u_{i} \ast_{i} v_{i})\trl_{i} x_{i} = u_{i}\trl_{i}( v_{i}\ppr_{i} x_{i}+x_{i}\ppl_{i} v_{i})+u_{i} \ast_{i}( v_{i}\trl_{i} x_{i} + x_{i} \trr_{i} v_{i} ),$
\item[(BZ6)]$( x_{i}\cdot y_{i} )\ppl_{i} w_{i}  = x_{i}\cdot(y_{i} \ppl_{i} w_{i}+w_{i}\ppr_{i} y_{i})+x_{i}\ppl_{i}(y_{i} \trr_{i} w_{i}+w_{i}\trl_{i} y_{i}),$
\item[(BZ7)]
$( x_{i} \ppl_{i} v_{i})\ppl_{i} w_{i}
=x_{i}\ppl_{i}(v_{i} \ast_{i} w_{i} +w_{i} \ast_{i} v_{i} ),$
\item[(BZ8)] $(x_{i} \ppl_{i} v_{i}) \trr_{i} w_{i}+(x_{i}\trr_{i} v_{i})\ast_{i} w_{i} = x_{i}\trr_{i}( v_{i} \ast_{i} w_{i}+w_{i} \ast_{i} v_{i} ),$
\item[(BZ9)]
$( u_{i}\ppr_{i} x_{i})\ppl_{i} w_{i}
= u_{i}\ppr_{i}( x_{i} \ppl_{i} w_{i}+w_{i}\ppr_{i} x_{i} ),$
\item[(BZ10)]
$(u_{i}\ppr_{i} x_{i}) \trr_{i} w_{i}+ (u_{i}\trl_{i} x_{i})\ast_{i} w_{i}
= u_{i}\trl_{i}(x_{i} \ppl_{i} w_{i} + w_{i}\ppr_{i} x_{i} )+u_{i} \ast_{i}(x_{i} \trr_{i} w_{i} + w_{i}\trl_{i} x_{i}),$
\item[(BZ11)]
$(x_{0}\cdot x_{1}) \ppl_{1} u_{1}=x_{0}\cdot (x_{1} \ppl_{1}u_{1}+u_{1}\ppr_{1} x_{1})+ x_{0} \ppl_{2} (x_{1} \trr_{1} u_{1}+u_{1}\trl_{1} x_{1}),$
\item[(BZ12)]
$(x_{0} \ppl_{2} u_{1})\cdot x_{1} + ( x_{0} \trr_{2}u_{1})\ppr_{1}x_{1}=x_{0}\cdot (u_{1}\ppr_{1}x_{1}+x_{1} \ppl_{1}u_{1})+ x_{0}\ppl_{2}(u_{1}\trl_{1} x_{1}+x_{1} \trr_{1} u_{1}),$
\item[(BZ13)]
$(x_{0} \ppl_{2} u_{1}) \ppl_{1} v_{1} =x_{0} \ppl_{2} (u_{1} \ast_{1} v_{1}+v_{1} \ast_{1} u_{1}),$
\item[(BZ14)] $(u_{0}\ppr_{2} x_{1})\cdot y_{1} + (u_{0}\trl_{2} x_{1})\ppr_{1} y_{1}=u_{0}\ppr_{2} ( x_{1}\cdot y_{1}+y_{1}\cdot x_{1}),$
\item[(BZ15)]
$(u_{0}\ppr_{2} x_{1}) \ppl_{1}u_{1}=u_{0}\ppr_{2} (x_{1} \ppl_{1}u_{1}+u_{1}\ppr_{1} x_{1}),$
\item[(BZ16)]
$(u_{0} \ast_{2} u_{1})\ppr_{1} x_{1}=u_{0}\ppr_{2} (u_{1}\ppr_{1} x_{1}+x_{1} \ppl_{1} u_{1}),$
\item[(BZ17)] $(x_{0}\cdot x_{1}) \trr_{1} u_{1}=x_{0} \trr_{2} (x_{1} \trr_{1} u_{1}+u_{1}\trl_{1} x_{1}),$
\item[(BZ18)] $( x_{0} \trr_{2} u_{1})\trl_{1} y_{1} =x_{0} \trr_{2} (u_{1}\trl_{1} y_{1}+y_{1} \trr_{1} u_{1}),$
\item[(BZ19)] $(x_{0} \ppl_{2} u_{1}) \trr_{1} v_{1}+( x_{0} \trr_{2} u_{1}) \ast_{1} v_{1}=x_{0} \trr_{2} (u_{1} \ast_{1} v_{1}+v_{1} \ast_{1} u_{1}),$
\item[(BZ20)] $(u_{0}\trl_{2} x_{1})\trl_{1} y_{1}=u_{0}\trl_{2} (x_{1}\cdot y_{1}+y_{1}\cdot x_{1}),$
\item[(BZ21)]
$(u_{0}\ppr_{2} x_{1}) \trr_{1} u_{1} + (u_{0}\trl_{2} x_{1}) \ast_{1} u_{1}=u_{0}\trl_{2} (x_{1} \ppl_{1} u_{1}+u_{1}\ppr_{1} x_{1}) + u_{0} \ast_{2} (x_{1} \trr_{1} u_{1}+u_{1}\trl_{1} x_{1}),$
\item[(BZ22)]
$(u_{0} \ast_{2} u_{1})\trl_{1} x_{1}=u_{0}\trl_{2} (u_{1}\ppr_{1} x_{1}+x_{1} \ppl_{1} u_{1}) + u_{0} \ast_{2} (u_{1}\trl_{1} x_{1}+x_{1} \trr_{1} u_{1}),$
\item[(BZ23)]
$(x_{0} \ppl_{0} u_{0})\cdot x_{1}+ (x_{0} \trr_{0} u_{0})\ppr_{2} x_{1}=x_{0}\cdot (u_{0}\ppr_{2} x_{1}+x_{1} \ppl_{3} u_{0})+ x_{0} \ppl_{2} (u_{0}\trl_{2} x_{1}+x_{1} \trr_{3} u_{0}),$
\item[(BZ24)]
$(u_{0}\ppr_{0} x_{0})\cdot x_{1}+ (u_{0}\trl_{0} x_{0})\ppr_{2} x_{1}=u_{0}\ppr_{2} (x_{0}\cdot x_{1}+x_{1}\cdot x_{0}),$
\item[(BZ25)]
$(u_{0} \ast_{0} v_{0})\ppr_{2} x_{1}= u_{0}\ppr_{2} (v_{0}\ppr_{2} x_{1}+x_{1} \ppl_{3} v_{0}),$
\item[(BZ26)]
$(x_{0}\cdot y_{0}) \ppl_{2} u_{1}=x_{0}\cdot (y_{0} \ppl_{2} u_{1}+u_{1}\ppr_{3} y_{0})+x_{0} \ppl_{2} (y_{0} \trr_{2} u_{1}+u_{1}\trl_{3} y_{0}),$
\item[(BZ27)]
$(x_{0} \ppl_{0} u_{0}) \ppl_{2} u_{1}=x_{0} \ppl_{2} (u_{0} \ast_{2} u_{1}+u_{1} \ast_{3} u_{0}),$
\item[(BZ28)]
$(u_{0}\ppr_{0} x_{0}) \ppl_{2} u_{1}= u_{0}\ppr_{2} (x_{0} \ppl_{2} u_{1}+u_{1}\ppr_{3} x_{0}),$
\item[(BZ29)]
$(x_{0}\cdot y_{0}) \trr_{2} u_{1}=x_{0} \trr_{2} (y_{0} \trr_{2} u_{1}+u_{1}\trl_{3} y_{0}),$
\item[(BZ30)]
$(x_{0} \ppl_{0} u_{0}) \trr_{2} u_{1} + (x_{0} \trr_{0} u_{0}) \ast_{2} u_{1}=x_{0} \trr_{2} (u_{0} \ast_{2} u_{1}+u_{1} \ast_{3} u_{0}),$
\item[(BZ31)]
$(u_{0}\ppr_{0} x_{0}) \trr_{2} u_{1}+ (u_{0}\trl_{0} x_{0}) \ast_{2} u_{1}=u_{0}\trl_{2} (x_{0} \ppl_{2} u_{1}+u_{1}\ppr_{3} x_{0})+ u_{0} \ast_{2} (x_{0} \trr_{2} u_{1}+u_{1}\trl_{3} x_{0}),$
\item[(BZ32)]
$(x_{0} \trr_{0} u_{0})\trl_{2} x_{1}=x_{0} \trr_{2} (u_{0}\trl_{2} x_{1}+ x_{1} \trr_{3} u_{0}),$
\item[(BZ33)]
$(u_{0}\trl_{0} x_{0})\trl_{2} x_{1}=u_{0}\trl_{2} (x_{0}\cdot x_{1}+x_{1}\cdot x_{0}),$
\item[(BZ34)]
$(u_{0} \ast_{0} v_{0})\trl_{2} x_{1}=u_{0}\trl_{2} (v_{0}\ppr_{2} x_{1}+x_{1} \ppl_{3} v_{0})+ u_{0} \ast_{2} (v_{0}\trl_{2} x_{1}+x_{1} \trr_{3} v_{0}),$
\item[(BZ35)]
$(x_{1} \ppl_{3} u_{0})\cdot x_{0}+ (x_{1} \trr_{3} u_{0})\ppr_{3} x_{0}=x_{1}\cdot (u_{0}\ppr_{0} x_{0}+x_{0} \ppl_{0} u_{0}) + x_{1} \ppl_{3} (u_{0}\trl_{0} x_{0}+x_{0} \trr_{0} u_{0}),$
\item[(BZ36)]
$(u_{1}\ppr_{3} x_{0})\cdot y_{0} + (u_{1}\trl_{3} x_{0})\ppr_{3} y_{0}=u_{1}\ppr_{3} (x_{0}\cdot y_{0}+y_{0}\cdot x_{0}),$
\item[(BZ37)]
$(u_{1} \ast_{3} u_{0})\ppr_{3} x_{0}=u_{1}\ppr_{3} (u_{0}\ppr_{0} x_{0}+x_{0} \ppl_{0} u_{0}),$
\item[(BZ38)]
$(x_{1}\cdot x_{0}) \ppl_{3} u_{0}=x_{1}\cdot (x_{0} \ppl_{0} u_{0}+u_{0}\ppr_{0} x_{0}) + x_{1} \ppl_{3} (x_{0} \trr_{0} u_{0}+u_{0}\trl_{0} x_{0}),$
\item[(BZ39)]
$(x_{1} \ppl_{3} u_{0}) \ppl_{3} v_{0}=x_{1} \ppl_{3} (u_{0} \ast_{0} v_{0}+v_{0} \ast_{0} u_{0}),$
\item[(BZ40)]
$(u_{1}\ppr_{3} x_{0}) \ppl_{3} u_{0} =u_{1}\ppr_{3} (x_{0} \ppl_{0} u_{0}+u_{0}\ppr_{0} x_{0}),$
\item[(BZ41)]
$(x_{1}\cdot x_{0}) \trr_{3} u_{0}=x_{1} \trr_{3} (x_{0} \trr_{0} u_{0}+u_{0}\trl_{0} x_{0}),$
\item[(BZ42)]
$(x_{1} \ppl_{3} u_{0}) \trr_{3} v_{0}+ (x_{1} \trr_{3} u_{0}) \ast_{3} v_{0}=x_{1} \trr_{3} (u_{0} \ast_{0} v_{0}+v_{0} \ast_{0} u_{0}),$
\item[(BZ43)]
$(u_{1}\ppr_{3} x_{0}) \trr_{3} u_{0}+ (u_{1}\trl_{3} x_{0}) \ast_{3} u_{0}=u_{1}\trl_{3} (x_{0} \ppl_{0} u_{0}+u_{0}\ppr_{0} x_{0})+ u_{1} \ast_{3} (x_{0} \trr_{0} u_{0}+u_{0}\trl_{0} x_{0}),$
\item[(BZ44)]
$(x_{1} \trr_{3} u_{0})\trl_{3} x_{0}=x_{1} \trr_{3} (u_{0}\trl_{0} x_{0}+x_{0} \trr_{0} u_{0}),$
\item[(BZ45)]
$(u_{1}\trl_{3} x_{0})\trl_{3} y_{0}=u_{1}\trl_{3} (x_{0}\cdot y_{0}+y_{0}\cdot x_{0}),$
\item[(BZ46)]
$(u_{1} \ast_{3} u_{0})\trl_{3} x_{0}=u_{1}\trl_{3} (u_{0}\ppr_{0} x_{0}+x_{0} \ppl_{0} u_{0})+ u_{1} \ast_{3} (u_{0}\trl_{0} x_{0}+x_{0} \trr_{0} u_{0}),$
\item[(BZ47)]
$(x_{0} \ppl_{2} u_{1})\cdot y_{0}+ (x_{0} \trr_{2} u_{1})\ppr_{3} y_{0}=x_{0}\cdot (u_{1}\ppr_{3} y_{0}+y_{0} \ppl_{2} u_{1})+ x_{0} \ppl_{2} (u_{1}\trl_{3} y_{0}+y_{0} \trr_{2} u_{1}),$
\item[(BZ48)]
$(u_{0}\ppr_{2} x_{1})\cdot y_{0}+ (u_{0}\trl_{2} x_{1})\ppr_{3} y_{0}= u_{0}\ppr_{2} (x_{1}\cdot y_{0}+y_{0}\cdot x_{1}),$
\item[(BZ49)]
$(u_{0} \ast_{2} u_{1})\ppr_{3} x_{0}= u_{0}\ppr_{2} (u_{1}\ppr_{3} x_{0}+x_{0} \ppl_{2} u_{1}),$
\item[(BZ50)]
$(x_{0}\cdot x_{1}) \ppl_{3} u_{0}=x_{0}\cdot (x_{1} \ppl_{3} u_{0}+u_{0}\ppr_{2} x_{1})+ x_{0} \ppl_{2} (x_{1} \trr_{3} u_{0}+u_{0}\trl_{2} x_{1}),$
\item[(BZ51)]
$(x_{0} \ppl_{2} u_{1}) \ppl_{3} v_{0}=x_{0} \ppl_{2} (u_{1} \ast_{3} v_{0}+v_{0} \ast_{2} u_{1}),$
\item[(BZ52)]
$(u_{0}\ppr_{2} x_{1}) \ppl_{3} v_{0}= u_{0}\ppr_{2} (x_{1} \ppl_{3} v_{0}+v_{0}\ppr_{2} x_{1}) ,$
\item[(BZ53)]
$(x_{0}\cdot x_{1}) \trr_{3} u_{0}=x_{0} \trr_{2} (x_{1} \trr_{3} u_{0}+u_{0}\trl_{2} x_{1}),$
\item[(BZ54)]
$(x_{0} \ppl_{2} u_{1}) \trr_{3} v_{0}+ (x_{0} \trr_{2} u_{1}) \ast_{3} v_{0}=x_{0} \trr_{2} (u_{1} \ast_{3} v_{0}+v_{0} \ast_{2} u_{1}),$
\item[(BZ55)]
$(u_{0}\ppr_{2} x_{1}) \trr_{3} v_{0}+ (u_{0}\trl_{2} x_{1}) \ast_{3} v_{0}=u_{0}\trl_{2} (x_{1} \ppl_{3} v_{0}+v_{0}\ppr_{2} x_{1})+ u_{0} \ast_{2} (x_{1} \trr_{3} v_{0}+v_{0}\trl_{2} x_{1}),$
\item[(BZ56)]
$(x_{0} \trr_{2} u_{1})\trl_{3} y_{0}=x_{0} \trr_{2} (u_{1}\trl_{3} y_{0}+y_{0} \trr_{2} u_{1}),$
\item[(BZ57)]
$(u_{0}\trl_{2} x_{1})\trl_{3} y_{0}=u_{0}\trl_{2} (x_{1}\cdot y_{0}+y_{0}\cdot x_{1}),$
\item[(BZ58)]
$(u_{0} \ast_{2} u_{1})\trl_{3} y_{0}=u_{0}\trl_{2} (u_{1}\ppr_{3} y_{0}+y_{0} \ppl_{2} u_{1})+ u_{0} \ast_{2} (u_{1}\trl_{3} y_{0}+y_{0} \trr_{2} u_{1}),$
\item[(BZ59)]
$(x_{1} \ppl_{3} u_{0})\cdot y_{1}+ (x_{1} \trr_{3} u_{0})\ppr_{1} y_{1}=x_{1}\cdot (u_{0}\ppr_{2} y_{1}+y_{1} \ppl_{3} u_{0})+ x_{1} \ppl_{1} (u_{0}\trl_{2} y_{1}+y_{1} \trr_{3} u_{0}),$
\item[(BZ60)]
$(u_{1}\ppr_{3} x_{0})\cdot x_{1} + (u_{1}\trl_{3} x_{0})\ppr_{1} x_{1}=u_{1}\ppr_{1} (x_{0}\cdot x_{1}+x_{1}\cdot x_{0}),$
\item[(BZ61)]
$(u_{1} \ast_{3} u_{0})\ppr_{1} x_{1}=u_{1}\ppr_{1} (u_{0}\ppr_{2} x_{1}+x_{1} \ppl_{3} u_{0}),$
\item[(BZ62)]
$(x_{1}\cdot x_{0}) \ppl_{1} u_{1}=x_{1}\cdot (x_{0} \ppl_{2} u_{1}+u_{1}\ppr_{3} x_{0})+x_{1} \ppl_{1} (x_{0} \trr_{2} u_{1}+u_{1}\trl_{3} x_{0}),$
\item[(BZ63)]
$(x_{1} \ppl_{3} u_{0}) \ppl_{1} u_{1}=x_{1} \ppl_{1} (u_{0} \ast_{2} u_{1}+u_{1} \ast_{3} u_{0}) ,$
\item[(BZ64)]
$(u_{1}\ppr_{3} x_{0}) \ppl_{1} v_{1} =u_{1}\ppr_{1} (x_{0} \ppl_{2} v_{1}+v_{1}\ppr_{3} x_{0}),$
\item[(BZ65)]
$(x_{1}\cdot x_{0}) \trr_{1} u_{1}=x_{1} \trr_{1} (x_{0} \trr_{2} u_{1}+u_{1}\trl_{3} x_{0}),$
\item[(BZ66)]
$(x_{1} \ppl_{3} u_{0}) \trr_{1} u_{1}+ (x_{1} \trr_{3} u_{0}) \ast_{1} u_{1}=x_{1} \trr_{1} (u_{0} \ast_{2} u_{1}+u_{1} \ast_{3} u_{0}),$
\item[(BZ67)]
$(u_{1}\ppr_{3} x_{0}) \trr_{1} v_{1} + (u_{1}\trl_{3} x_{0}) \ast_{1} v_{1}=u_{1}\trl_{1} (x_{0} \ppl_{2} v_{1}+v_{1}\ppr_{3} x_{0})+ u_{1} \ast_{1} (x_{0} \trr_{2} v_{1}+v_{1}\trl_{3} x_{0}),$
\item[(BZ68)]
$(x_{1} \trr_{3} u_{0})\trl_{1} y_{1}=x_{1} \trr_{1} (u_{0}\trl_{2} y_{1}+y_{1} \trr_{3} u_{0}),$
\item[(BZ69)]
$(u_{1}\trl_{3} x_{0})\trl_{1} x_{1}=u_{1}\trl_{1} (x_{0}\cdot x_{1}+x_{1}\cdot x_{0}),$
\item[(BZ70)]
$(u_{1} \ast_{3} u_{0})\trl_{1} x_{1}=u_{1}\trl_{1} (u_{0}\ppr_{2} x_{1}+x_{1} \ppl_{3} u_{0})+ u_{1} \ast_{1} (u_{0}\trl_{2} x_{1}+x_{1} \trr_{3} u_{0}) ,$
\item[(BZ71)]
$(x_{1} \ppl_{1} u_{1})\cdot x_{0}+ (x_{1} \trr_{1} u_{1})\ppr_{3} x_{0}=x_{1}\cdot (u_{1}\ppr_{3} x_{0}+x_{0} \ppl_{2} u_{1})+ x_{1} \ppl_{1} (u_{1}\trl_{3} x_{0}+x_{0} \trr_{2} u_{1}),$
\item[(BZ72)]
$(u_{1}\ppr_{1} x_{1})\cdot x_{0} + (u_{1}\trl_{1} x_{1})\ppr_{3} x_{0}= u_{1}\ppr_{1} (x_{1}\cdot x_{0}+x_{0}\cdot x_{1}),$
\item[(BZ73)]
$(u_{1} \ast_{1} v_{1})\ppr_{3} x_{0}=u_{1}\ppr_{1} (v_{1}\ppr_{3} x_{0}+x_{0} \ppl_{2} v_{1}),$
\item[(BZ74)]
$(x_{1}\cdot y_{1}) \ppl_{3} u_{0}=x_{1}\cdot (y_{1} \ppl_{3} u_{0}+u_{0}\ppr_{2} y_{1})+ x_{1} \ppl_{1} (y_{1} \trr_{3} u_{0}+u_{0}\trl_{2} y_{1}),$
\item[(BZ75)]
$(x_{1} \ppl_{1} u_{1}) \ppl_{3} u_{0}=x_{1} \ppl_{1} (u_{1} \ast_{3} u_{0}+u_{0} \ast_{2} u_{1}),$
\item[(BZ76)]
$(u_{1}\ppr_{1} x_{1}) \ppl_{3} u_{0} = u_{1}\ppr_{1} (x_{1} \ppl_{3} u_{0}+u_{0}\ppr_{2} x_{1}) ,$
\item[(BZ77)]
$(x_{1}\cdot y_{1}) \trr_{3} u_{0}=x_{1} \trr_{1} (y_{1} \trr_{3} u_{0}+u_{0}\trl_{2} y_{1}),$
\item[(BZ78)]
$(x_{1} \ppl_{1} u_{1}) \trr_{3} u_{0}+ (x_{1} \trr_{1} u_{1}) \ast_{3} u_{0}=x_{1} \trr_{1} (u_{1} \ast_{3} u_{0}+u_{0} \ast_{2} u_{1}),$
\item[(BZ79)]
$(u_{1}\ppr_{1} x_{1}) \trr_{3} u_{0} + (u_{1}\trl_{1} x_{1}) \ast_{3} u_{0}=u_{1}\trl_{1} (x_{1} \ppl_{3} u_{0}+u_{0}\ppr_{2} x_{1})+ u_{1} \ast_{1} (x_{1} \trr_{3} u_{0}+u_{0}\trl_{2} x_{1}),$
\item[(BZ80)]
$(x_{1} \trr_{1} u_{1})\trl_{3} x_{0}=x_{1} \trr_{1} (u_{1}\trl_{3} x_{0}+x_{0} \trr_{2} u_{1}),$
\item[(BZ81)]
$(u_{1}\trl_{1} x_{1})\trl_{3} x_{0}=u_{1}\trl_{1} (x_{1}\cdot x_{0}+x_{0}\cdot x_{1}),$
\item[(BZ82)]
$(u_{1} \ast_{1} v_{1})\trl_{3} x_{0}=u_{1}\trl_{1} (v_{1}\ppr_{3} x_{0}+x_{0} \ppl_{2} v_{1})+ u_{1} \ast_{1} (v_{1}\trl_{3} x_{0}+x_{0} \trr_{2} v_{1}),$
\item[(BZ83)]
$\varphi(x_{0} \ppl_{2} u_{1})+\sigma(x_{0} \trr_{2} u_{1})=x_{0}\cdot \sigma(u_{1}) + x_{0} \ppl_{0} d(u_{1}),$
\item[(BZ84)]
$\varphi(u_{0}\ppr_{2} x_{1})+ \sigma(u_{0}\trl_{2} x_{1})=u_{0}\ppr_{0} \varphi(x_{1}),$
\item[(BZ85)]
$\sigma(u_{0} \ast_{2} u_{1})=u_{0}\ppr_{0} \sigma(u_{1}) ,$
\item[(BZ86)]
$d(x_{0} \trr_{2} u_{1})=x_{0} \trr_{0} d(u_{1}),$
\item[(BZ87)]
$d(u_{0}\trl_{2} x_{1})= u_{0}\trl_{0} \varphi(x_{1}),$
\item[(BZ88)]
$d(u_{0} \ast_{2} u_{1})= u_{0}\trl_{0}\sigma(u_{1}) + u_{0} \ast_{0} d(u_{1}),$
\item[(BZ89)]$\varphi(x_{1} \ppl_{3} u_{0})+\sigma(x_{1} \trr_{3} u_{0})=\varphi(x_{1})\ppl_{0} u_{0},$
\item[(BZ90)]$\varphi(u_{1}\ppr_{3} x_{0}) + \sigma(u_{1}\trl_{3} x_{0})=\sigma(u_{1})\cdot x_{0}+ d(u_{1})\ppr_{0} x_{0},$
\item[(BZ91)]$\sigma(u_{1} \ast_{3} u_{0})=\sigma(u_{1}) \ppl_{0} u_{0},$
\item[(BZ92)]$d(x_{1} \trr_{3} u_{0})=\varphi(x_{1})\trr_{0} u_{0},$
\item[(BZ93)]$d(u_{1}\trl_{3} x_{0}) =d(u_{1})\trl_{0} x_{0},$
\item[(BZ94)]$d(u_{1} \ast_{3} u_{0}))=\sigma(u_{1}) \trr_{0} u_{0} + d(u_{1}) \ast_{0} u_{0},$
\item[(BZ95)]$\sigma(u_{1})\cdot x_{1}+ d(u_{1})\ppr_{2} x_{1}=u_{1}\ppr_{1} x_{1},$
\item[(BZ96)]$\varphi(x_{1})\ppl_{2} u_{1}=x_{1} \ppl_{1} u_{1},$
\item[(BZ97)]$\sigma(u_{1}) \ppl_{2} v_{1}=0,$
\item[(BZ98)]$\varphi(x_{1})\trr_{2} u_{1}=x_{1} \trr_{1} u_{1},$
\item[(BZ99)]$\sigma(u_{1}) \trr_{2} v_{1} + d(u_{1}) \ast_{2} v_{1}=u_{1} \ast_{1} v_{1},$
\item[(BZ100)]$d(u_{1})\trl_{2} x_{1} =u_{1}\trl_{1} x_{1},$
\item[(BZ101)]$x_{1}\cdot \sigma(u_{1})+x_{1} \ppl_{3} d(u_{1})=x_{1} \ppl_{1} u_{1},$
\item[(BZ102)]$u_{1}\ppr_{3} \varphi(x_{1})=u_{1}\ppr_{1} x_{1},$
\item[(BZ103)]$u_{1}\ppr_{3}\sigma(v_{1}) = 0,$
\item[(BZ104)]$x_{1} \trr_{3} d(u_{1})=x_{1} \trr_{1} u_{1},$
\item[(BZ105)]$u_{1}\trl_{3} \varphi(x_{1})=u_{1}\trl_{1} x_{1},$
\item[(BZ106)]$u_{1}\trl_{3}\sigma(v_{1}) + u_{1} \ast_{3} d(v_{1})=u_{1} \ast_{1} v_{1}.$
\end{enumerate}
\end{theorem}

By the above theorem we obtain the following result.

\begin{corollary}\colabel{bicrfactor}
A Zinbiel 2-algebra $(E_1, E_0,\varphi_E)$ factorizes through $(Z_1, Z_0,\varphi)$ and
$(V_1, V_0, d)$ if and only if there exists a matched pair of Zinbiel 2-algebras $(Z_1, Z_0,\varphi)$ and $(V_1, V_0, d)$ such
that $(E_1, E_0,\varphi_E)  \cong {Z} \bowtie {V}$.
\end{corollary}

\vskip7pt
\footnotesize{
\noindent Ling Zhang\\
College of Mathematics and Information Science,\\
Henan Normal University, Xinxiang 453007, P. R. China;\\
 E-mail address: \texttt{{zhanglingny@163.com}}

\vskip7pt
\footnotesize{
\noindent Tao Zhang\\
College of Mathematics and Information Science,\\
Henan Normal University, Xinxiang 453007, P. R. China;\\
 E-mail address: \texttt{{zhangtao@htu.edu.cn}}

\end{document}